\documentclass[12pt]{article}
\usepackage{amsmath}
\usepackage{amssymb}
\usepackage{amsthm}
\usepackage{amscd, newlfont, enumerate}
\usepackage[latin2]{inputenc}
\usepackage[dvips]{graphicx}
\usepackage[cmtip,arrow]{xy}
\usepackage[active]{srcltx}
\usepackage{color}
\usepackage{mathtools}
\usepackage{breqn}

 \newtheorem{thrm}{Theorem}[section]
 \newtheorem{corollary}[thrm]{Corollary}
 \newtheorem{lemma}[thrm]{Lemma}
 \newtheorem{proposition}[thrm]{Proposition}
 \newtheorem{definition}[thrm]{Definition}
 \newtheorem{remark}[thrm]{Remark}

\def\d{\delta}
\def\rt{\rtimes}

\def\ot{\otimes}

\def\s{\sigma}

\def\ti{\times}

\newcommand{\Jord}{\mathop{\text{Jord}}}
\newcommand{\Irr}{\mathop{\text{Irr}}}

\title{Discrete series of odd general spin groups}
\author{Yeansu Kim, Ivan Mati\'{c}}
\date{\today}

\begin{document}

\maketitle

\begin{abstract}
We obtain a M{\oe}glin-Tadi\'{c} type classification of the non-cuspidal discrete series of odd general spin groups over non-archimedean local fields of characteristic zero. Our approach presents a simplified, uniform, and slightly different construction of a bijective correspondence between the set of isomorphism classes of non-cuspidal discrete series representations and the set of so called admissible triples, a notion obtained by readily extending the analogous notion due to M{\oe}glin-Tadi\'{c} from the case of classical groups to that of odd general spin groups. We use almost exclusively algebraic methods, one of which replaces subtle theory on intertwining operators for odd general spin groups by calculation of Jacquet modules. In this way we provide a slightly different proof of the classification of discrete series representations for classical groups, which contains some more concrete information on the admissible triples. We expect that our classification has an advantage of being rather directly applicable to several other reductive $p$-adic groups, including even general spin groups and similitude classical groups.
\end{abstract}

{\renewcommand{\thefootnote}{} 
}

\section{Introduction}

Irreducible square-integrable representations present a prominent part of the unitary dual of reductive groups over non-archimedean local fields, with numerous applications in harmonic analysis and in the theory of automorphic forms. Such representations, also called the discrete series, have been classified by M{\oe}glin and Tadi\'{c} in the case of classical groups defined over non-archimedean local fields of characteristic zero in their seminal work \cite{Moe1, MT1}. Their work completely describes that prominent class of irreducible representations, modulo cuspidal ones, in terms of the so-called admissible triples which consist of a Jordan block, the partial cuspidal support and the $\epsilon$-function. We note that the Jordan block encodes the Langlands parameter in Langlands parametrization, while the latter two components together constitute a datum that, modulo the extended Langlands parameterization for cuspidal representations, is equivalent to a character of the component group of the centralizer of the Langlands parameter. Their work relies on the Basic Assumption, which now follows from \cite{Arthur}, \cite[Th\'{e}or\`{e}me~3.1.1]{Moe2}, and \cite[Theorem~7.8]{GanLom}, and is also known to hold for the general spin groups in the characteristic zero case \cite[Th\'{e}oreme~3.1.1]{Moe2}. We discuss the Basic Assumption in more detail in Section $3$.


Although the M{\oe}glin-Tadi\'{c} classification is intrinsically combinatorial and can be elegantly used for various computations, the proofs which appear in \cite{Moe1} and \cite{MT1} are rather long and sometimes happen to be highly involved.

The main purpose of this paper is to extend the M{\oe}glin-Tadi\'{c} classification to the case of the split odd general spin groups over non-archimedean local fields in a  simplified and more uniform way. We note that the rank $n$ split odd general spin group is a split reductive linear algebraic group of type $B_n$, whose derived subgroup is a double covering of a split special orthogonal group. Our classification is also given in terms of admissible triples, i.e., we construct a natural bijection between the set of all isomorphism classes of discrete series of the odd general spin group and the set of all admissible triples, but we have substantially shortened several proofs given in the original work of M{\oe}glin and Tadi\'{c}. 

There are certain differences which appear in our approach, although we intend to follow the classification for the classical groups. Firstly, we use mostly algebraic methods, e.g. the Jacquet module method, except for the fact that we employ $L$-function techniques in the $GL$ case to calculate the Jordan blocks of discrete series appearing in embeddings of a particular type (Proposition \ref{propjord}). 

Secondly, 
when we define the $\epsilon$-function on certain single elements of the Jordan blocks, we use a definition similar to the one suggested in \cite{Tad7}, which seems to be more appropriate in the GSpin situation than the original one (Definition \ref{deftwo}). In this way a definition of the $\epsilon$-function is entirely provided in terms of the Jacquet modules, which helps us to provide a restriction type results (Theorem \ref{propprva}) without the usage of the intertwining operators method, as was the case in \cite[Section~6]{Moe1}. We note that in this case we need to fix a labeling of irreducible tempered subrepresentations appearing in the induced representation of a particular type. On the other hand, in \cite{Moe1} the definition of the $\epsilon$-function in the considered case ($\tau_1$ and $\tau_{-1}$ in the above Theorem \ref{firstthm}), and such a labeling, was provided using much more subtle theory involving standard intertwining operators, their analytic continuation, and coherent normalization using $L$-functions and $\epsilon$-factors. Up to our knowledge, such results are not yet known for the GSpin groups in their full extent.

Thirdly, we use a slightly different approach when we construct and define the admissible triples. We define the $\epsilon$-function only on 'consecutive' pairs appearing in the Jordan blocks, i.e., only on ordered pairs of the form $((a\_, \rho), (a, \rho))$ (Definition \ref{defone}). This approach provides as many properties as the original one, and also helps us to avoid many technical difficulties arising when studying the restrictions of the $\epsilon$-function, which happens to be crucial for the proof of the injectivity part (Theorem \ref{tminj}). Moreover, we obtain some more concrete information on the properties of the $\epsilon$-function which can be read off from Jacquet modules (characterizations $(i)$, $(ii)$, and $(iii)$ from Theorem \ref{firstthm}$(3)$).

Fourthly, we prove the surjectivity part of the classification (Theorem \ref{tmsurj}) in a completely different manner than in \cite{MT1}, using an inductive procedure based on an approach introduced in \cite{Mu3} and further developed in \cite{Mu11} and in \cite{Matic11}, which enables us to significantly shorten the proofs of results analogous to ones appearing in \cite[Sections~9,~10,~11]{MT1}. 

Note that our main ideas of the proof of the classification results in the current paper are completely independent of the M{\oe}glin-Tadi\'{c} one, and all our methods and proofs can be used in the classical group case. Note that we only use some results appearing in \cite[Section~4]{MT1}, to avoid repeating some calculations of the Jacquet modules, which happen to be completely analogous in both classical and GSpin group cases. In sum, our classification relies only on the theory of intertwining operators in the general linear groups case, the square-integrability criterion, the uniqueness of the partial cuspidal support, and the structural formula in the GSpin case. Written in this way, the classification has an advantage of being rather directly applicable to many other reductive $p$-adic groups which are of particular interest, such as the metaplectic group or the even general spin group and, partially, the generalized unitary group and the generalized symplectic group.

Let us now briefly explain the main idea of the proof of our main results. The classification of discrete series is done in two stages, the first being the classification of the so-called strongly positive discrete series (Definition \ref{defsp}). We note that the classification of the strongly positive discrete series for the general spin groups is given by the first author in \cite[Theorem~A]{Kim1}, closely following the methods introduced in \cite{Matic3}. 
Using a description of the Jacquet modules of the strongly positive representations, we can show that they correspond to admissible triples of 'alternated type' (see Definition \ref{deftripl} and Proposition \ref{propsp}).

In the second stage, to provide the inductive construction of discrete series, we obtain several properties of the attached $\epsilon$-functions (see Theorem \ref{prepairemain} and Theorem \ref{propprva}). Initially, such properties rely on certain embeddings of discrete series, while the other crucial role in the classification is played by the behavior of restrictions of the $\epsilon$-functions, which we completely describe using the Jacquet modules method and methods of intertwining operators.

Note that one of the main steps in such a description involves identifying certain prominent irreducible constituents of the Jacquet modules of certain tempered representations (see Theorem \ref{firstthm}$(3)$ $(i)$, $(ii)$, and $(iii)$). This approach was first introduced in \cite[Section~6]{Tad6}, and further enhanced in \cite{Matic7, Matic11}. In other words, this approach enables one to extract certain kinds of information about representation from 
'general linear' contributions to irreducible subquotients of its Jacquet modules.

For the convenience of the reader, we cite our main results here. For the notation and definitions we refer the reader to the following section and Definitions \ref{defjord}, \ref{defone}, \ref{deftwo}, \ref{deftripl}. We note that for an irreducible self-dual cuspidal representation $\rho$ of the general linear group and an irreducible cuspidal representation $\sigma_{cusp}$ of the GSpin group such that $\rho \rtimes \sigma_{cusp}$ reduces we fix a choice of labeling irreducible tempered representations $\tau_1$ and $\tau_{-1}$ such that $\rho \rtimes \sigma_{cusp} = \tau_1 + \tau_{-1}$.

\begin{thrm} \label{firstthm}
There exists a bijective correspondence between the set of all discrete series $\sigma$ of the odd GSpin group and the set of all admissible triples $(\Jord, \sigma', \epsilon)$, denoted by $\sigma = \sigma_{(\Jord, \sigma', \epsilon)}$
such that the following holds:
\begin{enumerate}[(1)]
    \item $\Jord(\sigma) = \Jord$ and $\sigma'$ is isomorphic to the partial cuspidal support of $\sigma$.
    \item If $(\Jord, \sigma', \epsilon)$ is an admissible triple of alternated type, then $\sigma$ is a strongly positive discrete series.
    \item Let $(a, \rho) \in \Jord$ such that $a\_$ is defined and $\epsilon((a\_, \rho), (a, \rho)) = 1$. Let $\epsilon : D \subseteq \Jord \cup \Jord \times \Jord \rightarrow \{ 1, -1 \}$. We put $\Jord' = \Jord \setminus \{ (a\_, \rho), (a, \rho) \}$. If $(a, \rho) \in D$, put $D_1 = \{ (a\_, \rho), (a, \rho) \}$, otherwise put $D_1 = \emptyset$, and let $D' = D \setminus D_1$. Let $\epsilon': D' \rightarrow \{ 1, -1 \}$ be defined in the following way:
    \begin{itemize}
        \item for $(b, \rho') \in D'$ such that $\epsilon(b, \rho')$ is defined let $\epsilon'(b, \rho') = \epsilon(b, \rho')$,
        \item for $(b, \rho') \in D'$ such that $b\_$ is defined in $\Jord'_{\rho}$ and $(b\_, \rho') \neq ((a\_)\_, \rho)$, let $\epsilon'((b\_, \rho'), (b, \rho')) = \epsilon((b\_, \rho'), (b, \rho'))$,
        \item if in $\Jord_{\rho}$ we have $a = b\_$ and $c = (a\_)\_$, let
        \begin{equation*}
        \epsilon'((c, \rho), (b, \rho)) = \epsilon((c, \rho), (a\_, \rho)) \cdot \epsilon((a, \rho), (b, \rho)).
        \end{equation*}
    \end{itemize}
    Then $(\Jord', \sigma', \epsilon')$ is an admissible triple and $\sigma$ is a subrepresentation of
    \begin{equation*}
        \delta([\nu^{-\frac{a\_ - 1}{2}} \rho, \nu^{\frac{a-1}{2}} \rho]) \rtimes \sigma_{(\Jord', \sigma', \epsilon')}.
    \end{equation*}
    Moreover, there is a unique irreducible tempered subrepresentation $\tau$ of 
    \begin{equation} \label{indtemp}
        \delta([\nu^{-\frac{a\_ - 1}{2}} \rho, \nu^{\frac{a\_-1}{2}} \rho]) \rtimes \sigma_{(\Jord', \sigma', \epsilon')}
    \end{equation}
    such that $\sigma$ is a unique irreducible subrepresentation of $\delta([\nu^{\frac{a\_ + 1}{2}} \rho, \nu^{\frac{a-1}{2}} \rho]) \rtimes \tau$. Furthermore, we have
    \begin{enumerate}[(i)]
        \item If there is $b \in \Jord_{\rho}$ such that $b\_ = a$, then $\epsilon((a, \rho), (b, \rho)) = 1$ if and only if $\tau$ is a unique irreducible subrepresentation of (\ref{indtemp}) which contains an irreducible constituent of the form 
        \begin{equation*}
        \delta([\nu^{\frac{a\_ + 1}{2}} \rho, \nu^{\frac{b-1}{2}} \rho]) \otimes \pi
       \end{equation*}
        in the Jacquet module with respect to an appropriate parabolic subgroup. 
        \item If there is $b \in \Jord_{\rho}$ such that $(a\_)\_ = b$, then $\epsilon((b, \rho), (a\_, \rho)) = 1$ if and only if $\tau$ is a unique irreducible subrepresentation of (\ref{indtemp}) which contains an irreducible constituent of the form 
        \begin{equation*}
        \delta([\nu^{\frac{b + 1}{2}} \rho, \nu^{\frac{a\_ - 1}{2}} \rho]) \times \delta([\nu^{\frac{b + 1}{2}} \rho, \nu^{\frac{a\_ - 1}{2}} \rho]) \otimes \pi
       \end{equation*}
        in the Jacquet module with respect to an appropriate parabolic subgroup.
        \item If $a$ is even and $a\_ = \min(\Jord_{\rho})$, then
        $\epsilon(a\_, \rho) = 1$ if and only if $\tau$ is a unique irreducible subrepresentation of (\ref{indtemp}) which contains an irreducible constituent of the form 
        \begin{equation*}
        \delta([\nu^{\frac{1}{2}} \rho, \nu^{\frac{a\_ - 1}{2}} \rho]) \times \delta([\nu^{\frac{1}{2}} \rho, \nu^{\frac{a\_ - 1}{2}} \rho]) \otimes \pi
       \end{equation*}
        in the Jacquet module with respect to an appropriate parabolic subgroup.
    \end{enumerate}
\end{enumerate}
\end{thrm}
We note that characterizations $(i), (ii)$ and $(iii)$ from the previous theorem do not appear in \cite{Moe1, MT1}. Since all our results also hold in the classical group case, in this way we obtain a useful tool for identification of discrete series subquotients of induced representations of both classical and odd GSpin groups. Simpler versions of these characterizations have recently played an important role in \cite{KLM}, \cite{Matic12} and \cite{Matic13}.  

It is not our aim to discuss the relation between the obtained bijective correspondence and the Langlands parametrization, since the extended Langlands parametrization for the GSpin groups is still conjectural. However, we expect the consequences analogous to the ones observed in \cite{MT1}, i.e., elements appearing in the Jordan blocks should correspond to the conjectural discrete Langlands parameter, i.e., the equivalence classes of semi-simple morphisms $\varphi$ from $W_F \times SL(2, \mathbb{C})$ to the $L$-group $^{L}GSpin_{2n+1} = GSp_{2n}(\mathbb{C}) \rtimes W_F'$ satisfying several conditions where $W_F'$ is the Weil-Deligne group of $F$. Note that the Langlands parameter is called discrete if it does not factor through a proper Levi subgroup. Also, the $\epsilon$-function should correspond to morphisms from Cent$_{^{L}G}($Im$(\varphi)$) to $\{ \pm 1 \}$ (see \cite{CFK, HK, Kim3, Moe2} for recent works on the local Langlands correspondence for odd GSpin groups). In our article, we are not able to discuss the dimension relation for $GSpin_{2n+1}$; it is conjecturally $\displaystyle\sum\limits_{(a, \rho) \in Jord(\sigma)} a \cdot $dim$\rho = 2n$ for a discrete series $\sigma$ of $GSpin_{2n+1}$.

A classification of discrete series for symplectic and odd-orthogonal groups over a non-archimedean local field of characteristic zero based on the LLC approach is also given in \cite{Xu1}, but at the moment these methods do not seem to be applicable to the GSpin case.

We take a moment to describe the contents of the paper in more detail. In the second section we recall the required notation and preliminaries. The third and the fourth sections are the technical heart of the paper. In those two sections we introduce several invariants of discrete series and prove many of their basic properties. In the fifth section our main results are stated and proved.

First author has been supported by the National Research Foundation
of Korea (NRF) grant funded by the Korea government (MSIP) \\ (No.\ $2017R1C1B2010081$).



\section{Preliminaries}

Let $F$ be a non-archimedean local field of characteristic zero. Let $\textbf{G}_n$ denote a split general spin group $\textbf{GSpin}_{2n+1}$ of semisimple rank $n$ defined over $F$, i.e., the $F$-split connected reductive algebraic group having based root datum dual to that of $GSp_{2n}$. Here $GSp_{2n}$ stands for the split reductive linear algebraic group of type $B_n$ whose derived subgroup is a double covering of a split special orthogonal group. Let $G_n$ denote the group of $F$-points of $\textbf{G}_n$. Similarly, let $\textbf{GL}_n$ denote a general linear group of rank $n$ defined over $F$ and let $GL_n$ denote the group of its $F$-points. All the representations of $p$-adic groups which we consider will be smooth.


Let $\Irr(GL_n)$ denote the set of all irreducible admissible representations of $GL_n$, and let $\Irr(G_n)$ denote the set of all irreducible admissible representations of $G_n$. Let $R(GL_n)$ stand for the Grothendieck group of the abelian category of finite length admissible representations of $GL_n$ and define $R(GL) = \oplus_{n \geq 0} R(GL_n)$. Similarly, let $R(G_n)$ stand for the Grothendieck group of the abelian category of finite length admissible representations of $G_{n}$ and define $R(G) = \oplus_{n \geq 0} R(G_n)$.

We fix a choice of a Borel subgroup as in \cite{BanGold}. Let $s = (n_1, n_2, \ldots, n_k)$ denote an ordered partition of some $n' \leq n$ and let $P_s = M_s N_s$ denote the standard parabolic subgroup of $G_n$ corresponding to the partition $s$. It follows from \cite[Theorem~2.7]{Asgari1} that the Levi factor $M_s$ is isomorphic to $GL_{n_1} \times GL_{n_2} \times \cdots \times GL_{n_k} \times G_{n-n'}$. If $\delta_{i}$ is a representation of $GL_{n_{i}}$, for $i = 1, 2, \ldots, k$, and $\tau$ a representation of $G_{n-n'}$, we denote by $\delta_{1} \times \d_2 \times \cdots \times \delta_{k} \rtimes \tau$ the representation Ind$^{G_n}_{M_s}(\delta_{1} \otimes \delta_2 \otimes \cdots \otimes \delta_{k} \otimes \tau)$ of $G_n$ induced from the representation $\delta_{1} \otimes \delta_2 \otimes \cdots \otimes \delta_{k} \otimes \tau$ of $M_s$ using normalized parabolic induction. We use a similar notation to denote a parabolically induced representation of $GL_{m}$.

For any irreducible admissible representation $\pi$ of $GL_n$, and for $0 \leq k \leq n$, let $r_{(k)}(\pi)$ denote the normalized Jacquet module of $\pi$ with respect to the standard parabolic subgroup having Levi subgroup isomorphic to $GL_{k} \times GL_{n-k}$, and we abuse notation to identify $r_{(k)}(\pi)$ with its semisimplification in $R(GL_{k}) \otimes R(GL_{n-k})$. We define $m^{\ast}(\pi) = \sum_{k=0}^{n} (r_{(k)} (\pi)) \in R(GL) \otimes R(GL)$, for an irreducible representation $\pi$ of $GL_n$, and then extend $m^{\ast}$ linearly to the whole of $R(GL)$.

Let us denote by $\nu$ the composition of the determinant mapping with the normalized absolute value on $F$. Let $\rho \in \Irr(GL_{k})$ denote a cuspidal
representation. By a segment of cuspidal representations of $GL_k$ we mean a set of the form $\{ \rho, \nu \rho, \ldots, \nu^{m} \rho \} \subset R(GL_k)$, which we denote by $[\rho, \nu^{m} \rho]$.

Let $\rho$ denote an irreducible unitary cuspidal representation of $GL_{k}$ and let $a, b \in \mathbb{R}$ are such that $b - a$ is a non-negative integer.
The induced representation $\nu^{b} \rho \times \nu^{b-1} \rho \times \cdots \times \nu^{a} \rho$ has a unique irreducible subrepresentation, which we denote by $\delta ([\nu^{a} \rho, \nu^{b} \rho ])$. By the results of \cite{Zel}, assigning $\delta ([\nu^{a} \rho, \nu^{b} \rho ])$ to segment $[\nu^{a} \rho, \nu^{b} \rho ]$ gives a bijection between appropriately long segments and isomorphism classes of irreducible essentially square-integrable representations in $R(GL)$. 

We frequently use the following equation \cite[Proposition~3.4]{Zel}:
\begin{equation*}
m^{\ast}(\delta([\nu^{a} \rho, \nu^{b} \rho])) = \sum_{i=a-1}^{b} \delta([\nu^{i+1} \rho, \nu^{b} \rho]) \otimes \delta([\nu^{a} \rho, \nu^{i} \rho]).
\end{equation*}

For a representation $\sigma \in R(G_{n})$ and $1 \leq k \leq n$, we denote by $r_{(k)}(\sigma)$ the normalized Jacquet module of $\sigma$ with respect to the parabolic subgroup $P_{(k)}$ having the Levi subgroup equal to $GL_{k} \times G_{n-k}$. We abuse notation to identify $r_{(k)}(\sigma)$ with its semisimplification in $R(GL_{k}) \otimes R(G_{n-k})$ and consider
\begin{equation*}
\mu^{\ast}(\sigma) = 1 \otimes \sigma + \sum_{k=1}^{n} r_{(k)}(\sigma) \in R(GL) \otimes R(G).
\end{equation*}

There is a natural partial order on the Grothendieck groups which we consider: $\pi_1 \leq \pi_2$ if $m(\rho, \pi_1) \leq m(\rho, \pi_2)$ for all irreducible smooth $\rho$, where $m(\rho, \pi_i)$ denotes the multiplicity of $\rho$ in $\pi_i$ for $i=1, 2$.

Note that if a twist of some irreducible unitarizable representation $\rho \in R(GL)$ 'appears in' the cuspidal support of a discrete series $\s \in R(G)$, then $\rho$ is an essentially self-dual representation \cite[Proposition~2.5,~Remark~2.3]{Kim1}, i.e., if $\rho \cong \nu^{-e(\rho)} \rho^{u}$, where $\rho^{u}$ is unitarizable, then $\rho^{u} \cong \widetilde{\rho^{u}}\otimes (\omega_{\sigma} \circ {\textrm det})$, where $\omega_{\sigma}$ is the central character of $\sigma$ restricted to the identity component of the center of $G$ and consider it as the character of $F^{\times}$.


We take a moment to state a result, derived in \cite[Theorem~3.4]{Kim1}, which presents a crucial structural formula for our calculations of Jacquet modules of induced representations.

\begin{thrm}[Structural formula] \label{osn}
Let $\rho$ denote an irreducible essentially self-dual unitarizable cuspidal representation of $GL_n$ and let $k, l \in \mathbb{R}$ be such that $k + l$ is a non-negative integer. Let $\sigma \in R(G)$ be an admissible representation of finite length. Write $\mu^{\ast}(\sigma) = \sum_{\tau, \sigma'} \tau \otimes \sigma'$. Then the following holds:
\begin{align*}
\mu^{\ast}(\delta([\nu^{-k} \rho, \nu^{l} \rho]) \rtimes \sigma) & = \sum_{i = -k-1}^{l} \sum_{j=i}^{l} \sum_{\tau, \sigma'} \delta([ \nu^{-i} \rho , \nu^{k} \rho]) \times \\ & \qquad \times  \delta([\nu^{j+1} \rho, \nu^{l} \rho]) \times \tau \nonumber \otimes \delta([\nu^{i+1} \rho, \nu^{j} \rho]) \rtimes \sigma'.
\end{align*}
We omit $\delta([\nu^{x} \rho, \nu^{y} \rho])$ if $x > y$.
\end{thrm}

We recall the definition of strongly positive representations of GSpin groups. Note that the classification of strongly positive representation is typically the first step towards the classification of discrete series. As with the M{\oe}glin-Tadi\'{c} classification, it will turn out that strongly positive discrete series representations correspond to admissible triples of alternated type (Proposition \ref{propsp} and Definition \ref{deftripl}). 

\begin{definition}[Strongly positive]\label{defsp}
An irreducible representation $\sigma \in Irr(G)$ is called strongly positive if for every embedding
\begin{equation*}
    \sigma \hookrightarrow \nu^{s_1}\rho_1 \times \nu^{s_2}\rho_2 \times \cdots \times \nu^{s_k}\rho_k \rtimes \sigma_{cusp}
\end{equation*}
\begin{sloppypar}
\noindent where $\rho_i \in R(GL)$ is an irreducible unitary cuspidal representation for $i=1, 2, \ldots, k$, $\sigma_{cusp}$ is an irreducible cuspidal representation of $G_{n'}$ and $s_i \in \mathbb{R}, i=1, 2, \ldots, k$, we have $s_i >0$ for each $i$.
\end{sloppypar}
\end{definition}

We also recall the square-integrability criterion for GSpin groups \cite[Proposition 4.2]{Asgari1}.

\begin{proposition}[Square-integrability criterion]\label{propsi}
Let $\beta_i = (1, \ldots, 1, 0, \ldots$, $0) \in \mathbb{R}^n$, where $1$ appears $i$ times. Let $\sigma \hookrightarrow \nu^{e(\rho_1)}\rho_1 \times \nu^{e(\rho_2)}\rho_2 \times \cdots \times \nu^{e(\rho_k)}\rho_k \rtimes \sigma_{cusp}$ be an irreducible representation of $G_n$, where $\rho_i$ is an irreducible cuspidal unitary representation of $GL_{n_i}$ and $\sigma_{cusp}$ is an irreducible cuspidal representation of $G_{n'}$. We set
\begin{equation*}
e_\star(\sigma)=(e(\rho_1), \ldots, e(\rho_1), \ldots, e(\rho_k), \ldots, e(\rho_k), 0, \ldots, 0) \in \mathbb{R}^n
\end{equation*}
(Here $e(\rho_i)$ appears $n_i$ times for $i=1, 2, \ldots, k$).
If $\sigma$ is square integrable (resp.\ tempered), then
\begin{equation*}
(e_\star(\sigma), \beta_{n_1}) >0, (e_\star(\sigma), \beta_{n_1+n_2}) >0,  \ldots, (e_\star(\sigma), \beta_{n_1+ \cdots + n_k}) >0
\end{equation*}
\begin{equation*}
(resp.\ (e_\star(\sigma), \beta_{n_1}) \geq 0, (e_\star(\sigma), \beta_{n_1+n_2}) \geq 0, \ldots, (e_\star(\sigma), \beta_{n_1+ \cdots + n_k}) \geq 0).
\end{equation*}
Conversely, if the above inequalities hold for all embeddings $\sigma \hookrightarrow \nu^{e(\rho_1)}\rho_1 \times \nu^{e(\rho_2)}\rho_2 \times \cdots \times \nu^{e(\rho_k)}\rho_k \rtimes \sigma_{cusp}$ where $\rho_i$ is an irreducible cuspidal unitary representation of $GL_{n_i}$ and $\sigma_{cusp}$ is an irreducible cuspidal representation of $G_{n'}$, then 
$\sigma$ is square integrable (resp.\ tempered).
\end{proposition}

Throughout the paper, for simplicity of the notation we abbreviate a `discrete series representation' or a `representation belonging to the discrete series' to simply `discrete series'.

Following the same lines as in proofs of \cite[Lemma~3.4, Theorem~3.5]{Matic5}, which completely rely on the representation theory of the general linear group and the square-integrability criterion, so can be applied to our situation, we obtain:

\begin{corollary} \label{notds}
Suppose that $\pi \in R(G)$ is not square-integrable (resp. not tempered). Then there exist $a, b$ such that $b - a \in \mathbb{Z}$ and $a + b \leq 0$ (resp. $a+b < 0$), an irreducible cuspidal representation $\rho \in R(GL)$, and an irreducible representation $\pi' \in R(G)$, such that $\pi$ is a subrepresentation of $\delta([\nu^{a} \rho, \nu^{b} \rho]) \rtimes \pi'$.
\end{corollary}

We briefly recall the subrepresentation version of the Langlands classification for general linear groups.

\begin{sloppypar}
For an irreducible essentially square-integrable representation $\delta \in R(GL)$, there is a unique $e(\delta) \in \mathbb{R}$ such that $\nu^{-e(\delta)} \delta$ is unitarizable. Note that $e(\delta([\nu^{a} \rho, \nu^{b} \rho ])) = (a + b) / 2$. Suppose that $\delta_{1}, \delta_{2}, \ldots, \delta_{k}$ are irreducible essentially square-integrable representations such that $e(\delta_{1}) \leq e(\delta_{2}) \leq \cdots \leq e(\delta_{k})$. Then the induced representation $\delta_{1} \times \delta_{2} \times \cdots \times \delta_{k}$ has a unique irreducible subrepresentation, which we denote by $L(\delta_{1}, \delta_{2}, \ldots, \delta_{k})$. This irreducible subrepresentation is called the Langlands subrepresentation, and it appears with multiplicity one in the composition series of $\delta_{1} \times \delta_{2} \times \cdots \times \delta_{k}$. Every irreducible representation $\pi \in R(GL)$ is isomorphic to some $L(\delta_{1}, \delta_{2}, \ldots, \delta_{k})$ and, for a given $\pi$, the representations $\delta_{1}, \delta_{2}, \ldots, \delta_{k}$ are unique up to a permutation.
\end{sloppypar}

Similarly, throughout the paper we use the subrepresentation version of the Langlands classification for $G_n$, since it is more appropriate for our Jacquet module considerations. So, we realize a non-tempered irreducible representation $\pi$ of $G_{n}$ as a unique irreducible subrepresentation of an induced representation of the form $\delta_{1} \times \delta_{2} \times \cdots \times \delta_{k} \rtimes \tau$, where $\tau$ is a tempered representation of some $G_{t}$, and $\d_1, \d_2, \ldots, \d_k \in R(GL)$ are irreducible essentially square-integrable representations such that $e(\d_1) \leq e(\d_2) \leq \cdots \leq e(\d_k) < 0$. In this case, we write $\pi = L(\delta_{1}, \delta_{2}, \ldots, \delta_{k}, \tau)$ and, for a given $\pi$, the representations $\delta_{1}, \delta_{2}, \ldots, \delta_{k}$ are unique up to a permutation.

The following result \cite[Lemma~5.5]{Jan3}, whose proof is valid in the GSpin case, is used several times.
\begin{lemma} \label{lemajantz}
Suppose that $\pi \in R(G_n)$ is an irreducible representation, $\lambda$ an irreducible representation of the Levi subgroup $M$ of $G_n$, and $\pi$ is a subrepresentation of Ind$_{M}^{G_{n}}(\lambda)$. If $L > M$ is a Levi subgroup of $G_n$, then there is an irreducible subquotient $\rho$ of Ind$_{M}^{L}(\lambda)$ such that $\pi$ is a subrepresentation of Ind$_{L}^{G_{n}}(\rho)$.
\end{lemma}

\section{Invariants of discrete series I: the $\epsilon$-function on pairs}

In this section we introduce several invariants of discrete series and obtain their basic properties.

A partial cuspidal support of a discrete series $\sigma \in \Irr(G_{n})$ is an irreducible cuspidal representation $\sigma_{cusp}$ of some $G_{m}$ such that there exists a representation $\pi \in R(GL_{n - m})$ such that $\sigma$ is a subrepresentation of $\pi \rtimes \sigma_{cusp}$. We note that it follows directly from \cite[Proposition~2.5]{Kim1} that such a representation $\sigma_{cusp}$ is unique. From now on, for any irreducible admissible representation $\sigma$, $\sigma_{cusp}$ will denote its partial cuspidal support.

\begin{definition} \label{defjord}
The Jordan block of a discrete series $\sigma \in R(G)$, which we denote by $\Jord(\sigma)$, is the set of all pairs $(a, \rho)$ where $\rho$ is an irreducible cuspidal unitarizable essentially self-dual representation of some $GL_{n_{\rho}}$ and $a$ is a positive integer such that the following two conditions are satisfied:
\begin{enumerate}[(1)]
  \item The positive integer $a$ is even if and only if $L(s, \rho, r)$ has a pole at $s=0$. Here, the local $L$-function $L(s, \rho, r)$ is the one defined by Shahidi (see for instance \cite{Sh1}, \cite{Sh2}), and $r =$ Sym$^2 \mathbb{C}^{n_{\rho}} \otimes \mu^{-1}$, where Sym$^2 \mathbb{C}^{n_{\rho}}$ is the symmetric-square representation of the standard representation on $\mathbb{C}^{n_{\rho}}$ of $\textbf{GL}_{n_{\rho}}(\mathbb{C})$ and $\mu$ is the similitude character of $\hat{G}_n$, as in \cite[Proposition~5.6]{Asgari1}.
  \item The induced representation $\delta([\nu^{-\frac{a-1}{2}} \rho, \nu^{\frac{a-1}{2}} \rho]) \rtimes \sigma$ is irreducible.
\end{enumerate}
\end{definition}

Note that $\Jord(\sigma)$ is not a multiset.
For a given irreducible cuspidal unitarizable essentially self-dual representation $\rho$ of $GL$, we let $\Jord_{\rho}(\sigma):=\{a :  (a, \rho) \in \Jord(\sigma) \}$.

For $z \in (1/2)\mathbb{Z}$ let $\mu(z, \s)(s)$ be the Plancherel measure which is the composite of two standard intertwining operators:
$$
\nu^s \delta([\nu^{-\frac{z-1}{2}}\rho,\nu^{\frac{z-1}{2}}\rho])\rtimes \s \rightarrow \nu^{-s} \delta([\nu^{-\frac{z-1}{2}}\rho,\nu^{\frac{z-1}{2}}\rho])\rtimes \s \rightarrow \nu^s \delta([\nu^{-\frac{z-1}{2}}\rho,\nu^{\frac{z-1}{2}}\rho])\rtimes \s.$$

\begin{lemma}\label{poleofplancherel}
$z \in \Jord(\s)$ if and only if $\mu(z, \s)(s)$ has a pole at $s=0$
\end{lemma}
\begin{proof}
This follows from the condition (2) from the definition of $\Jord(\s)$ and the result of Harish-Chandra (\cite[Section~2]{BanGold}, see also \cite{Wald3}).
\end{proof}
First we prove an analogue of \cite[Proposition~2.1]{MT1}.

\begin{proposition} \label{propjord}
Let $\s \in \Irr(G_n)$ denote a discrete series and let $\rho \in \Irr(GL_{n_{\rho}})$ denote an essentially self-dual cuspidal unitarizable representation. Suppose that $x, y$ are half-integers such that $x - y$ is a non-negative integer, and assume that $x, y \in \mathbb{Z}$ if and only if $L(s, \rho, r)$ has no pole at $s=0$. If there is an embedding
\begin{equation*}
\s \hookrightarrow \nu^{x} \rho \times \nu^{x-1} \rho \times \cdots \times \nu^{y} \rho \rt \s',
\end{equation*}
where $\s' \in \Irr(G_m)$ is a discrete series, then the following holds:
\begin{enumerate}[(1)]
\item If $y > 0$, then $2y- 1 \in \Jord_{\rho}(\s')$ and $\Jord_{\rho}(\s) = \Jord_{\rho}(\s') \cup \{ 2x + 1 \} \setminus \{ 2y-1 \}$.
\item If $y < 0$, then $\Jord_{\rho}(\s) = \Jord_{\rho}(\s') \cup \{ 2x + 1, 1-2y \}$. Also, $2x + 1, 1-2y \notin \Jord_{\rho}(\s')$.
\end{enumerate}
\end{proposition}
\begin{proof}

First, we note that for cuspidal unitarizable representations $\rho_1$ and $\rho_2 \in \Irr(GL)$ and non-negative integers $z_1$ and $z_2$, the normalizing factor, modulo a holomorphic invertible function of $s$, corresponding to an intertwining operator between two essentially square integrable representations
\begin{multline*}
\nu^{s}\delta([ \nu^{-\frac{z_1-1}{2}} \rho_1, \nu^{\frac{z_1-1}{2}} \rho_1]) \times \nu^{s'}\delta([ \nu^{-\frac{z_2-1}{2}} \rho_2, \nu^{\frac{z_2-1}{2}} \rho_2]) \rightarrow
\\
\nu^{s'}\delta([ \nu^{-\frac{z_2-1}{2}} \rho_2, \nu^{\frac{z_2-1}{2}} \rho_2]) \times \nu^{s}\delta([ \nu^{-\frac{z_1-1}{2}} \rho_1, \nu^{\frac{z_1-1}{2}} \rho_1])
\end{multline*}
is, by \cite{Sh0} (see also \cite[I.4.~the~formula~(1)]{MW}), equal to
\begin{multline} \label{formulaone}
L(s-s', \delta([ \nu^{-\frac{z_1-1}{2}} \rho_1, \nu^{\frac{z_1-1}{2}} \rho_1]) \times \widetilde{\delta([ \nu^{-\frac{z_2-1}{2}} \rho_2, \nu^{\frac{z_2-1}{2}} \rho_2])})\cdot \\
\cdot \left(L(1+s-s', \delta([ \nu^{-\frac{z_1-1}{2}} \rho_1, \nu^{\frac{z_1-1}{2}} \rho_1]) \times \widetilde{\delta([ \nu^{-\frac{z_2-1}{2}} \rho_2, \nu^{\frac{z_2-1}{2}} \rho_2])}) \right)^{-1} = \\
=L(s-s'+ |(z_1-z_2)/2|, \rho_1 \times \widetilde{\rho_2})\left(L(s-s'+(z_1+z_2)/2, \rho_1 \times \widetilde{\rho_2}) \right)^{-1}.
\end{multline}

As in \cite[Proposition~2.1]{MT1}, to calculate the Plancherel measure $\mu(z, \s)(s)$ modulo a holomorphic invertible function of $s$, we apply the formula (\ref{formulaone}). Using the factorization of intertwining operators \cite{Sh0} or \cite[Theorem 4.2.2]{Sh3}, $\mu(z, \s)(s)$ consists of two parts: $\mu(z, \s')(s)$ and the standard intertwining operators between representations of $GL$. Since the computation of the standard intertwining operators between representations of $GL$ modulo a holomorphic invertible function of $s$ in terms of $L$-functions is well known (\ref{formulaone}), we have the following equality modulo a holomorphic invertible function of $s$:

\begin{multline}\label{Plancherel}
\mu(z, \s)(s)= \mu(z, \s')(s) \cdot \frac{L(s-x+(z-1)/2, \rho \times \widetilde{\rho})}{L(s-y+(z-1)/2+1, \rho \times \widetilde{\rho})} \cdot \\
\cdot \frac{L(s+y+(z-1)/2, \rho \ti \widetilde{\rho})}{L(s+x+(z-1)/2+1, \rho \ti \widetilde{\rho})} \cdot \frac{L(-s-x+(z-1)/2, \rho \ti \widetilde{\rho})}{L(-s-y+(z-1)/2+1, \rho \ti \widetilde{\rho})} \cdot \\
\cdot \frac{L(-s+y+(z-1)/2, \rho \ti \widetilde{\rho})}{L(-s+x+(z-1)/2+1, \rho \ti \widetilde{\rho})}.
\end{multline}

We now use Lemma \ref{poleofplancherel} to prove the proposition.
First, it is known that $L(s, \rho \times \widetilde{\rho})$ has a pole only at $s=0$ and it is a simple pole. Furthermore, it is non-zero. Therefore, the product of the $L$-functions in (\ref{Plancherel}) has a (double) pole at $s=0$ if and only if either $x=(z-1)/2$ or $y=-(z-1)/2$ and has zero if and only if either $y=(z-1)/2+1$ or $x=-(z-1)/2-1$. Furthermore, it is also known that the Plancherel measure (both $\mu(z, \s)(s)$ and $\mu(z, \s')(s)$) has order zero or two at $s=0$. Therefore, $\mu(z, \s)(s)$ has a pole at $s=0$ if and only if either one of the following cases holds:
\begin{itemize}
\item $\mu(z, \s')(s)$ has a pole at $s=0$ and $y \neq (z-1)/2+1$,
\item $x=(z-1)/2$,
\item $y=-(z-1)/2$.
\end{itemize}
Note that the case $x \neq -(z-1)/2-1$ always holds since $-(z-1)/2-1 \leq -1$.

If $y >0$, the case $y=-(z-1)/2$ cannot happen and therefore, $\Jord_{\rho}(\s) = \Jord_{\rho}(\s') \cup \{ 2x+1 \} \setminus \{ 2y-1 \}$.    

If $y \leq 0$, the case $y \neq (z-1)/2+1$ always holds. Therefore, we have
$\Jord_{\rho}(\s) = \Jord_{\rho}(\s') \cup \{ 2x+1, -2y+1 \}$.
\end{proof}

Let $\sigma_{cusp} \in R(G)$ denote an irreducible cuspidal representation, and let $\rho \in R(GL)$ stand for an irreducible self-dual cuspidal representation. By  \cite[Th\'{e}oreme~3.1.1]{Moe2}, there is a unique positive integer $\alpha$ such that $\nu^{\frac{\alpha-1}{2}} \rho \rtimes \sigma_{cusp}$ reduces. 

Following the same lines as in the proofs of \cite[Theorems~2.3,~2.5]{Tad8} and in the proofs of \cite[Propositions~4.1,~4.2]{Tad2}, all of which completely rely on the structural formula, so can be directly applied to the odd GSpin situation, we deduce that 
$\delta([\nu^{-\frac{a-1}{2}} \rho, \nu^{\frac{a-1}{2}} \rho]) \rtimes \sigma_{cusp}$ reduces if and only if $a \geq \alpha$. This implies the so-called Basic Assumption for the odd GSpin groups: $\Jord_{\rho}(\sigma_{cusp})$ is finite and
\begin{itemize}
    \item if $\nu^{\frac{\alpha-1}{2}} \rho \rtimes \sigma_{cusp}$ reduces for $\alpha \geq 3$ then $\alpha = \max(\Jord_{\rho}(\sigma_{cusp}))$, and for $a \in \Jord_{\rho}(\sigma_{cusp})$ such that $a \geq 3$ we have $a-2 \in \Jord_{\rho}(\sigma_{cusp})$,
    \item if $\nu^{\frac{\alpha-1}{2}} \rho \rtimes \sigma_{cusp}$ reduces for $\alpha \in \{ 1, 2 \}$, then $\Jord_{\rho}(\sigma_{cusp}) = \emptyset$.
\end{itemize}
The previous proposition will enable us to connect $\Jord(\sigma)$ and $\Jord(\sigma_{cusp})$.

For any ordered pair $(\rho, \sigma_{cusp})$ consisting of an irreducible self-dual cuspidal representation $\rho \in R(GL)$ and an irreducible cuspidal representation $\sigma_{cusp} \in R(G)$ such that $\rho \rtimes \sigma_{cusp}$ reduces, we fix once and for all a choice of labeling of mutually non-isomorphic irreducible tempered representations $\tau_1$ and $\tau_{-1}$ such that in $R(G)$ we have $\rho \rtimes \sigma_{cusp} = \tau_1 + \tau_{-1}$. We note that such equality in $R(G)$ follows from \cite[Theorem~2.6]{BanGold}.

\begin{lemma} \label{remarkprva}
Let $\s \in \Irr(G_n)$ be a discrete series, let $\rho \in \Irr(GL_{k})$ be a cuspidal unitarizable representation, and let $x \in \mathbb{R}$ be such that there exists an irreducible representation $\sigma'$ of $G_{n-k}$ such that $\s$ is a subrepresentation of the induced representation
\begin{equation*}
\nu^{x} \rho \rt \sigma'.
\end{equation*}
Then $x$ is a half-integer, and $(2x+1, \rho) \in \Jord(\s)$.
\end{lemma}
\begin{proof}
First, \cite[Remark~2.3]{Kim1} and \cite[Th\'{e}oreme~3.1.1]{Moe2} imply that $x$ is a half-integer.
Let us now prove that $\s'$ has to be a tempered representation. Otherwise, by Corollary \ref{notds}, there are $x_1, y_1$ such that $x_1 - y_1 \in \mathbb{Z}$ and $x_1 + y_1 < 0$ and representations $\rho_1 \in \Irr(GL_{k_1}), \pi_1 \in \Irr(G_{n_1})$, such that $\s' \hookrightarrow \delta([\nu^{x_1} \rho_1, \nu^{y_1} \rho_1]) \rtimes \pi_1$. Thus, $\s$ is a subrepresentation of $\nu^{x} \rho \times \delta([\nu^{x_1} \rho_1, \nu^{y_1} \rho_1]) \rtimes \pi_1$. 

If $\nu^{x} \rho \times \delta([\nu^{x_1} \rho_1, \nu^{y_1} \rho_1])$ is irreducible, we have 
\begin{equation*}
\nu^{x} \rho \times \delta([\nu^{x_1} \rho_1, \nu^{y_1} \rho_1]) \cong \delta([\nu^{x_1} \rho_1, \nu^{y_1} \rho_1]) \times \nu^{x} \rho,  
\end{equation*}
which leads to embeddings
\begin{equation*}
\sigma \hookrightarrow \delta([\nu^{x_1} \rho_1, \nu^{y_1} \rho_1]) \times \nu^{x} \rho \rtimes \pi_1  
\end{equation*}
and
\begin{equation*}
\sigma \hookrightarrow \nu^{y_1} \rho_1 \times \nu^{y_1 - 1} \rho_1 \times \cdots \times \nu^{x_1} \rho_1 \times \nu^{z_1} \rho'_1 \times \cdots \times \nu^{z_k} \rho'_k \rtimes \pi_2, 
\end{equation*}
\begin{sloppypar}
\noindent for cuspidal representations $\rho'_1, \ldots, \rho'_k \in \Irr(GL), \pi_2 \in \Irr(G_{n_2})$. Since $x_1 + y_1 < 0$ and $\sigma$ is square-integrable, this contradicts Proposition \ref{propsi}. Thus, $\nu^{x} \rho \times \delta([\nu^{x_1} \rho_1, \nu^{y_1} \rho_1])$ reduces, so  since $x > 0$ by \cite[Theorem~4.2]{Zel} we have $\rho_1 \cong \rho$ and $y_1 = x - 1$. 
\end{sloppypar}
By Lemma \ref{lemajantz}, there is an irreducible subquotient $\pi$ of $\nu^{x} \rho \times \delta([\nu^{x_1} \rho_1, \nu^{y_1} \rho_1])$ such that $\sigma$ is a subrepresentation of $\pi \rtimes \pi_1$, and it follows that $\pi \cong \delta([\nu^{x_1} \rho_1, \nu^{x} \rho_1])$, which is impossible since $x_1 + x \leq 0$. Thus, $\s'$ is tempered, and if it is a discrete series the claim of the lemma follows from Proposition \ref{propjord}.

Let us now suppose that $\s'$ is not a discrete series. Then, by Corollary \ref{notds}, $\s'$ is a subrepresentation of an induced representation of the form
\begin{equation*}
\delta([\nu^{-x_1} \rho_1, \nu^{x_1} \rho_1]) \times \delta([\nu^{-x_2} \rho_2, \nu^{x_2} \rho_2]) \times  \cdots \times \delta([\nu^{-x_k} \rho_k, \nu^{x_k} \rho_k]) \rtimes \s'',
\end{equation*}
for a discrete series $\s'' \in \Irr(G_{n''})$ and $k \geq 1$. In the same way as in the first part of the proof, we deduce that $\rho_1 \cong \rho$, $x_1 = x-1$, and $\s$ is a subrepresentation of
\begin{equation*}
\delta([\nu^{-x + 1} \rho, \nu^{x} \rho]) \times \delta([\nu^{-x_2} \rho_2, \nu^{x_2} \rho_2]) \times  \cdots \times \delta([\nu^{-x_k} \rho_k, \nu^{x_k} \rho_k]) \rtimes \s''.
\end{equation*}
The square-integrability of $\s$ implies that $k=1$, since otherwise we would have $\s \hookrightarrow \delta([\nu^{-x_2} \rho_2, \nu^{x_2} \rho_2]) \rtimes \pi_2$, for some irreducible representation $\pi_2$. Consequently, $\s \hookrightarrow \delta([\nu^{-x + 1} \rho, \nu^{x} \rho]) \rtimes \s''$, and an application of Proposition \ref{propjord} finishes the proof.
\end{proof}

The following technical result will be used several times in the paper:

\begin{lemma} \label{lemaprva}
Let $\rho \in \Irr(GL_{n_{\rho}})$ denote a cuspidal unitarizable representation and let $a, b \in \mathbb{R}$ be such that $b - a$ is a nonnegative integer. Let $\s \in \Irr(G_n)$ be such that $\mu^{\ast}(\s)$ contains an irreducible constituent of the form $\d([\nu^{a} \rho, \nu^{b} \rho]) \ot \pi$ and such that $\mu^{\ast}(\s)$ does not contain an irreducible constituent of the form $\nu^{x} \rho \ot \pi_1$, for $a \leq x < b$. Then there is an irreducible representation $\pi_2$ such that $\s$ is a subrepresentation of $\d([\nu^{a} \rho, \nu^{b} \rho]) \rt \pi_2$.
\end{lemma}
\begin{proof}
It follows at once that there is an irreducible cuspidal representation $\pi_1 \in \Irr(G_{n_1})$ such that the Jacquet module of $\s$ with respect to the appropriate parabolic subgroup contains $\nu^{b} \rho \ot \nu^{b-1} \rho \ot \cdots \ot \nu^{a} \rho \ot \pi_1$. Using cuspidality of $\nu^{b} \rho \ot \nu^{b-1} \rho \ot \cdots \ot \nu^{a} \rho \otimes \pi_1$ and \cite[Corollary~6.2(3)]{Tad6}, which holds for reductive $p$-adic groups, we get that there is a representation $\pi_2 \in \Irr(G_{n_2})$ such that $\s$ is a subrepresentation of the induced representation $\nu^{b} \rho \times \nu^{b-1} \rho \times \cdots \times \nu^{a} \rho \rt \pi_2$. By Lemma \ref{lemajantz} there is an irreducible subquotient $\pi'$ of $\nu^{b} \rho \times \nu^{b-1} \rho \times \cdots \times \nu^{a} \rho$ such that $\sigma$ is a subrepresentation of $\pi' \rtimes \pi_2$. Since $\mu^{\ast}(\s)$ does not contain an irreducible constituent of the form $\nu^{x} \rho \otimes \pi_3$, for $a \leq x < b$, we directly obtain $\pi' \cong \d([\nu^{a} \rho, \nu^{b} \rho])$. This ends the proof.
\end{proof}

\begin{definition} \label{defone}
Let $\s \in \Irr(G_n)$ denote a discrete series. For an irreducible essentially self-dual cuspidal unitarizable representation $\rho$ of $GL_{n_{\rho}}$, we write $\Jord_{\rho}(\s) = \{ a : (a, \rho) \in \Jord(\s) \}$. Let $\sigma_{cusp}$ stand for the partial cuspidal support of $\sigma$. If $\Jord_{\rho}(\s) \neq \emptyset$ and $a \in \Jord_{\rho}(\s)$, we put $a\_ = \max \{ b \in \Jord_{\rho}(\s): b < a \}$, if it exists. The $\epsilon$-function is defined on a subset $D$ of $\Jord(\sigma) \cup \Jord(\sigma) \times \Jord(\sigma)$ such that
\begin{itemize}
    \item $D \cap \Jord(\sigma) = \{ (x, \rho) \in \Jord(\sigma) : x$ is even or $\Jord_{\rho}(\sigma_{cusp}) = \emptyset \}$,
    \item $D \cap (\Jord(\sigma) \times \Jord(\sigma)) = \{ ((x, \rho_1), (y, \rho_2)) : \rho_1 \cong \rho_2$ and in $\Jord_{\rho_1}(\s)$ we have $x = y\_\}$.
\end{itemize}
For $a \in \Jord_{\rho}(\s)$ such that $a\_$ is defined, we set
\begin{equation*}
\epsilon_{\s}((a\_, \rho), (a, \rho)) = 1
\end{equation*}
if there exists an irreducible representation $\pi$ of some $G_{n'}$ such that
\begin{equation}\label{ds:em1}
\s \hookrightarrow \d([\nu^{\frac{a\_ + 1}{2}} \rho, \nu^{\frac{a-1}{2}} \rho]) \rt \pi.
\end{equation}
Otherwise, let
\begin{equation*}
\epsilon_{\s}((a\_, \rho), (a, \rho)) = -1.
\end{equation*}
\end{definition}

In Lemmas \ref{lemaequiv}, \ref{priprema1}, \ref{epsilonequality} and Corollary \ref{corembedjord}, we let $\s \in \Irr(G_n)$ denote a discrete series.

Using Lemmas \ref{remarkprva} and \ref{lemaprva}, we obtain:
\begin{lemma} \label{lemaequiv}
\begin{sloppypar}
Let $(a, \rho) \in \Jord(\s)$ be such that $a\_$ is defined. Then $\epsilon_{\s}((a\_, \rho), (a, \rho))$ $= 1$ if and only if there exists an irreducible representation $\s'$ such that
\end{sloppypar}
\begin{equation*}
\mu^{\ast}(\s) \geq \d([\nu^{\frac{a\_ + 1}{2}} \rho, \nu^{\frac{a-1}{2}} \rho]) \otimes \s'.
\end{equation*}
\end{lemma}

Let us prove another useful technical result:

\begin{lemma} \label{priprema1}
Let $(2x+1, \rho) \in Jord(\s)$. Suppose that there is some $y$ such that $\s$ is a subrepresentation of an induced representation of the form $\delta([\nu^{y} \rho, \nu^{x} \rho]) \rtimes \pi$, for some irreducible representation $\pi$, and let $y_{\min}$ be the smallest such number. If $\pi$ is an irreducible representation such that $\s$ is a subrepresentation of $\delta([\nu^{y_{\min}} \rho, \nu^{x} \rho]) \rtimes \pi$, then $\pi$ is a discrete series.
\end{lemma}
\begin{proof}
Suppose, on the contrary, that $\pi$ is not a discrete series. By Corollary \ref{notds}, there are $x_1, y_1$ such that $x_1 - y_1 \in \mathbb{Z}$ and $x_1 + y_1 \leq 0$ and irreducible representations $\rho_1, \pi_1$, such that $\pi$ is a subrepresentation of $\delta([\nu^{y_1} \rho_1, \nu^{x_1} \rho_1]) \rtimes \pi_1$. By Lemma \ref{lemajantz}, there is an irreducible subquotient $\pi_2$ of $\delta([\nu^{y_{\min}} \rho, \nu^{x} \rho]) \times \delta([\nu^{y_1} \rho_1, \nu^{x_1} \rho_1])$ such that $\sigma$ is a subrepresentation of $\pi_2 \rtimes \pi_1$. The square-integrability criterion implies $\rho \cong \rho_1$ and $x_1 \geq y_{\min}-1$. It can now be easily seen that $\pi_2 \cong \delta([\nu^{y_1} \rho, \nu^{x} \rho]) \times \delta([\nu^{y_{\min}} \rho, \nu^{x_1} \rho])$.

Using Lemma \ref{lemajantz} again, we deduce that there is an irreducible representation $\pi_3 \in R(G)$ such that $\s$ is a subrepresentation of
$\delta([\nu^{y_1} \rho, \nu^{x} \rho]) \rtimes \pi_3$
and $y_1 < y_{\min}$, a contradiction.
\end{proof}

We note the following consequence of Proposition \ref{propjord} and Lemma \ref{priprema1}:

\begin{corollary} \label{corembedjord}
Let $(2x+1, \rho) \in Jord(\s)$. Suppose that there is some $y$, $y \leq 0$, such that $\s$ is a subrepresentation of an induced representation of the form $\delta([\nu^{y} \rho, \nu^{x} \rho]) \rtimes \s'$, for some irreducible representation $\s'$, and let $y_{\min}$ denote the smallest such number. Then $(-2y_{\min}+1, \rho) \in Jord(\s)$.
\end{corollary}

Now we prove:

\begin{lemma} \label{epsilonequality}
If $\epsilon_{\s}((a\_, \rho),(a, \rho)) = 1$, $\pi$ in $(\ref{ds:em1})$ is an irreducible tempered subrepresentation of an induced representation of the form
\begin{equation*}
\d([\nu^{-\frac{a\_ - 1}{2}} \rho, \nu^{\frac{a\_ - 1}{2}} \rho]) \rt \s_1,
\end{equation*}
\begin{sloppypar}
\noindent where $\s_1 \in \Irr(G_{n_1})$ is a discrete series such that $Jord(\s_1) = Jord(\s) \setminus \{(a, \rho), (a\_, \rho) \}$. Therefore,
\end{sloppypar}
\begin{equation*}
\s \hookrightarrow \d([\nu^{-\frac{a\_ - 1}{2}} \rho, \nu^{\frac{a - 1}{2}} \rho]) \rt \s_1.
\end{equation*}
Furthermore, $\epsilon_{\s}((a\_, \rho),(a, \rho)) = 1$ if and only if there is an irreducible representation $\pi_1 \in \Irr(G_{n_2})$ such that $\s$ is a subrepresentation of $\d([\nu^{-\frac{a\_ - 1}{2}} \rho, \nu^{\frac{a - 1}{2}} \rho]) \rt \pi_1$.
\end{lemma}
\begin{proof}
If $\pi$ were a discrete series, Proposition \ref{propjord} would give $(a\_, \rho) \notin \Jord(\sigma)$, a contradiction. Let us show that $\pi$ is tempered. Otherwise, by Corollary \ref{notds}, there are $x_1, y_1$ such that $x_1 - y_1 \in \mathbb{Z}$ and $x_1 + y_1 < 0$, and irreducible representations $\rho_1, \pi_1$, such that $\pi$ is a subrepresentation of
$\delta([\nu^{x_1} \rho_1, \nu^{y_1} \rho_1]) \rtimes \pi_1$. By Lemma \ref{lemajantz}, there is an irreducible subquotient $\pi'_1$ of $\d([\nu^{\frac{a\_ + 1}{2}} \rho, \nu^{\frac{a-1}{2}} \rho]) \times \delta([\nu^{x_1} \rho_1, \nu^{y_1} \rho_1])$ such that $\sigma$ is a subrepresentation of $\pi'_1 \rtimes \pi_1$.

The square-integrability criterion for $\s$ implies that $\d([\nu^{\frac{a\_ + 1}{2}} \rho, \nu^{\frac{a-1}{2}} \rho]) \times \delta([\nu^{x_1} \rho_1, \nu^{y_1} \rho_1])$ is reducible. Therefore Lemma \ref{remarkprva} implies that $\rho \cong \rho_1$, $y_1= \frac{a\_ - 1}{2}$, and $\pi'_1 \cong \delta([\nu^{x_1} \rho, \nu^{\frac{a-1}{2}} \rho]) \rtimes \pi_1$, for $x_1< - \frac{a\_ - 1}{2}$, which is impossible by Proposition \ref{propjord} and Lemma \ref{priprema1}.

So, $\pi$ is tempered and one readily sees that it is a subrepresentation of an induced representation of the form
\begin{equation} \label{indcud}
\delta([\nu^{- \frac{a\_ - 1}{2}} \rho, \nu^{\frac{a\_ - 1}{2}} \rho]) \times \cdots \times \delta([\nu^{- \frac{a\_ - 1}{2}} \rho, \nu^{\frac{a\_ - 1}{2}} \rho]) \rtimes \s_1,
\end{equation}
with $\s_1 \in \Irr(G_{n_1})$ discrete series.

Using Lemma \ref{lemajantz} and the square-integrability criterion, we get that $\sigma$ is a subrepresentation of
\begin{equation*}
\d([\nu^{-\frac{a\_ - 1}{2}} \rho, \nu^{\frac{a - 1}{2}} \rho]) \times \delta([\nu^{- \frac{a\_ - 1}{2}} \rho, \nu^{\frac{a\_ - 1}{2}} \rho]) \times \cdots \times \delta([\nu^{- \frac{a\_ - 1}{2}} \rho, \nu^{\frac{a\_ - 1}{2}} \rho]) \rt \s_1.
\end{equation*}
If $\delta([\nu^{- \frac{a\_ - 1}{2}} \rho, \nu^{\frac{a\_ - 1}{2}} \rho])$ appears in (\ref{indcud}) more than once, it follows that there is an irreducible representation $\pi'_1$ such that $\s$ is a subrepresentation of  $\delta([\nu^{- \frac{a\_ - 1}{2}} \rho,$ $\nu^{\frac{a\_ - 1}{2}} \rho]) \rtimes \pi'_1$, contradicting the square-integrability criterion. Thus, $\pi$ is a subrepresentation of
$\delta([\nu^{- \frac{a\_ - 1}{2}} \rho, \nu^{\frac{a\_ - 1}{2}} \rho]) \rtimes \s_1$. 

This implies that $\Jord(\s_1) = \Jord(\s) \setminus \{ (a\_, \rho), (a, \rho)\}$ and lemma is proved.
\end{proof}

Let us recall the characterization of strongly positive discrete series (see Definition \ref{defsp} for the definition), which can be deduced directly from \cite[Theorem~4.6]{Matic4} or from \cite[Section~7]{MatTad}. We emphasize that, although representations of the GSpin groups have not been studied in  \cite{Matic4} and \cite{MatTad}, all the proofs given in \cite{Matic4} and in \cite[Section~7]{MatTad} can be carried to the odd GSpin situation without any change, since they completely rely on the structural formula, classification of the strongly positive discrete series, which is analogous to the one for the odd GSpin groups, and well-known facts on the representation theory of general linear groups.

\begin{proposition} \label{propsp}
Let $\s \in \Irr(G_{n})$ denote a discrete series. Then $\s$ is strongly positive if and only if for all $(a, \rho) \in \Jord(\s)$ such that $a\_$ is defined we have
\begin{equation*}
\epsilon_{\s}((a\_, \rho), (a, \rho)) = -1.
\end{equation*}
\end{proposition}

We also note

\begin{thrm}  \label{embed}
Let $\s \in \Irr(G_{n})$ denote a non-strongly positive discrete series. Then there exist 
\begin{itemize}
    \item $(a, \rho) \in \Jord(\s)$ such that $a\_$ is defined and $\epsilon_{\s}((a\_, \rho), (a, \rho)) = 1$,
    \item a discrete series $\s'$ such that $\s$ is a subrepresentation of $\delta([\nu^{-\frac{a\_ - 1}{2}} \rho$, $\nu^{\frac{a-1}{2}} \rho])$ $\rtimes \s'$,
\end{itemize}
and one of the following holds:
\begin{enumerate}[(1)]
    \item if there is a $b \in \Jord_{\rho}(\s')$ such that $b\_ \in \Jord_{\rho}(\s')$ is defined and satisfies $\epsilon_{\s'}((b\_, \rho), (b, \rho)) = 1$, then $a\_ > b\_$,
\item if there is a $b \in \Jord_{\rho}(\s')$ such that $b\_ \in \Jord_{\rho}(\s')$ is defined and satisfies $\epsilon_{\s'}((b\_, \rho), (b, \rho)) = 1$, then $a < b$.
\end{enumerate}
In particular, for a discrete series $\s$ there exists an ordered $n$-tuple of discrete series $(\s_1, \s_2, \ldots, \s_n)$, $\s_i \in \Irr(G_{m_i})$, such that 
\begin{itemize}
    \item $\s_1$ is strongly positive,
    \item $\s_n \cong \s$,
    \item for every $i = 2, 3, \ldots, n$, there is $(a_i, \rho_i) \in \Jord(\s)$ such that $(a_i)\_$ is defined, $\sigma_i$ is a subrepresentation of $\d([\nu^{-\frac{(a_i)\_ -1}{2}} \rho_i, \nu^{\frac{a_i -1}{2}} \rho_i]) \rt \s_{i-1}$,
\end{itemize}
 and one of the following holds:
\begin{enumerate}[(1)]
    \item $a_i > a_j$ for all $j > i$ such that $\rho_i \cong \rho_j,$
    \item $a_i < a_j$ for all $j > i$ such that $\rho_i \cong \rho_j$.
\end{enumerate}
\end{thrm}
\begin{proof}
\begin{sloppypar}
For a discrete series $\s$ which is not strongly positive, there is an ordered pair $(a, \rho) \in \Jord(\s)$ such that $a\_$ is defined and $\epsilon_{\s}((a\_, \rho), (a, \rho)) = 1$. By Lemma \ref{epsilonequality}, there is a discrete series $\s''$ such that $\s$ is a subrepresentation of $\d([\nu^{-\frac{a\_-1}{2}} \rho, \nu^{\frac{a-1}{2}} \rho]) \rt \s''$.
\end{sloppypar}

For $(b, \rho') \in \Jord(\s)$, $\rho' \not\cong \rho$, such that $b\_$ is defined, it follows directly from Lemma \ref{lemaequiv} that $\epsilon_{\s}((b\_, \rho'), (b, \rho')) = 1$ if and only if $\epsilon_{\s''}((b\_, \rho'), (b, \rho')) = 1$. In the same way one can see that, for $(b, \rho) \in \Jord(\s)$ such that $b\_$ is defined and either $b < a\_$ or $b\_ > a$, $\epsilon_{\s}((b\_, \rho), (b, \rho)) = 1$ if and only if $\epsilon_{\s''}((b\_, \rho), (b, \rho)) = 1$.

Now let $\rho$ denote an irreducible cuspidal essentially self-dual unitarizable representation such that $\Jord_{\rho}(\s) \neq \emptyset$ and there is an $a \in \Jord_{\rho}(\s)$ such that $a\_$ is defined and $\epsilon_{\s}((a\_, \rho),(a, \rho)) = 1$. Let $S_1$ denote the set of all $b \in \Jord_{\rho}(\s)$ such that $b\_$ is defined and $\epsilon_{\s}((b\_, \rho),(b, \rho)) = 1$. Taking either $a = \min(S_1)$ or $a=\max(S_1)$, the rest of the proof follows from an inductive application of this procedure, together with Propositions \ref{propjord} and \ref{propsp}.
\end{proof}

Suppose that $\sigma \in R(G)$ is a strongly positive discrete series, with the partial cuspidal support $\sigma_{cusp}$, and let $\rho \in R(GL)$ denote an irreducible cuspidal unitary representation such that some twists of $\rho$ appear in the cuspidal support of $\sigma$. By the classification of strongly positive discrete series \cite[Theorem~A]{Kim1}, $\rho$ is self-dual, there exist unique positive half-integers $a$ and $b$, and the unique strongly positive discrete series representation $\sigma'$ without $\nu^{a} \rho$ in the cuspidal support, such that $\sigma$ is the unique irreducible subrepresentation of $\delta([\nu^{a} \rho, \nu^{b} \rho]) \rtimes \sigma'$. Furthermore, there is a
non-negative integer $k$ such that $a + k = s$, for $s > 0$
such that $\nu^{s} \rho \rtimes \sigma_{cusp}$ reduces. If $k = 0$,
there are no twists of $\rho$ appearing in the cuspidal support of
$\sigma'$ and if $k > 0$ there exist unique $b' > b$ and the
unique strongly positive discrete series $\sigma''$, which
contains neither $\nu^{a} \rho$ nor $\nu^{a+1} \rho$ in its
cuspidal support, such that $\sigma'$ can be written as the unique
irreducible subrepresentation of $\delta([\nu^{a+1} \rho, \nu^{b'}
\rho]) \rtimes \sigma''$.

Thus, we obtain that for a strongly positive discrete series $\sigma \in R(G)$ there is an ordered $n$-tuple of discrete series $(\sigma_1, \sigma_2, \ldots, \sigma_n)$, $\sigma_i \in R(G)$ for $i=1,2, \ldots, n$, such that 
\begin{itemize}
    \item $\sigma \cong \sigma_n$,
    \item $\sigma_1$ is cuspidal,
    \item for every $i = 2, 3, \ldots, n$, there is an irreducible self-dual cuspidal $\rho_i \in R(GL)$ and positive $a_i, b_i$, such that $\sigma_i \hookrightarrow \delta([\nu^{a_i} \rho_i, \nu^{b_i} \rho_i]) \rtimes \sigma_{i-1}$.
\end{itemize}
Note that $\sigma_1$ is the partial cuspidal support of $\sigma$. 

\begin{remark} \label{remarkone}
Using the previous inductive description of the strongly positive discrete series, Proposition \ref{propjord} and Theorem \ref{embed} enable us to relate the Jordan block of a discrete series $\sigma$ with the Jordan block of its partial cuspidal support. Now the Basic Assumption implies that $\Jord_{\rho}(\sigma)$ is finite.
\end{remark}

We recall a result which can be obtained following the same lines as in \cite[Section~4]{MT1}. We note that the results obtained in \cite[Section~4]{MT1} completely rely on the structural formula and the definition of the Jordan block of a discrete series, so can be directly applied to our situation.

\begin{lemma} \label{comser}
Let $\s \in \Irr(G_n)$ denote a discrete series and let $\rho$ be an irreducible essentially self-dual cuspidal unitarizable representation of some $GL_{n_{\rho}}$. Also, let $a, b$ denote positive integers, $a < b$, such that for $x \in \Jord_{\rho}(\s)$ we have $\frac{x - a}{2}, \frac{x- b}{2} \in \mathbb{Z}$ and $x \not\in \{ a, a+1, \ldots, b \}$. Then the induced representation
\begin{equation*}
\delta([\nu^{- \frac{a - 1}{2}} \rho, \nu^{\frac{b - 1}{2}} \rho]) \rtimes \s
\end{equation*}
contains two irreducible subrepresentations, which are mutually non-isomorphic.
\end{lemma}

\begin{lemma} \label{lemapripremazadnja}
Let $\s \in \Irr(G_n)$ denote a discrete series. Let $(a, \rho) \in \Jord(\s)$ be such that $a\_$ is defined and $a\_ \leq a - 4$. Then for every $x$ such that $\frac{a - x}{2}$ is an integer and $a\_ + 4 \leq x \leq a$, there exists a discrete series $\pi$ such that $\s$ is a subrepresentation of $\d([\nu^{\frac{x-1}{2}} \rho, \nu^{\frac{a-1}{2}} \rho]) \rt \pi$. Furthermore, if an irreducible constituent of the form $\d([\nu^{\frac{x-1}{2}} \rho, \nu^{\frac{a-1}{2}} \rho]) \ot \pi'$ appears in $\mu^{\ast}(\s)$, then $\pi' \cong \pi$ and $\d([\nu^{\frac{x-1}{2}} \rho, \nu^{\frac{a-1}{2}} \rho]) \ot \pi'$ appears in $\mu^{\ast}(\s)$ with multiplicity one.
\end{lemma}
\begin{proof}
We divide the proof in a series of claims.
\begin{itemize}
    \item Claim $1$: There is an irreducible representation $\pi$ such that $\s$ is a subrepresentation of $\d([\nu^{\frac{x-1}{2}} \rho, \nu^{\frac{a-1}{2}} \rho]) \rt \pi$.
\end{itemize}
Similarly as in Theorem \ref{embed}, letting $\s_{k+1} = \s$, there are irreducible essentially square-integrable representations $\d_1, \ldots, \d_k$, $\d_i = \d([\nu^{-x_i} \rho_i, \nu^{y_i} \rho_i]) \in R(GL)$, $x_i, y_i \geq 0$, for $i = 1, \ldots, k$, and discrete series $\s_1, \ldots, \s_k \in R(G)$, such that $\s_{i+1} \hookrightarrow \d_{i} \rt \s_{i}$ for $i = 1, \ldots, k$, $\s_1$ is strongly positive, in $\Jord_{\rho_k}(\s)$ we have $(2y_k + 1)\_ = 2x_k + 1$, and in $\Jord_{\rho_{i+1}}(\s_i)$ we have $(2y_{i} + 1)\_ = 2x_{i} + 1$.

Obviously, $\s$ is a subrepresentation of $\d_k \times \d_{k-1} \times \cdots \times \d_1 \rtimes \s_1$, and either $(a, \rho) \in \Jord(\s_1)$, or there is a unique $i \in \{ 1, \ldots, k \}$ such that $\rho_i \cong \rho$ and $a \in \{ 2x_i + 1, 2y_i + 1 \}$ since $(a, \rho) \in \Jord(\s)$.

We consider several possibilities.
\begin{enumerate}[(i)]
    \item If $(a, \rho) \in \Jord(\s_1)$, it follows from the classification of strongly positive discrete series \cite[Theorem~A]{Kim1} and \cite[Theorem~5.3]{Matic4} that there is an irreducible strongly positive representation $\s'$ such that $\s_1 \hookrightarrow \d([\nu^{\frac{x-1}{2}} \rho, \nu^{\frac{a-1}{2}} \rho]) \rt \s'$. Since $2x_i +1, 2y_i + 1 \in \Jord_{\rho_i}(\s)$, for $i = 1, \ldots, k$ we have
\begin{equation*}
\d_i \times \d([\nu^{\frac{x-1}{2}} \rho, \nu^{\frac{a-1}{2}} \rho]) \cong \d([\nu^{\frac{x-1}{2}} \rho, \nu^{\frac{a-1}{2}} \rho]) \times \d_i,
\end{equation*}
and it follows that $\s$ is a subrepresentation of
\begin{equation*}
\d([\nu^{\frac{x-1}{2}} \rho, \nu^{\frac{a-1}{2}} \rho]) \times \d_k \times \cdots \times \d_1 \rtimes \s'.
\end{equation*}
Now Lemma \ref{lemajantz} implies the Claim $1$ in this case.
   \item If $\rho_i \cong \rho$ and $a = 2y_i + 1$, for some $i \in \{ 1, 2, \ldots, k \}$, we have an embedding $\d_i \hookrightarrow \d([\nu^{\frac{x-1}{2}} \rho, \nu^{\frac{a-1}{2}} \rho]) \times \d([\nu^{-x_i} \rho, \nu^{\frac{x-3}{2}} \rho])$ and, since for $j = i+1, \ldots, k$ we have $\d_j \times \d([\nu^{\frac{x-1}{2}} \rho, \nu^{\frac{a-1}{2}} \rho]) \cong \d([\nu^{\frac{x-1}{2}} \rho, \nu^{\frac{a-1}{2}} \rho]) \times \d_j$, we obtain that $\s$ is a subrepresentation of
\begin{equation*}
\d([\nu^{\frac{x-1}{2}} \rho, \nu^{\frac{a-1}{2}} \rho]) \times \d_k \times \cdots \times \d_{i+1} \times \d([\nu^{-x_i} \rho, \nu^{\frac{x-3}{2}} \rho]) \times \d_{i-1} \times \cdots \times \d_1 \rtimes \s_1.
\end{equation*}
Again, Lemma \ref{lemajantz} implies the Claim $1$.
   \item Suppose that $\rho_i \cong \rho$ and $a = 2x_i + 1$, for some $i \in \{ 1, 2, \ldots, k \}$. Obviously, $(a, \rho) \notin \Jord(\sigma_1)$. Similarly as in the previous case we obtain that $\s$ is an irreducible subrepresentation of
\begin{equation*}
\d_k \times \cdots \times \d_{i+1} \times \d([\nu^{- \frac{x-3}{2}} \rho, \nu^{y_i} \rho]) \times \d_{i-1} \times \cdots \times \d_1 \times \d([\nu^{-\frac{a-1}{2}} \rho, \nu^{-\frac{x-1}{2}} \rho]) \rtimes \s_1.
\end{equation*}

Let us prove that the induced representation $\d([\nu^{-\frac{a-1}{2}} \rho, \nu^{-\frac{x-1}{2}} \rho]) \rtimes \s_1$ is irreducible. Following the same lines as in the proof of \cite[Theorem~3.4]{Kim1}, using the self-duality of $\rho$, we obtain that in $R(G)$ holds
\begin{equation*}
\d([\nu^{-\frac{a-1}{2}} \rho, \nu^{-\frac{x-1}{2}} \rho]) \rtimes \s_1 = \d([\nu^{\frac{x-1}{2}} \rho, \nu^{\frac{a-1}{2}} \rho]) \rtimes \s_1.
\end{equation*}
Since $x > 1$, from the cuspidal support of the strongly positive representation, which we have described before Remark \ref{remarkone}, follows that an irreducible tempered subquotient of $\d([\nu^{-\frac{a-1}{2}} \rho$, $\nu^{-\frac{x-1}{2}} \rho]) \rtimes \s_1$ would have to be strongly positive since $\mu^{\ast}(\d([\nu^{-\frac{a-1}{2}} \rho$, $\nu^{-\frac{x-1}{2}} \rho]) \rtimes \s_1)$ does not contain an irreducible constituent of the form $\d([\nu^{y_1} \rho, \nu^{y_2} \rho]) \otimes \pi$ for $y_1 \leq 0$ and $y_1 + y_2 \geq 0$. But, if $\d([\nu^{-\frac{a-1}{2}} \rho, \nu^{-\frac{x-1}{2}} \rho]) \rtimes \s_1$ contains a strongly positive discrete series, then it can be directly seen from the cuspidal support of such a representation that $x - 2 \in \Jord_{\rho}(\sigma_1)$, a contradiction. Thus, 
$\d([\nu^{-\frac{a-1}{2}} \rho, \nu^{-\frac{x-1}{2}} \rho]) \rtimes \s_1$ does not contain an irreducible tempered subquotient. 

We write a non-tempered irreducible subquotient of $\d([\nu^{\frac{x-1}{2}} \rho, \nu^{\frac{a-1}{2}} \rho]) \rtimes \s_1$ in the form $L(\d'_1, \ldots, \d'_l, \tau)$, where $\d'_1, \ldots, \d'_l$ are irreducible essentially square-integrable representations, $\d'_j \in \Irr(GL_{n'_j})$, such that $e(\d'_1) \leq \cdots \leq e(\d'_{l}) < 0$, and $\tau \in \Irr(G_{n''})$ is an irreducible tempered representation. Using Frobenius reciprocity, together with the transitivity of Jacquet modules, we deduce that $\mu^{\ast}(\d([\nu^{\frac{x-1}{2}} \rho, \nu^{\frac{a-1}{2}} \rho]) \rtimes \s_1) \geq \d'_1 \otimes \s'$ for some irreducible representation $\s'$ such that the Jacquet module of $\s'$ with respect to the appropriate parabolic subgroup
contains $\d'_2 \otimes \cdots \otimes \d'_l \otimes \tau$.

The structural formula implies that there are $\frac{x-3}{2} \leq i_1 \leq j_1 \leq \frac{a-1}{2}$ and an irreducible constituent $\pi_1 \otimes \s''$ of $\mu^{\ast}(\s_1)$ such that
\begin{align*}
\d'_1 & \leq \d([\nu^{- i_1} \rho, \nu^{-\frac{x-1}{2}} \rho]) \times \d([\nu^{j_1 + 1} \rho, \nu^{\frac{a-1}{2}} \rho]) \times \pi_1 \\
\intertext{and}
\s' & \leq \d([\nu^{i_1 + 1} \rho, \nu^{j_1} \rho]) \rtimes \s''.
\end{align*}
Since $e(\d'_1) < 0$ and $\s_1$ is strongly positive, it follows at once that $\d'_1 \cong \d([\nu^{- i_1} \rho, \nu^{-\frac{x-1}{2}} \rho])$ and $\s' \leq \d([\nu^{i_1 + 1} \rho, \nu^{\frac{a-1}{2}} \rho]) \rtimes \s_1$.

\begin{sloppypar}
If $l \geq 2$, in the same way we obtain that $\d'_2 \cong \d([\nu^{- i_2} \rho, \nu^{-i_1 - 1} \rho])$ for some $i_2 \geq i_1 + 1$, which is impossible since $e(\d'_1) \leq e(\d'_2)$. Thus, $l = 1$ and $\d([\nu^{i_1 + 1} \rho, \nu^{\frac{a-1}{2}} \rho]) \rtimes \s_1$ contains a tempered subquotient $\sigma'=\tau$. It can be easily seen that this happens only if $i_1 = \frac{a-1}{2}$, so every irreducible constituent of $\d([\nu^{\frac{x-1}{2}} \rho, \nu^{\frac{a-1}{2}} \rho]) \rtimes \s_1$ is isomorphic to its Langlands quotient, which appears with multiplicity one. Consequently, the induced representation $\d([\nu^{-\frac{a-1}{2}} \rho, \nu^{-\frac{x-1}{2}} \rho]) \rtimes \s_1$ is irreducible and isomorphic to $\d([\nu^{\frac{x-1}{2}} \rho, \nu^{\frac{a-1}{2}} \rho]) \rtimes \s_1$.
\end{sloppypar}

Since $\d_j \times \d([\nu^{\frac{x-1}{2}} \rho, \nu^{\frac{a-1}{2}} \rho]) \cong \d([\nu^{\frac{x-1}{2}} \rho, \nu^{\frac{a-1}{2}} \rho]) \times \d_j$ for $j = 1, \ldots, i-1, i+1, \ldots, k$ and 
\begin{equation*}
\d([\nu^{-\frac{x-3}{2}} \rho, \nu^{y_i} \rho]) \times \d([\nu^{\frac{x-1}{2}} \rho, \nu^{\frac{a-1}{2}} \rho]) \cong \d([\nu^{\frac{x-1}{2}} \rho, \nu^{\frac{a-1}{2}} \rho]) \times \d([\nu^{-\frac{x-3}{2}} \rho, \nu^{y_i} \rho]),     
\end{equation*}
we obtain that $\s$ is an irreducible subrepresentation of
\begin{equation*}
\d([\nu^{\frac{x-1}{2}} \rho, \nu^{\frac{a-1}{2}} \rho])  \times \d_k \times \cdots \times \d_{i+1} \times \d([\nu^{- \frac{x-3}{2}} \rho, \nu^{y_i} \rho]) \times \d_{i-1} \times \cdots \times \d_1 \rtimes \s_1.
\end{equation*}

Now Lemma \ref{lemajantz} implies that there is an irreducible representation $\pi \in R(G)$ such that $\s$ is a subrepresentation of $\d([\nu^{\frac{x-1}{2}} \rho, \nu^{\frac{a-1}{2}} \rho])  \rtimes \pi$.

\end{enumerate}

\begin{itemize}
    \item Claim $2$: The representation $\pi$ is a discrete series. 
\end{itemize}
Suppose, on the contrary, that $\pi$ is not a discrete series. By Corollary \ref{notds}, there are $x', y'$, $x' - y' \in \mathbb{Z}$ and $x' + y' \leq 0$, an irreducible cuspidal  representation $\rho' \in R(GL)$, and an irreducible representation $\pi_1 \in R(G)$, such that $\pi$ is a subrepresentation of
\begin{equation*}
\delta([\nu^{x'} \rho', \nu^{y'} \rho' ]) \rtimes \pi_1.
\end{equation*}
\begin{sloppypar}
\noindent Therefore,  a discrete series $\s$ is a subrepresentation of $\delta([\nu^{\frac{x-1}{2}} \rho, \nu^{\frac{a-1}{2}} \rho ]) \times \delta([\nu^{x'} \rho', \nu^{y'} \rho' ]) \rtimes \pi_1.$ In the same way as in the proof of Lemma \ref{priprema1}, we see that $\rho' \cong \rho$, $y' = \frac{x-3}{2}$, and $\s$ is a subrepresentation of $\d([\nu^{x'} \rho, \nu^{\frac{a-1}{2}} \rho]) \rtimes \pi_1$. Since $x' \leq -\frac{x-3}{2} < - \frac{a\_ - 1}{2}$, this contradicts Corollary \ref{corembedjord}.
\end{sloppypar}

\begin{itemize}
    \item Claim $3$: If some irreducible constituent of the form $\d([\nu^{\frac{x-1}{2}} \rho$, $\nu^{\frac{a-1}{2}} \rho]) \ot \pi'$ appears in $\mu^{\ast}(\s)$, then $\pi' \cong \pi$ and $\d([\nu^{\frac{x-1}{2}} \rho, \nu^{\frac{a-1}{2}} \rho]) \ot \pi'$ appears in $\mu^{\ast}(\s)$ with multiplicity one.
\end{itemize}
Since $\s$ is a subrepresentation of $\d([\nu^{\frac{x-1}{2}} \rho, \nu^{\frac{a-1}{2}} \rho]) \rtimes \pi$, it follows from Proposition \ref{propjord} that $\Jord_{\rho}(\s) = \Jord_{\rho}(\pi) \cup \{ a \} \setminus \{ x - 2 \}$ and, consequently, $\Jord_{\rho}(\pi) \cap [ x-1, a ] \subset \Jord_{\rho}(\sigma) \cap [ x-1, a ] \setminus \{ a \} = \emptyset$. 

Let us determine all irreducible constituents of $\mu^{\ast}(\d([\nu^{\frac{x-1}{2}} \rho, \nu^{\frac{a-1}{2}} \rho]) \rtimes \pi)$ of the form $\d([\nu^{\frac{x-1}{2}} \rho, \nu^{\frac{a-1}{2}} \rho]) \ot \pi'$. By the structural formula, there are $i, j$, $\frac{x-1}{2} \leq i \leq j \leq \frac{a-1}{2}$ and an irreducible constituent $\delta_1 \otimes \pi_1$ of $\mu^{\ast}(\pi)$ such that 
\begin{align*}
\d([\nu^{\frac{x-1}{2}} \rho, \nu^{\frac{a-1}{2}} \rho])& \leq \delta([\nu^{-i} \rho, \nu^{-\frac{x - 1}{2}} \rho]) \times \delta([\nu^{j+1} \rho, \nu^{\frac{a-1}{2}} \rho]) \times \delta_1 \\
\intertext{and}
\pi' & \leq \delta([\nu^{i+1} \rho, \nu^{j} \rho]) \rtimes \pi_1.
\end{align*}
\begin{sloppypar}
\noindent Considering cuspidal supports, we have $i = \frac{x-3}{2}$ and $\delta_1 = \delta([\nu^{\frac{x-1}{2}}\rho, \nu^{j}\rho])$. From Lemma \ref{remarkprva} and the description of $\Jord_{\rho}(\pi)$ we obtain that $\pi_1 \cong \pi$, so $j = \frac{x-3}{2}$. Thus, $\d([\nu^{\frac{x-1}{2}} \rho, \nu^{\frac{a-1}{2}} \rho]) \otimes \pi$ is a unique irreducible constituent of $\mu^{\ast}(\d([\nu^{\frac{x-1}{2}} \rho, \nu^{\frac{a-1}{2}} \rho]) \rtimes \pi)$ of the form $\d([\nu^{\frac{x-1}{2}} \rho, \nu^{\frac{a-1}{2}} \rho]) \ot \pi'$, and it appears there with multiplicity one. This finishes the proof.
\end{sloppypar}
\end{proof}

In a similar way as in the proof of the previous lemma, we also obtain:

\begin{lemma} \label{lemapripremazadnjajos}
Let $\s \in \Irr(G_n)$ denote a discrete series. Let $(a, \rho) \in \Jord(\s)$ such that $a\_$ is not defined. Also, suppose that $\mu^{\ast}(\s)$ contains an irreducible constituent of the form $\nu^{\frac{a-1}{2}} \rho \otimes \s'$. Then there exists a discrete series $\pi$ such that $\s$ is a subrepresentation of
\begin{equation*}
\d([\nu^{\frac{a-1}{2} - \lfloor \frac{a-1}{2} \rfloor + 1} \rho, \nu^{\frac{a-1}{2}} \rho]) \rt \pi,
\end{equation*}
where $\lfloor \frac{a-1}{2} \rfloor$ stands for the largest integer which is not greater than $\frac{a-1}{2}$. Furthermore, if an irreducible constituent of the form $\d([\nu^{\frac{a-1}{2} - \lfloor \frac{a-1}{2} \rfloor + 1} \rho$, $\nu^{\frac{a-1}{2}} \rho]) \ot \pi'$ appears in $\mu^{\ast}(\s)$, then $\pi' \cong \pi$ and such a constituent appears in $\mu^{\ast}(\s)$ with multiplicity one. Also, $\mu^{\ast}(\s)$ does not contain an irreducible constituent of the form $\d([\nu^{x} \rho, \nu^{\frac{a-1}{2}} \rho]) \ot \pi$ for $x < 0$.
\end{lemma}

\begin{thrm} \label{dissub}
Let $\s \in \Irr(G_n)$ denote a non-strongly positive discrete series. Let $(a, \rho) \in \Jord(\s)$ be such that $a\_$ is defined and $\epsilon_{\s}((a\_, \rho), (a, \rho)) = 1$. Also, let $\s'$ denote a discrete series such that $\s$ is a subrepresentation of the induced representation
\begin{equation} \label{repindosn}
\delta([\nu^{-\frac{a\_ - 1}{2}} \rho, \nu^{\frac{a-1}{2}} \rho]) \rtimes \s'.
\end{equation}
Then the induced representation $(\ref{repindosn})$ contains exactly two irreducible subrepresentations. Moreover, these representations are square-integrable and mutually non-isomorphic. 
\end{thrm}
\begin{proof}
\begin{sloppypar}
By Lemmas \ref{epsilonequality} and \ref{comser}, the induced representation (\ref{repindosn}) contains exactly two irreducible subrepresentations which are mutually non-isomorphic. Let us show the square-integrability of the irreducible subrepresentations of (\ref{repindosn}). To achieve this, we follow an approach introduced in \cite{Mu3}. First we prove that there are no irreducible tempered subquotients of the induced representation (\ref{repindosn}) which are not square-integrable. On the contrary, suppose that there is some irreducible tempered but not square-integrable representation $\tau$ such that $\tau \leq \delta([\nu^{-\frac{a\_ - 1}{2}} \rho, \nu^{\frac{a-1}{2}} \rho]) \rtimes \s'$. Then there is a cuspidal unitarizable representation $\rho' \in \Irr(GL)$, a non-negative integer $b$ and a tempered representation $\tau' \in \Irr(G_{n'})$ such that
\end{sloppypar}
\begin{equation*}
\tau \hookrightarrow \delta([\nu^{-b} \rho', \nu^{b} \rho' ]) \rtimes \tau'.
\end{equation*}
Frobenius reciprocity shows that $\mu^{\ast}(\tau) \geq \delta([\nu^{-b} \rho', \nu^{b} \rho' ]) \otimes \tau'$. Thus,
\begin{equation*}
\mu^{\ast}(\delta([\nu^{-\frac{a\_ - 1}{2}} \rho, \nu^{\frac{a-1}{2}} \rho]) \rtimes \s') \geq \delta([\nu^{-b} \rho', \nu^{b} \rho' ]) \otimes \tau'.
\end{equation*}
Since $\s'$ is a discrete series, using the structural formula and Lemma \ref{remarkprva} one readily sees that $\rho' \cong \rho$ and $b = \frac{a\_ - 1}{2}$. Also, $\tau'$ is an irreducible subquotient of $\delta([\nu^{\frac{a\_ + 1}{2}} \rho, \nu^{\frac{a-1}{2}} \rho]) \rt \s'$. Similarly as in the previous sentence, it can be easily seen that $\mu^{\ast}(\delta([\nu^{\frac{a\_ + 1}{2}} \rho, \nu^{\frac{a-1}{2}} \rho]) \rt \s')$ does not contain an irreducible constituent of the form $\delta([\nu^{-b'} \rho'', \nu^{b'} \rho'' ]) \otimes \pi$, so $\tau'$ has to be a discrete series.

Using Theorem \ref{embed} we deduce that there are $(a_1, \rho_1)$, $(b_1, \rho_1)$, $(a_2, \rho_2)$, $(b_2, \rho_2), \ldots, (a_k, \rho_k)$, $(b_k, \rho_k) \in \Jord(\s')$ and a strongly positive discrete series $\s_1 \in R(G)$ such that the cuspidal support of $\sigma'$ equals
\begin{equation*}
    \bigcup_{i=1}^{k}[-\nu^{\frac{a_i - 1}{2}} \rho_i, \nu^{\frac{b_i - 1}{2}} \rho_i] \cup [\sigma_1],
\end{equation*}
where $[\sigma_1]$ stands for the cuspidal support of $\sigma_1$. Since an analogous result holds for a discrete series subquotient $\tau'$ of $\delta([\nu^{\frac{a\_ + 1}{2}} \rho, \nu^{\frac{a-1}{2}} \rho]) \rt \s'$, we deduce that $a\_ \in \Jord_{\rho}(\sigma')$. On the other hand, Proposition \ref{propjord} implies that $\Jord(\sigma')$ does not contain $(a\_, \rho)$ since $\s \hookrightarrow \delta([\nu^{-\frac{a\_ + 1}{2}} \rho, \nu^{\frac{a-1}{2}} \rho]) \rt \s'$. Therefore the induced representation $\delta([\nu^{-\frac{a\_ + 1}{2}} \rho, \nu^{\frac{a-1}{2}} \rho]) \rt \s'$ does not contain an irreducible tempered subquotient.

Let us now prove that the only irreducible non-tempered subquotient of the induced representation (\ref{repindosn}) is its Langlands quotient. Let us denote an irreducible non-tempered subquotient of (\ref{repindosn}) by $\pi$ and write $\pi \cong L(\d_1, \ldots, \d_k, \tau)$, where $\d_1, \ldots, \d_k$ are irreducible essentially square-integrable representations of some $GL_{n_1}, \ldots$, $GL_{n_k}$ such that $e(\d_1) \leq \cdots \leq e(\d_k) < 0$, and $\tau \in \Irr(G_{n'})$ is a tempered representation. Write $\d_i = \d([ \nu^{a_i} \rho_i, \nu^{b_i} \rho_i])$.

Using Frobenius reciprocity and the transitivity of Jacquet modules, we see that
\begin{equation*}
\mu^{\ast}(\delta([\nu^{-\frac{a\_ - 1}{2}} \rho, \nu^{\frac{a-1}{2}} \rho]) \rtimes \s') \geq \delta_1 \otimes \pi',
\end{equation*}
for some irreducible representation $\pi'$ such that its Jacquet module with respect to the appropriate parabolic subgroup contains $\d_2 \otimes \cdots \otimes \d_k \otimes \tau$.

It follows from the structural formula that there are $i, j$ such that $-\frac{a\_ + 1}{2} \leq i \leq j \leq \frac{a-1}{2}$ and an irreducible constituent $\d \otimes \pi''$ of $\mu^{\ast}(\s')$ such that
\begin{align*}
\d([ \nu^{a_1} \rho_1, \nu^{b_1} \rho_1]) (= \delta_1 )& \leq \delta([\nu^{-i} \rho, \nu^{\frac{a\_ - 1}{2}} \rho]) \times \delta([\nu^{j+1} \rho, \nu^{\frac{a-1}{2}} \rho]) \times \d \\
\intertext{and}
\pi' & \leq \delta([\nu^{i+1} \rho, \nu^{j} \rho]) \rtimes \pi''.
\end{align*}
Since $e(\d_1) = \frac{a_1 + b_1}{2} < 0$ and $\s'$ is square-integrable, it follows that $\rho_1 \cong \rho$. From the description of $\Jord_{\rho}(\s')$ and Lemma \ref{remarkprva}, we obtain that $m^{\ast}(\d)$ does not contain an irreducible constituent of the form $\nu^{x} \rho \otimes \d'$ for $\frac{a\_ + 1}{2} \leq x \leq \frac{a - 1}{2}$. Thus, $j = \frac{a-1}{2}$, $a_1 = -i$, $i > \frac{a\_ - 1}{2}$, $b_1 = \frac{a\_-1}{2}$, and
\begin{equation*}
\pi' \leq \delta([\nu^{i+1} \rho, \nu^{\frac{a-1}{2}} \rho]) \rtimes \s'.
\end{equation*}
Suppose that $i < \frac{a-1}{2}$. Since $i > \frac{a\_ - 1}{2}$, in the first part of the proof we have seen that the induced representation $\delta([\nu^{i+1} \rho, \nu^{\frac{a-1}{2}} \rho]) \rtimes \s'$ does not contain a tempered subquotient, so $k \geq 2$. Now, repeating the same procedure as for $\delta_1$, we deduce that $\d_2$ is of the form $\d([\nu^{i'} \rho, \nu^{-i-1} \rho])$, for some $i' \leq -i-1$, which is impossible since $e(\d_1) \leq e(\d_2)$. Thus, $i = \frac{a-1}{2}$ and the only non-tempered subquotient of the induced representation (\ref{repindosn}) is its Langlands quotient, which appears in the composition series of (\ref{repindosn}) with multiplicity one. This proves the theorem.
\end{proof}

Now we prove a result which happens to be crucial for our classification.

\begin{thrm} \label{prepairemain}
Let $\s \in \Irr(G_{n})$ denote a non-strongly positive discrete series. Let $(a, \rho) \in \Jord(\s)$ be such that $a\_$ is defined and $\epsilon_{\s}((a\_, \rho), (a, \rho)) = 1$. Also, let $\s'$ denote a discrete series such that
\begin{equation*}
\s \hookrightarrow \delta([\nu^{-\frac{a\_ - 1}{2}} \rho, \nu^{\frac{a-1}{2}} \rho]) \rtimes \s'.
\end{equation*}
Suppose that $(b, \rho') \in \Jord(\s')$ is such that $b\_$ is defined. If $\rho' \not\cong \rho$, we have $\epsilon_{\s'}((b\_, \rho'), (b, \rho')) = \epsilon_{\s}((b\_, \rho'), (b, \rho'))$. Also, if either $b < a\_$ or $b\_ > a$, then $\epsilon_{\s'}((b\_, \rho), (b, \rho)) = \epsilon_{\s}((b\_, \rho), (b, \rho))$. If $b\_ < a\_$ and $a < b$, we have
\begin{equation*}
\epsilon_{\s'}((b\_, \rho), (b, \rho)) = \epsilon_{\s}((b\_, \rho), (a\_, \rho)) \cdot \epsilon_{\s}((a, \rho), (b, \rho)).
\end{equation*}
\end{thrm}
\begin{proof}
It follows from Lemma \ref{lemaequiv} and the structural formula that for $(b, \rho') \in \Jord(\s)$, $\rho' \not\cong \rho$, such that $b\_$ is defined, we have $\epsilon_{\s}((b\_, \rho'), (b, \rho')) = 1$ if and only if $\epsilon_{\s'}((b\_, \rho'), (b, \rho')) = 1$. Also, for $(b, \rho) \in \Jord(\s)$ such that $b\_$ is defined and either $b < a\_$ or $b\_ > a$, in the same way one can easily see that $\epsilon_{\s}((b\_, \rho), (b, \rho)) = 1$ if and only if $\epsilon_{\s'}((b\_, \rho), (b, \rho)) = 1$.

Let us describe the remaining case which is the most non-trivial case.

Let us change the notation for simplicity. In the remainder of the proof, underbar always means underbar in $\Jord(\s)$, not in  $\Jord(\s')$.
Let $(b, \rho) \in \Jord(\s)$ be such that $b\_$ is defined. Also, assume that $a \in \Jord_{\rho}(\s)$ is such that $a = (b\_)\_$ and $\epsilon_{\s}((a, \rho), (b\_, \rho)) = 1$. Suppose that $a\_$ is defined and let $\s'$ denote a discrete series such that $\s$ is a subrepresentation of
\begin{equation*}
\delta([\nu^{-\frac{a - 1}{2}} \rho, \nu^{\frac{b\_ - 1}{2}} \rho]) \rtimes \s'.
\end{equation*}

Note that we are in the case $a\_ < a< b\_ < b$. 

Then $a\_, b \in \Jord_{\rho}(\s')$ and $\Jord_{\rho}(\s') \cap [ a\_ + 1, b - 1] = \emptyset$.

Following exactly the same lines as in \cite[Section~4]{MT1}, which uses only the standard calculations of the Jacquet modules by the structural formula, we deduce that in $R(G)$ we have
\begin{equation*}
\delta([\nu^{-\frac{a - 1}{2}} \rho, \nu^{\frac{a - 1}{2}} \rho]) \rtimes \s' = \tau_1 + \tau_2,
\end{equation*}
for mutually non-isomorphic tempered representations $\tau_1$ and $\tau_2$. 

We split the rest of the proof in six claims.

\begin{itemize}
    \item Claim $1$: There exists a unique $\alpha \in \{ 1, 2 \}$ such that $\sigma$ is a subrepresentation of $\delta([\nu^{\frac{a + 1}{2}} \rho, \nu^{\frac{b\_ - 1}{2}} \rho]) \rtimes \tau_{\alpha}$.
\end{itemize}
Let us prove the Claim $1$. By Lemma \ref{epsilonequality}, there is an irreducible tempered subrepresentation $\tau$ of
\begin{equation*}
\delta([\nu^{-\frac{a - 1}{2}} \rho, \nu^{\frac{a - 1}{2}} \rho]) \rtimes \s'.
\end{equation*}
such that $\s$ is a subrepresentation of 
\begin{equation*}
\delta([\nu^{\frac{a + 1}{2}} \rho, \nu^{\frac{b\_ - 1}{2}} \rho]) \rtimes \tau.
\end{equation*}
By the Frobenius reciprocity, $\mu^{\ast}(\sigma)$ contains $\delta([\nu^{\frac{a + 1}{2}} \rho, \nu^{\frac{b\_ - 1}{2}} \rho]) \otimes \tau$. 

Theorem \ref{dissub} implies that there is another discrete series subrepresentation of $\delta([\nu^{-\frac{a - 1}{2}} \rho, \nu^{\frac{b\_ - 1}{2}} \rho]) \rtimes \s'$, which we denote by $\sigma_1$. In the same way as for $\sigma$ we conclude that there is an irreducible tempered representation $\tau'$ such that $\mu^{\ast}(\sigma_1)$ contains $\delta([\nu^{\frac{a + 1}{2}} \rho, \nu^{\frac{b\_ - 1}{2}} \rho]) \otimes \tau'$.

It can be seen in the same way as in the last part of the proof of Lemma \ref{lemapripremazadnja} that $\delta([\nu^{\frac{a + 1}{2}} \rho, \nu^{\frac{b\_ - 1}{2}} \rho]) \otimes \tau_{1}$ and $\delta([\nu^{\frac{a + 1}{2}} \rho, \nu^{\frac{b\_ - 1}{2}} \rho]) \otimes \tau_{2}$ are the only irreducible constituents of the form $\delta([\nu^{\frac{a + 1}{2}} \rho, \nu^{\frac{b\_ - 1}{2}} \rho]) \otimes \pi$ appearing in $\mu^{\ast}(\delta([\nu^{-\frac{a - 1}{2}} \rho, \nu^{\frac{b\_ - 1}{2}} \rho]) \rtimes \s')$, and both of them appear with multiplicity one. Thus, there is a unique $\alpha \in \{ 1, 2 \}$ such that $\tau \cong \tau_{\alpha}$.

\begin{itemize}
    \item Claim $2$: There exists a unique $\beta \in \{ 1, 2 \}$ such that $\mu^{\ast}(\tau_{\beta})$ contains 
    \begin{equation} \label{jmdrugi}
    \delta([\nu^{\frac{a\_ + 1}{2}} \rho, \nu^{\frac{a - 1}{2}} \rho]) \times \delta([\nu^{\frac{a\_ + 1}{2}} \rho, \nu^{\frac{a - 1}{2}} \rho])    \otimes \delta([\nu^{-\frac{a\_ - 1}{2}} \rho, \nu^{\frac{a\_ - 1}{2}} \rho]) \rtimes \sigma'.
    \end{equation}
    Furthermore, $\tau_{\beta} \hookrightarrow \delta([\nu^{\frac{a\_ + 1}{2}} \rho, \nu^{\frac{a - 1}{2}} \rho]) \times \delta([\nu^{\frac{a\_ + 1}{2}} \rho, \nu^{\frac{a - 1}{2}} \rho])  \times \delta([\nu^{-\frac{a\_ - 1}{2}} \rho,$ $ \nu^{\frac{a\_ - 1}{2}} \rho]) \rtimes \sigma'$.
\end{itemize}
Since $a\_ \in \Jord_{\rho}(\sigma')$, it follows from the definition of the Jordan block that $\delta([\nu^{-\frac{a\_ - 1}{2}} \rho, \nu^{\frac{a\_ - 1}{2}} \rho]) \rtimes \sigma'$ is irreducible. Now we prove the Claim $2$. Suppose that $\delta([\nu^{\frac{a\_ + 1}{2}} \rho, \nu^{\frac{a - 1}{2}} \rho]) \times \delta([\nu^{\frac{a\_ + 1}{2}} \rho$, $\nu^{\frac{a - 1}{2}} \rho]) \otimes \pi$ is an irreducible constituent appearing in $\mu^{\ast}(\delta([\nu^{-\frac{a - 1}{2}} \rho, \nu^{\frac{a - 1}{2}} \rho]) \rtimes \s')$. Then there are $i, j$ such that $-\frac{a + 1}{2} \leq i \leq j \leq \frac{a - 1}{2}$ and an irreducible constituent $\d \otimes \pi'$ of $\mu^{\ast}(\s')$ such that
\begin{align*}
\delta([\nu^{\frac{a\_ + 1}{2}} \rho, \nu^{\frac{a - 1}{2}} \rho]) \times \delta([\nu^{\frac{a\_ + 1}{2}} \rho, \nu^{\frac{a - 1}{2}} \rho]) & \leq \delta([\nu^{-i} \rho, \nu^{\frac{a - 1}{2}} \rho]) \times \delta([\nu^{j+1} \rho, \nu^{\frac{a - 1}{2}} \rho]) \times \d
\intertext{and}
\pi & \leq  \delta([\nu^{i+1} \rho, \nu^{j} \rho]) \rtimes \pi'.
\end{align*}
From the description of $\Jord_{\rho}(\s')$ and Lemma \ref{remarkprva} we obtain that $m^{\ast}(\d)$ does not contain an irreducible constituent of the form $\nu^{x} \rho \otimes \d'$ for $\frac{a\_ + 1}{2} \leq x \leq \frac{a - 1}{2}$. Thus, $i = - \frac{a\_ + 1}{2}$, $j = \frac{a\_ - 1}{2}$, and $\pi' \cong \sigma'$.

Thus, (\ref{jmdrugi}) is a unique irreducible constituent of $\mu^{\ast}(\delta([\nu^{-\frac{a - 1}{2}} \rho$, $\nu^{\frac{a - 1}{2}} \rho]) \rtimes \s')$ of the form $\delta([\nu^{\frac{a\_ + 1}{2}} \rho, \nu^{\frac{a - 1}{2}} \rho]) \times \delta([\nu^{\frac{a\_ + 1}{2}} \rho, \nu^{\frac{a - 1}{2}} \rho]) \otimes \pi$, and appears there with multiplicity one. This proves the first statement of Claim 2.

Let us also prove that $\mu^{\ast}(\tau_{\beta})$ contains (\ref{jmdrugi}) if and only if $\tau_{\beta}$ is a subrepresentation of 
\begin{equation*} 
    \delta([\nu^{\frac{a\_ + 1}{2}} \rho, \nu^{\frac{a - 1}{2}} \rho]) \times \delta([\nu^{\frac{a\_ + 1}{2}} \rho, \nu^{\frac{a - 1}{2}} \rho])    \times \delta([\nu^{-\frac{a\_ - 1}{2}} \rho, \nu^{\frac{a\_ - 1}{2}} \rho]) \rtimes \sigma'.
\end{equation*}
If $\mu^{\ast}(\tau_{\beta})$ contains (\ref{jmdrugi}), the transitivity of the Jacquet modules implies that the Jacquet module of $\tau_{\beta}$ with respect to the appropriate parabolic subgroup contains
\begin{gather*} 
    \nu^{\frac{a - 1}{2}} \rho \otimes \nu^{\frac{a - 3}{2}} \rho \otimes \cdots \otimes \nu^{\frac{a\_ + 1}{2}} \rho \otimes \\
    \nu^{\frac{a - 1}{2}} \rho \otimes \nu^{\frac{a - 3}{2}} \rho \otimes \cdots \otimes \nu^{\frac{a\_ + 1}{2}} \rho \otimes \delta([\nu^{-\frac{a\_ - 1}{2}} \rho, \nu^{\frac{a\_ - 1}{2}} \rho]) \rtimes \sigma'.
\end{gather*}
Now \cite[Corollary~6.2(3)]{Tad6}, which holds for reductive groups, implies that there is an irreducible representation $\pi_1$ such that 
$\tau_{\beta}$ is a subrepresentation of
\begin{equation*} 
    \nu^{\frac{a - 1}{2}} \rho \times \nu^{\frac{a - 3}{2}} \rho \times \cdots \times \nu^{\frac{a\_ + 1}{2}} \rho \times \nu^{\frac{a - 1}{2}} \rho \times \nu^{\frac{a - 3}{2}} \rho \times \cdots \times \nu^{\frac{a\_ + 1}{2}} \rho \rtimes \pi_1.
\end{equation*}
From Lemma \ref{lemajantz} we obtain that there is an irreducible subquotient $\pi_2$ of 
\begin{equation*} 
    \nu^{\frac{a - 1}{2}} \rho \times \nu^{\frac{a - 3}{2}} \rho \times \cdots \times \nu^{\frac{a\_ + 1}{2}} \rho \times \nu^{\frac{a - 1}{2}} \rho \times \nu^{\frac{a - 3}{2}} \rho \times \cdots \times \nu^{\frac{a\_ + 1}{2}} \rho
\end{equation*}
such that $\tau_{\beta}$ is a subrepresentation of $\pi_2 \rtimes \pi_1$. Since $\mu^{\ast}(\tau_{\beta})$ does not contain an irreducible constituent of the form $\nu^x \rho \otimes \pi$ for $\frac{a\_ + 1}{2} \leq x \leq \frac{a - 3}{2}$, we easily obtain that $\pi_2 \cong \delta([\nu^{\frac{a\_ + 1}{2}} \rho, \nu^{\frac{a - 1}{2}} \rho]) \times \delta([\nu^{\frac{a\_ + 1}{2}} \rho, \nu^{\frac{a - 1}{2}} \rho])$. Frobenius reciprocity gives
\begin{equation*} 
    \mu^{\ast}(\tau_{\beta}) \geq \delta([\nu^{\frac{a\_ + 1}{2}} \rho, \nu^{\frac{a - 1}{2}} \rho]) \times \delta([\nu^{\frac{a\_ + 1}{2}} \rho, \nu^{\frac{a - 1}{2}} \rho]) \otimes \pi_1,
\end{equation*}
so $\pi_1 \cong \delta([\nu^{-\frac{a\_ - 1}{2}} \rho, \nu^{\frac{a\_ - 1}{2}} \rho]) \rtimes \sigma'$.

The other direction is an immediate consequence of the Frobenius reciprocity.    
\begin{itemize}
    \item Claim $3$: There exists a unique $\gamma \in \{ 1, 2 \}$ such that $\mu^{\ast}(\tau_{\gamma})$ contains an irreducible constituent of the form
    $\delta([\nu^{\frac{a + 1}{2}} \rho, \nu^{\frac{b - 1}{2}} \rho]) \otimes \pi$.
    Furthermore, $\tau_{\gamma} \hookrightarrow \delta([\nu^{\frac{a + 1}{2}} \rho, \nu^{\frac{b - 1}{2}} \rho]) \rtimes \pi.$
\end{itemize}
Let us first determine all irreducible constituents of the form $\delta([\nu^{\frac{a + 1}{2}} \rho, \nu^{\frac{b - 1}{2}} \rho]) \otimes \pi$
appearing in $\mu^{\ast}(\delta([\nu^{-\frac{a - 1}{2}} \rho, \nu^{\frac{a - 1}{2}} \rho]) \rtimes \s')$. By the structural formula, there are $i, j$ such that $-\frac{a + 1}{2} \leq i \leq j \leq \frac{a - 1}{2}$ and an irreducible constituent $\d \otimes \pi'$ of $\mu^{\ast}(\s')$ such that
\begin{align*}
\delta([\nu^{\frac{a + 1}{2}} \rho, \nu^{\frac{b - 1}{2}} \rho]) & \leq \delta([\nu^{-i} \rho, \nu^{\frac{a - 1}{2}} \rho]) \times \delta([\nu^{j+1} \rho, \nu^{\frac{a - 1}{2}} \rho]) \times \delta
\intertext{and}
\pi & \leq  \delta([\nu^{i+1} \rho, \nu^{j} \rho]) \rtimes \pi'.
\end{align*}
\begin{sloppypar}
\noindent Considering cuspidal supports, it follows at once that $i = -\frac{a + 1}{2}$, $j = \frac{a - 1}{2}$ and  $\delta \cong \delta([\nu^{\frac{a + 1}{2}} \rho, \nu^{\frac{b - 1}{2}} \rho])$. Lemma \ref{lemapripremazadnja} implies that $\pi'$ is a discrete series and $\delta([\nu^{\frac{a + 1}{2}} \rho, \nu^{\frac{b - 1}{2}} \rho]) \otimes \pi'$ appears in $\mu^{\ast}(\sigma')$ with multiplicity one. Also, from Lemma \ref{lemapripremazadnja} follows that $\sigma'$ is a subrepresentation of $\delta([\nu^{\frac{a + 1}{2}} \rho, \nu^{\frac{b - 1}{2}} \rho]) \rtimes \pi'$, so Proposition \ref{propjord} implies that $a \in \Jord_{\rho}(\pi')$. Consequently, $\delta([\nu^{-\frac{a - 1}{2}} \rho, \nu^{\frac{a - 1}{2}} \rho]) \rtimes \pi'$ is irreducible and $\pi = \delta([\nu^{-\frac{a - 1}{2}} \rho, \nu^{\frac{a - 1}{2}} \rho]) \rtimes \pi'$. 
\end{sloppypar}

Thus, $\delta([\nu^{\frac{a + 1}{2}} \rho, \nu^{\frac{b - 1}{2}} \rho]) \otimes \delta([\nu^{-\frac{a - 1}{2}} \rho, \nu^{\frac{a - 1}{2}} \rho]) \rtimes \pi'$ is a unique irreducible constituent of $\mu^{\ast}(\delta([\nu^{-\frac{a - 1}{2}} \rho, \nu^{\frac{a - 1}{2}} \rho]) \rtimes \s')$ of the form $\delta([\nu^{\frac{a + 1}{2}} \rho, \nu^{\frac{b - 1}{2}} \rho]) \otimes \pi$, and it appears there with multiplicity one.

In the same way as in the proof of Claim $2$ one can prove that $\mu^{\ast}(\tau_{\gamma})$ contains $\delta([\nu^{\frac{a + 1}{2}} \rho, \nu^{\frac{b - 1}{2}} \rho]) \otimes \pi$ if and only if $\tau_{\gamma}$ is a subrepresentation of $\delta([\nu^{\frac{a + 1}{2}} \rho, \nu^{\frac{b - 1}{2}} \rho]) \rtimes \pi$.

\begin{itemize}
    \item Claim $4$: $\epsilon_{\sigma}((a\_, \rho), (a, \rho)) = 1$ if and only if $\alpha = \beta$.
\end{itemize}
Suppose that $\epsilon_{\sigma}((a\_,\rho),(a,\rho)) = 1$. Using Lemma \ref{lemajantz} and the description of $\Jord_{\rho}(\sigma)$, we obtain the following embeddings:
\begin{align} \label{eqntri}
\s & \hookrightarrow \delta([\nu^{\frac{a\_ + 1}{2}} \rho, \nu^{\frac{a - 1}{2}} \rho]) \rt \pi_1 \nonumber \\
\s & \hookrightarrow \delta([\nu^{\frac{a\_ + 1}{2}} \rho, \nu^{\frac{b\_ - 1}{2}} \rho]) \rt \pi_2 \nonumber \\
\s & \hookrightarrow \delta([\nu^{\frac{a + 1}{2}} \rho, \nu^{\frac{b\_ - 1}{2}} \rho]) \rtimes \tau_{\alpha},
\end{align}
for some irreducible representations $\pi_1$ and $\pi_2$ in $R(G)$.

Using Frobenius reciprocity and the structural formula, we conclude that $\mu^{\ast}(\pi_2)$ contains an irreducible constituent of the form $\delta([\nu^{\frac{a\_ + 1}{2}} \rho, \nu^{\frac{a - 1}{2}} \rho]) \otimes \pi'_2$. In the same way as in the last part of the proof of Claim $2$ we deduce that there is an irreducible representation $\pi \in R(G)$ such that $\pi_2$ is a subrepresentation of $\delta([\nu^{\frac{a\_ + 1}{2}} \rho, \nu^{\frac{a - 1}{2}} \rho]) \rtimes \pi$. 

Using the embedding $\s \hookrightarrow \delta([\nu^{\frac{a\_ + 1}{2}} \rho, \nu^{\frac{b\_ - 1}{2}} \rho]) \rt \pi_2$ and the Frobenius reciprocity, we get that $\mu^{\ast}(\s) \geq \delta([\nu^{\frac{a\_ + 1}{2}} \rho, \nu^{\frac{b\_ - 1}{2}} \rho]) \times \delta([\nu^{\frac{a\_ + 1}{2}} \rho, \nu^{\frac{a - 1}{2}} \rho]) \otimes \pi$. Note that the induced representation $\delta([\nu^{\frac{a\_ + 1}{2}} \rho, \nu^{\frac{b\_ - 1}{2}} \rho]) \times \delta([\nu^{\frac{a\_ + 1}{2}} \rho, \nu^{\frac{a - 1}{2}} \rho])$ is irreducible by \cite[Theorem~4.2]{Zel}. 

It follows from (\ref{eqntri}) and the structural formula that there are $i, j$ such that $\frac{a - 1}{2} \leq i \leq j \leq \frac{b\_ - 1}{2}$ and an irreducible constituent $\d' \otimes \pi'$ of $\mu^{\ast}(\tau_{\alpha})$ such that
\begin{multline*}
\delta([\nu^{\frac{a\_ + 1}{2}} \rho, \nu^{\frac{b\_ - 1}{2}} \rho]) \times \delta([\nu^{\frac{a\_ + 1}{2}} \rho, \nu^{\frac{a - 1}{2}} \rho]) \\ \leq \delta([\nu^{-i} \rho, \nu^{-\frac{a + 1}{2}} \rho]) \times \delta([\nu^{j+1} \rho, \nu^{\frac{b\_ - 1}{2}} \rho]) \times \d'.
\end{multline*}
Obviously, $i = \frac{a - 1}{2}$. Since $\tau_{\alpha}$ is a subrepresentation of $\delta([\nu^{-\frac{a - 1}{2}} \rho, \nu^{\frac{a - 1}{2}} \rho]) \rtimes \s'$, the fact that $a, a+1, \cdots, b\_ \notin \Jord_{\rho}(\s')$ and Lemma \ref{remarkprva} imply that $m^{\ast}(\d')$ does not contain an irreducible constituent of the form $\nu^{x} \rho \otimes \pi''$ for $\frac{a + 1}{2} \leq x \leq \frac{b\_ - 1}{2}$.

Consequently, $j = \frac{a - 1}{2}$ and $\delta' \cong \delta([\nu^{\frac{a\_ + 1}{2}} \rho, \nu^{\frac{a - 1}{2}} \rho]) \times \delta([\nu^{\frac{a\_ + 1}{2}} \rho, \nu^{\frac{a - 1}{2}} \rho])$. It follows that $\mu^{\ast}(\tau_{\alpha})$ contains an irreducible constituent of the form
\begin{equation*}
\delta([\nu^{\frac{a\_ + 1}{2}} \rho, \nu^{\frac{a - 1}{2}} \rho]) \times \delta([\nu^{\frac{a\_ + 1}{2}} \rho, \nu^{\frac{a - 1}{2}} \rho]) \otimes \pi    
\end{equation*}
and Claim $2$ implies $\alpha = \beta$.

Conversely, suppose that $\alpha = \beta$. Using Claims $1$ and $2$, we obtain that $\sigma$ is a subrepresentation of
\begin{gather*}
\delta([\nu^{\frac{a + 1}{2}} \rho, \nu^{\frac{b\_ - 1}{2}} \rho]) \times \delta([\nu^{\frac{a\_ + 1}{2}} \rho, \nu^{\frac{a - 1}{2}} \rho]) \times \delta([\nu^{\frac{a\_ + 1}{2}} \rho, \nu^{\frac{a - 1}{2}} \rho]) \times \\ \delta([\nu^{-\frac{a\_ - 1}{2}} \rho, \nu^{\frac{a\_ - 1}{2}} \rho]) \rtimes \s'.
\end{gather*}

By Lemma \ref{lemajantz}, there is an irreducible subquotient $\pi$ of 
\begin{equation}  \label{indglprvi}
    \delta([\nu^{\frac{a + 1}{2}} \rho, \nu^{\frac{b\_ - 1}{2}} \rho]) \times \delta([\nu^{\frac{a\_ + 1}{2}} \rho, \nu^{\frac{a - 1}{2}} \rho]) \times \delta([\nu^{\frac{a\_ + 1}{2}} \rho, \nu^{\frac{a - 1}{2}} \rho])
\end{equation}
such that $\sigma$ is a subrepresentation of $\pi \times \delta([\nu^{-\frac{a\_ - 1}{2}} \rho, \nu^{\frac{a\_ - 1}{2}} \rho]) \rtimes \s'$. It is an easy combinatorial exercise to see that the only irreducible subquotients of (\ref{indglprvi}) are 
\begin{gather*}
 \delta([\nu^{\frac{a\_ + 1}{2}} \rho, \nu^{\frac{a - 1}{2}} \rho]) \times \delta([\nu^{\frac{a\_ + 1}{2}} \rho, \nu^{\frac{b\_ - 1}{2}} \rho]) \\ 
    \intertext{and}
L(\delta([\nu^{\frac{a\_ + 1}{2}} \rho, \nu^{\frac{a - 1}{2}} \rho]), \delta([\nu^{\frac{a\_ + 1}{2}} \rho, \nu^{\frac{a - 1}{2}} \rho]), \delta([\nu^{\frac{a + 1}{2}} \rho, \nu^{\frac{b\_ - 1}{2}} \rho])).
\end{gather*}
In both cases we obtain that there is an irreducible representation $\pi_1 \in R(GL)$ such that $\pi$ is a subrepresentation of $\delta([\nu^{\frac{a\_ + 1}{2}} \rho, \nu^{\frac{a - 1}{2}} \rho]) \times \pi_1$. 

From 
\begin{equation*}
    \sigma \hookrightarrow \pi \times \delta([\nu^{-\frac{a\_ - 1}{2}} \rho, \nu^{\frac{a\_ - 1}{2}} \rho]) \rtimes \s'
\end{equation*}
and Lemma \ref{lemajantz} we get that there is also an irreducible representation $\pi_2 \in R(G)$ such that $\sigma$ is a subrepresentation of
$\delta([\nu^{\frac{a\_ + 1}{2}} \rho, \nu^{\frac{a - 1}{2}} \rho]) \times \pi_2$. Consequently, $\epsilon_{\sigma}((a\_, \rho), (a, \rho)) = 1$.

\begin{itemize}
    \item Claim $5$: $\epsilon_{\sigma}((b\_, \rho), (b, \rho)) = 1$ if and only if $\alpha = \gamma$.
\end{itemize}

\begin{sloppypar}
If $\epsilon_{\sigma}((b\_,\rho),(b,\rho)) = 1$, by Lemma \ref{epsilonequality}
there is a discrete series $\s''$ such that $\mu^{\ast}(\s) \geq \d([\nu^{-\frac{b\_ - 1}{2}} \rho, \nu^{\frac{b - 1}{2}} \rho]) \otimes \s''$. Since $\s$ is a subrepresentation of $\delta([\nu^{\frac{a + 1}{2}} \rho, \nu^{\frac{b\_ - 1}{2}} \rho]) \rtimes \tau_{\alpha}$, using the structural formula, together with the definition of $\tau_{\alpha}$, we conclude that $\mu^{\ast}(\tau_{\alpha})$ contains an irreducible constituent of the form $\delta([\nu^{x} \rho, \nu^{\frac{b - 1}{2}} \rho])$ $\otimes \pi_1$, for $x \leq 0$. The transitivity of Jacquet modules now implies that
$\mu^{\ast}(\tau_{\alpha})$ contains an irreducible constituent of the form $\delta([\nu^{\frac{a+1}{2}} \rho, \nu^{\frac{b-1}{2}} \rho]) \otimes \pi$, and the Claim $3$ implies $\alpha = \gamma$.
\end{sloppypar}

Conversely, suppose that $\alpha = \gamma$. Then the Jacquet module of $\sigma$ with respect to the appropriate parabolic subgroup contains an irreducible constituent of the form
\begin{equation*}
\delta([\nu^{\frac{a+1}{2}} \rho, \nu^{\frac{b\_-1}{2}} \rho]) \otimes \delta([\nu^{\frac{a+1}{2}} \rho, \nu^{\frac{b-1}{2}} \rho]) \otimes \pi.
\end{equation*}
By the transitivity of Jacquet modules, there is an irreducible constituent $\delta \otimes \pi$ in $\mu^{\ast}(\s)$ such that $m^{\ast}(\delta)$ contains
\begin{equation*}
\delta([\nu^{\frac{a+1}{2}} \rho, \nu^{\frac{b\_-1}{2}} \rho]) \otimes \delta([\nu^{\frac{a+1}{2}} \rho, \nu^{\frac{b-1}{2}} \rho]). 
\end{equation*}
Since $\s$ is a subrepresentation of $\delta([\nu^{-\frac{a-1}{2}} \rho, \nu^{\frac{b\_-1}{2}} \rho]) \rtimes \s'$, we obtain that there are $i, j$ such that $-\frac{a+1}{2} \leq i \leq j \leq \frac{b\_-1}{2}$ and an irreducible constituent $\d' \otimes \pi'$ of $\mu^{\ast}(\s')$ such that
\begin{equation*}
\d \leq \delta([\nu^{-i} \rho, \nu^{\frac{a-1}{2}} \rho]) \times \delta([\nu^{j+1} \rho, \nu^{\frac{b\_-1}{2}} \rho]) \times \d'.
\end{equation*}
Considering cuspidal supports of $\delta$, obviously, $i = -\frac{a+1}{2}$. Since $b\_ \not\in \Jord_{\rho}(\s')$, $a\_ < a$, and $b \in \Jord_{\rho}(\s')$, using Lemma \ref{remarkprva} we deduce that $j = \frac{a-1}{2}$ and $\d' \cong \delta([\nu^{\frac{a+1}{2}} \rho, \nu^{\frac{b-1}{2}} \rho])$. It follows that
\begin{equation*}
\delta \cong \delta([\nu^{\frac{a+1}{2}} \rho, \nu^{\frac{b\_-1}{2}} \rho]) \times \delta([\nu^{\frac{a+1}{2}} \rho, \nu^{\frac{b-1}{2}} \rho]) \cong \delta([\nu^{\frac{a+1}{2}} \rho, \nu^{\frac{b-1}{2}} \rho]) \times \delta([\nu^{\frac{a+1}{2}} \rho, \nu^{\frac{b\_-1}{2}} \rho]).    
\end{equation*}
Using Frobenius reciprocity and the transitivity of Jacquet modules, we get that $\mu^{\ast}(\s)$ contains an irreducible constituent of the form $\delta([\nu^{\frac{b\_ + 1}{2}} \rho, \nu^{\frac{b-1}{2}} \rho]) \otimes \pi'$, so $\epsilon_{\s}((b\_,\rho),(b,\rho)) = 1$.

\begin{itemize}
    \item Claim $6$: $\epsilon_{\sigma'}((a\_, \rho), (b, \rho)) = 1$ if and only if $\beta = \gamma$.
\end{itemize}
By Lemma \ref{lemapripremazadnja}, there is a unique discrete series $\s''$ such that $\s'$ is a subrepresentation of $\delta([\nu^{\frac{a+1}{2}} \rho, \nu^{\frac{b-1}{2}} \rho]) \rtimes \s''$. Thus, every irreducible subrepresentation of $\delta([\nu^{-\frac{a- 1}{2}} \rho, \nu^{\frac{a-1}{2}} \rho]) \rtimes \s'$ is also a subrepresentation of 
\begin{equation*}
    \delta([\nu^{-\frac{a- 1}{2}} \rho, \nu^{\frac{a-1}{2}} \rho]) \times \delta([\nu^{\frac{a+1}{2}} \rho, \nu^{\frac{b-1}{2}} \rho]) \rtimes \s'', 
\end{equation*}
and in $R(G)$ we have
\begin{gather*}
\delta([\nu^{-\frac{a- 1}{2}} \rho, \nu^{\frac{a-1}{2}} \rho]) \times \delta([\nu^{\frac{a+1}{2}} \rho, \nu^{\frac{b-1}{2}} \rho]) \rtimes \s'' = \\ \delta([\nu^{-\frac{a- 1}{2}} \rho, \nu^{\frac{b-1}{2}} \rho])\rtimes \s'' + L(\delta([\nu^{-\frac{a- 1}{2}} \rho, \nu^{\frac{a-1}{2}} \rho]), \delta([\nu^{\frac{a+1}{2}} \rho, \nu^{\frac{b-1}{2}} \rho])) \rtimes \s''.
\end{gather*}
We split the proof of Claim $6$ in three sub-claims.
\begin{itemize}
    \item Sub-claim $1$: Each of the induced representations 
\begin{gather*}
\delta([\nu^{-\frac{a- 1}{2}} \rho, \nu^{\frac{b-1}{2}} \rho])\rtimes 
\s'' \\
\intertext{and}
L(\delta([\nu^{-\frac{a- 1}{2}} \rho, \nu^{\frac{a-1}{2}} \rho]), \delta([\nu^{\frac{a+1}{2}} \rho, \nu^{\frac{b-1}{2}} \rho])) \rtimes \s''
\end{gather*}
contains exactly one irreducible tempered subrepresentation of 
\begin{equation*}
\delta([\nu^{-\frac{a- 1}{2}} \rho, \nu^{\frac{a-1}{2}} \rho]) \rtimes \s'.
\end{equation*}
\end{itemize}
\begin{sloppypar}
Using the structural formula and the square-integrability criterion for $\sigma''$, we deduce that if $\mu^{\ast}(\delta([\nu^{-\frac{a- 1}{2}} \rho$, $\nu^{\frac{a-1}{2}} \rho]) \times \delta([\nu^{\frac{a+1}{2}} \rho, \nu^{\frac{b-1}{2}} \rho]) \rtimes \s'')$ contains $\delta([\nu^{-\frac{a- 1}{2}} \rho, \nu^{\frac{a-1}{2}} \rho]) \otimes \s'$, then $\s'$ is an irreducible subquotient of $\delta([\nu^{\frac{a+1}{2}} \rho, \nu^{\frac{b-1}{2}} \rho]) \rtimes \s''$. In the same way we obtain that $\mu^{\ast}(\delta([\nu^{-\frac{a- 1}{2}} \rho$, $\nu^{\frac{a-1}{2}} \rho]) \times \delta([\nu^{\frac{a+1}{2}} \rho, \nu^{\frac{b-1}{2}} \rho]) \rtimes \s'')$ contains $\delta([\nu^{-\frac{a- 1}{2}} \rho, \nu^{\frac{a-1}{2}} \rho]) \otimes \delta([\nu^{\frac{a+1}{2}} \rho, \nu^{\frac{b-1}{2}} \rho]) \rtimes \s''$ with multiplicity two.

It follows from Lemma \ref{lemapripremazadnja} that $\s'$ appears with multiplicity one in the composition series of $\delta([\nu^{\frac{a+1}{2}} \rho, \nu^{\frac{b-1}{2}} \rho]) \rtimes \s''$. Thus, $\mu^{\ast}(\delta([\nu^{-\frac{a- 1}{2}} \rho$, $\nu^{\frac{a-1}{2}} \rho]) \times \delta([\nu^{\frac{a+1}{2}} \rho, \nu^{\frac{b-1}{2}} \rho]) \rtimes \s'')$ contains $\delta([\nu^{-\frac{a- 1}{2}} \rho, \nu^{\frac{a-1}{2}} \rho]) \otimes \s'$ with multiplicity two. 
\end{sloppypar}

Repeating the same arguments, we see that $\mu^{\ast}(\delta([\nu^{-\frac{a- 1}{2}} \rho, \nu^{\frac{b-1}{2}} \rho])\rtimes \s'')$ contains $\delta([\nu^{-\frac{a- 1}{2}} \rho, \nu^{\frac{a-1}{2}} \rho]) \otimes \s'$ with multiplicity one. 
This proves Sub-claim $1$.

\begin{itemize}
    \item Sub-claim $2$: $\tau_{\gamma}$ is contained in $\delta([\nu^{-\frac{a- 1}{2}} \rho, \nu^{\frac{b-1}{2}} \rho]) \rtimes 
\s''$, but $\tau_{3 - \gamma}$ is not.
\end{itemize}
\begin{sloppypar}
Lemma \ref{remarkprva} and the fact that $a \in Jord_{\rho}(\s'')$ and $a+2, \cdots, b \notin Jord_{\rho}(\s'')$ imply that the induced representation $\delta([\nu^{-\frac{a- 1}{2}} \rho, \nu^{\frac{a-1}{2}} \rho]) \rtimes \s''$ is irreducible and $\mu^{\ast}(\s'')$ does not contain irreducible constituents of the form $\nu^{x} \rho \otimes \pi$ for $\frac{a+1}{2} \leq x \leq \frac{b-1}{2}$. Therefore, it follows from the structural formula that $\mu^{\ast}(\delta([\nu^{-\frac{a- 1}{2}} \rho, \nu^{\frac{a-1}{2}} \rho]) \times \delta([\nu^{\frac{a+1}{2}} \rho, \nu^{\frac{b-1}{2}} \rho]) \rtimes \s'')$ contains a unique irreducible constituent of the form $\delta([\nu^{\frac{a+1}{2}} \rho, \nu^{\frac{b-1}{2}} \rho]) \otimes \pi$ , which appears there with multiplicity one, and such an irreducible constituent is obviously also contained in $\mu^{\ast}(\delta([\nu^{-\frac{a- 1}{2}} \rho, \nu^{\frac{b-1}{2}} \rho]) \rtimes \s'')$. We note that $\pi \cong \delta([\nu^{-\frac{a- 1}{2}} \rho, \nu^{\frac{a-1}{2}} \rho]) \rtimes \s''$.
\end{sloppypar}

Thus, an irreducible tempered subrepresentation $\tau$ of $\delta([\nu^{-\frac{a- 1}{2}} \rho, \nu^{\frac{a-1}{2}} \rho]) \rtimes \s'$ contains an irreducible representation of the form $\delta([\nu^{\frac{a+1}{2}} \rho, \nu^{\frac{b-1}{2}} \rho]) \otimes \pi$ in the Jacquet module with respect to the appropriate parabolic subgroup if and only if $\tau$ is contained in $\delta([\nu^{-\frac{a- 1}{2}} \rho, \nu^{\frac{b-1}{2}} \rho])\rtimes \s''$. This proves Sub-claim $2$.

\begin{sloppypar}
\begin{itemize}
    \item Sub-claim $3$: $\epsilon_{\s'}((a\_, \rho), (b, \rho)) = 1$ if and only if $\tau_{\beta}$ is contained in $\delta([\nu^{-\frac{a- 1}{2}} \rho, \nu^{\frac{b-1}{2}} \rho]) \rtimes \s''$.
\end{itemize}
Note that Sub-claim 2 and 3 imply Claim 6 since $\tau_{\beta}$ is contained in $\delta([\nu^{-\frac{a- 1}{2}} \rho,$ $\nu^{\frac{b-1}{2}} \rho]) \rtimes \s''$ if and only if $\beta=\tau$. We now prove Sub-claim 3. Suppose that $\epsilon_{\s'}((a\_, \rho), (b, \rho)) = 1$. Since $\delta([\nu^{-\frac{a\_ - 1}{2}} \rho, \nu^{\frac{a\_ - 1}{2}} \rho]) \rtimes \s'$ is irreducible and $\epsilon_{\s'}((a\_, \rho), (b, \rho)) = 1$, it follows that $\mu^{\ast}(\delta([\nu^{-\frac{a\_ - 1}{2}} \rho, \nu^{\frac{a\_ - 1}{2}} \rho]) \rtimes \s')$ contains an irreducible constituent of the form $\delta([\nu^{\frac{a\_ + 1}{2}} \rho, \nu^{\frac{b - 1}{2}} \rho]) \otimes \pi_1$. Thus, the Jacquet module of $\tau_{\beta}$ with respect to the appropriate parabolic subgroup contains
\end{sloppypar}
\begin{equation*}
\delta([\nu^{\frac{a\_ + 1}{2}} \rho, \nu^{\frac{a - 1}{2}} \rho]) \times \delta([\nu^{\frac{a\_ + 1}{2}} \rho, \nu^{\frac{a - 1}{2}} \rho]) \otimes \delta([\nu^{\frac{a\_ + 1}{2}} \rho, \nu^{\frac{b - 1}{2}} \rho]) \otimes \pi_1.
\end{equation*}
The transitivity of Jacquet modules implies that there is an irreducible constituent $\delta_1 \otimes \pi_1$ in $\mu^{\ast}(\tau_{\beta})$ such that $m^{\ast}(\delta_1)$ contains
\begin{equation} \label{jmsedmi}
\delta([\nu^{\frac{a\_ + 1}{2}} \rho, \nu^{\frac{a - 1}{2}} \rho]) \times \delta([\nu^{\frac{a\_ + 1}{2}} \rho, \nu^{\frac{a - 1}{2}} \rho]) \otimes \delta([\nu^{\frac{a\_ + 1}{2}} \rho, \nu^{\frac{b - 1}{2}} \rho]).
\end{equation}
Since $\tau_{\beta}$ is a subrepresentation of $\delta([\nu^{-\frac{a- 1}{2}} \rho, \nu^{\frac{a-1}{2}} \rho]) \rtimes \s'$, it follows from the structural formula that there are $i, j$ such that $-\frac{a + 1}{2} \leq i \leq j \leq \frac{a-1}{2}$ and an irreducible constituent $\delta_2 \otimes \pi_2$ of $\mu^{\ast}(\s')$ such that
\begin{equation*}
\delta_1 \leq \delta([\nu^{-i} \rho, \nu^{\frac{a - 1}{2}} \rho]) \times \delta([\nu^{j+1} \rho, \nu^{\frac{a - 1}{2}} \rho]) \times \delta_2.
\end{equation*}
By Lemma \ref{remarkprva} and the transitivity of Jacquet modules, we get that $m^{\ast}(\delta_2)$ does not contain an irreducible constituent of the form $\nu^{x} \rho \otimes \pi$ such that $\frac{a\_ + 1}{2} \leq x \leq \frac{a - 1}{2}$ since $a\_ +2, \cdots, a \notin Jord_{\rho}(\s')$. Since $m^{\ast}(\delta([\nu^{-i} \rho, \nu^{\frac{a - 1}{2}} \rho]) \times \delta([\nu^{j+1} \rho, \nu^{\frac{a - 1}{2}} \rho]) \times \delta_2)$ contains the irreducible constituent (\ref{jmsedmi}), we deduce $i = - \frac{a\_ + 1}{2}$ and $j = \frac{a\_ - 1}{2}$. It directly follows that $\delta_2 \cong \delta([\nu^{\frac{a\_ + 1}{2}} \rho, \nu^{\frac{b - 1}{2}} \rho])$ and, consequently,
\begin{align*}
\delta_1 &\cong \delta([\nu^{\frac{a\_ + 1}{2}} \rho, \nu^{\frac{a - 1}{2}} \rho]) \times \delta([\nu^{\frac{a\_ + 1}{2}} \rho, \nu^{\frac{a - 1}{2}} \rho]) \times \delta([\nu^{\frac{a\_ + 1}{2}} \rho, \nu^{\frac{b - 1}{2}} \rho]) \\
& \cong \delta([\nu^{\frac{a\_ + 1}{2}} \rho, \nu^{\frac{b - 1}{2}} \rho]) \times \delta([\nu^{\frac{a\_ + 1}{2}} \rho, \nu^{\frac{a - 1}{2}} \rho]) \times \delta([\nu^{\frac{a\_ + 1}{2}} \rho, \nu^{\frac{a - 1}{2}} \rho]).
\end{align*}
Using the Frobenius reciprocity and the transitivity of Jacquet modules again, we conclude that $\mu^{\ast}(\tau_{\beta})$ contains an irreducible constituent of the form
\begin{equation*} 
\delta([\nu^{\frac{a + 1}{2}} \rho, \nu^{\frac{b - 1}{2}} \rho]) \otimes \pi.
\end{equation*}
Consequently, $\beta = \gamma$ and $\tau_{\beta}$ is contained in $\delta([\nu^{-\frac{a- 1}{2}} \rho, \nu^{\frac{b-1}{2}} \rho]) \rtimes \s''$, by Sub-claim $2$.

Conversely, let us assume that $\tau_{\beta}$ is contained in $\delta([\nu^{-\frac{a- 1}{2}} \rho, \nu^{\frac{b-1}{2}} \rho]) \rtimes \s''$. Thus, $\mu^{\ast}(\delta([\nu^{-\frac{a- 1}{2}} \rho, \nu^{\frac{b-1}{2}} \rho]) \rtimes \s'')$ contains (\ref{jmdrugi}). The structural formula implies that there are $i, j$ such that $-\frac{a + 1}{2} \leq i \leq j \leq \frac{b-1}{2}$ and
an irreducible constituent $\delta_1 \otimes \pi_1$ of $\mu^{\ast}(\s'')$ such that
\begin{equation*}
\delta([\nu^{\frac{a\_ + 1}{2}} \rho, \nu^{\frac{a - 1}{2}} \rho]) \times \delta([\nu^{\frac{a\_ + 1}{2}} \rho, \nu^{\frac{a - 1}{2}} \rho]) \leq \delta([\nu^{- i} \rho, \nu^{\frac{a - 1}{2}} \rho]) \times \delta([\nu^{j+1} \rho, \nu^{\frac{b - 1}{2}} \rho]) \times \delta_1.
\end{equation*}
Clearly, $i=-\frac{a\_ + 1}{2}$, $j = \frac{b-1}{2}$ and $\mu^{\ast}(\s'')$ contains an irreducible constituent of the form $\delta([\nu^{\frac{a\_ + 1}{2}} \rho, \nu^{\frac{a - 1}{2}} \rho]) \otimes \pi_1$. 

\begin{sloppypar}
Since $\s'$ is a subrepresentation of $\delta([\nu^{\frac{a+1}{2}} \rho, \nu^{\frac{b-1}{2}} \rho]) \rtimes \s''$, using Frobenius reciprocity and the transitivity of Jacquet modules we deduce that the Jacquet module of $\s'$ with respect to the appropriate parabolic subgroup contains $\delta([\nu^{\frac{a+1}{2}} \rho, \nu^{\frac{b-1}{2}} \rho]) \otimes \delta([\nu^{\frac{a\_ + 1}{2}} \rho, \nu^{\frac{a - 1}{2}} \rho]) \otimes \pi_1$. 
\end{sloppypar}
By the transitivity of Jacquet modules, there is an irreducible constituent $\delta_2 \otimes \pi_2$ of $\mu^{\ast}(\sigma')$ such that $m^{\ast}(\delta_2)$ contains $\delta([\nu^{\frac{a+1}{2}} \rho, \nu^{\frac{b-1}{2}} \rho]) \otimes \delta([\nu^{\frac{a\_ + 1}{2}} \rho, \nu^{\frac{a - 1}{2}} \rho])$. Repeating the same arguments as in the proof of the Claim $5$, or as in the first part of the proof of this sub-claim, we obtain $\delta_2 \cong \delta([\nu^{\frac{a\_ + 1}{2}} \rho, \nu^{\frac{b - 1}{2}} \rho])$. Now Lemma \ref{lemaequiv} implies $\epsilon_{\s'}((a\_, \rho), (b, \rho)) = 1$. This ends the proof of Sub-claim 3.

Finally, Claims 4, 5, and 6 imply the theorem:
\begin{equation}
\epsilon_{\s'}((a\_, \rho), (b, \rho))=\epsilon_{\s}((a\_, \rho), (a, \rho))\epsilon_{\s}((b\_, \rho), (b, \rho)).
 \tag*{\qedhere} 
\end{equation}
\end{proof}

\section{Invariants of discrete series II: the $\epsilon$-function on a certain subset of $\Jord(\sigma)$}

In this section we define and study the $\epsilon$-function on a certain subset of the set of the Jordan blocks. Throughout this section we denote the partial cuspidal support of the irreducible representation $\s$ by $\s_{cusp}$.

In Definition \ref{defone}, we defined the values of $\epsilon_{\sigma}$ on the intersection of its domain with $\Jord(\sigma) \times \Jord(\sigma)$. Now we will define it on the intersection of its domain with $\Jord(\sigma)$; the restriction of $\epsilon_{\sigma}$ to this intersection will be referred to as the "$\epsilon$-function on single pairs". 

In the following lemma we gather some results on the embeddings of discrete series.

\begin{lemma} \label{lemaulaganje}
Let $\s \in R(G)$ denote a discrete series and suppose that $\rho \in R(GL)$ is an irreducible essentially self-dual cuspidal representation such that some twists of $\rho$ appear in the cuspidal support of $\s$. 
\begin{enumerate}[(1)]
    \item If there is an $a \in \Jord_{\rho}(\s)$ such that $a\_$ is defined and $\epsilon_{\s}((a\_, \rho),(a, \rho)) = 1$, then there is a discrete series $\s'$ such that $\s$ is a subrepresentation of $\d([\nu^{-\frac{a\_ - 1}{2}} \rho, \nu^{\frac{a-1}{2}} \rho]) \rt \s'$, and $\mu^{\ast}(\s)$ contains $\d([\nu^{-\frac{a\_ - 1}{2}} \rho, \nu^{\frac{a-1}{2}} \rho]) \otimes \s'$ with multiplicity one. If $\delta([\nu^{-\frac{a\_ - 1}{2}} \rho, \nu^{\frac{a-1}{2}} \rho]) \otimes \pi$ is an irreducible constituent of $\mu^{\ast}(\sigma)$, then $\pi \cong \sigma'$, and $\delta([\nu^{-\frac{a\_ - 1}{2}} \rho, \nu^{\frac{a-1}{2}} \rho]) \otimes \sigma'$ appears in $\mu^{\ast}(\d([\nu^{-\frac{a\_ - 1}{2}} \rho, \nu^{\frac{a-1}{2}} \rho]) \rt \s')$ with multiplicity two.
    \item If $\epsilon_{\s}((a\_, \rho),(a, \rho)) = -1$ for every $a \in \Jord_{\rho}(\s)$ such that $a\_$ is defined, then there is an irreducible representation $\delta \in R(GL)$ whose cuspidal support consists only of twists of $\rho$ and a discrete series $\s' \in R(G)$ without twists of $\rho$ in the cuspidal support such that $\s$ is a subrepresentation of $\d \rt \s'$. Also, $\d \otimes \s'$ appears in $\mu^{\ast}(\s)$ with multiplicity one and for $(b, \rho') \in \Jord(\s)$, such that $b\_$ is defined and $\rho' \not\cong \rho$, we have $\epsilon_{\s}((b\_, \rho'), (b, \rho')) = \epsilon_{\s'}((b\_, \rho'), (b, \rho'))$.
\end{enumerate}
\end{lemma}
\begin{proof}
If there is an $a \in \Jord_{\rho}(\s)$ such that $a\_$ is defined and $\epsilon_{\s}((a\_, \rho),(a, \rho)) = 1$ holds, the statement of the lemma can be proved in the exactly same way as \cite[Theorem~2.3]{Mu7}, proof of which is completely based on the structural formula, the square-integrability criterion, and the result analogous to the one given in Lemma \ref{remarkprva}.

Let us now assume that $\epsilon_{\s}((a\_, \rho),(a, \rho)) = -1$ for every $a \in \Jord_{\rho}(\s)$ such that $a\_$ is defined. By Theorem \ref{embed}, we deduce that there is an ordered $n$-tuple $(\s_1, \s_2, \ldots, \s_n)$ of discrete series $\s_i \in R(G)$ such that $\s_1$ is strongly positive, $\s_n \cong \s$, and, for every $i = 2, 3, \ldots, n$, there are $(a_i, \rho_i), (b_i, \rho_i) \in \Jord(\s)$ such that in $\Jord_{\rho_i}(\s_i)$ we have $a_i = (b_i)\_$ and $\s_i$ is a subrepresentation of $\d([\nu^{-\frac{a_i - 1}{2}} \rho_i, \nu^{\frac{b_i-1}{2}} \rho_i]) \rtimes \s_{i-1}$ and $\epsilon_{\s}((a_i, \rho_i),(b_i, \rho_i)) = 1$.

Obviously, $\rho \not\cong \rho_i$ for $i = 2, 3, \ldots, n$. By \cite[Theorem~5.3]{Matic4}, there is an irreducible representation $\delta_1$, which is the unique irreducible subrepresentation of an induced representation of the form
\begin{equation*}
\d([\nu^{x} \rho, \nu^{y_1} \rho]) \times \d([\nu^{x+1} \rho, \nu^{y_2} \rho]) \times \cdots \times \d([\nu^{x + k -1} \rho, \nu^{y_k} \rho]),
\end{equation*}
where $x + i - 1 \leq y_i$ and $y_i < y_{i+1}$ for $i = 1, 2, \ldots, k$, and $\nu^{x + k -1} \rho \rtimes \s_{cusp}$ reduces, such that $\s_1$ is a subrepresentation of $\delta_1 \rtimes \sigma_{sp}$, where $\sigma_{sp}$ is a strongly positive discrete series without twists of $\rho$ in the cuspidal support. Also, $\delta_1 \otimes \sigma_{sp}$ appears in $\mu^{\ast}(\s_1)$ with multiplicity one. Note that $\Jord_{\rho_i}(\s_1) = \Jord_{\rho_i}(\sigma_{sp})$ for $i = 2, \ldots, n$, and for $(b, \rho') \in \Jord(\s_1)$, such that $b\_$ is defined and $\rho' \not\cong \rho$, we have $\epsilon_{\s_1}((b\_, \rho'), (b, \rho')) = \epsilon_{\sigma_{sp}}((b\_, \rho'), (b, \rho'))$. It is a direct consequence of \cite[Subsection~6.3]{LapidMinguez1} that for $i = 2, 3, \ldots, n$ we have
\begin{equation*}
\d([\nu^{-\frac{a_i - 1}{2}} \rho_i, \nu^{\frac{b_i-1}{2}} \rho_i]) \times \delta_1 \cong \delta_1 \times \d([\nu^{-\frac{a_i - 1}{2}} \rho_i, \nu^{\frac{b_i-1}{2}} \rho_i]).
\end{equation*}
Thus, there is an embedding
\begin{equation*}
\s_2 \hookrightarrow \delta_1 \times \d([\nu^{-\frac{a_2 - 1}{2}} \rho_2, \nu^{\frac{b_2-1}{2}} \rho_2]) \rtimes \sigma_{sp}.
\end{equation*}
So, there is an irreducible representation $\pi_1$ such that $\s_2$ is a subrepresentation of $\delta_1 \rtimes \pi_1$. Frobenius reciprocity implies that $\mu^{\ast}(\s_2) \geq \delta_1 \otimes \pi_1$ and, since no twists of $\rho$ appear in the cuspidal support of $\d([\nu^{-\frac{a_2 - 1}{2}} \rho_2, \nu^{\frac{b_2-1}{2}} \rho_2]) \rtimes \sigma_{sp}$, it follows that $\pi_1$ is a subquotient of $\d([\nu^{-\frac{a_2 - 1}{2}} \rho_2, \nu^{\frac{b_2-1}{2}} \rho_2]) \rtimes \sigma_{sp}$.

Since $\mu^{\ast}(\delta_1 \rtimes \pi_1) \geq \mu^{\ast}(\s_2) \geq \d([\nu^{-\frac{a_2 - 1}{2}} \rho_2, \nu^{\frac{b_2-1}{2}} \rho_2]) \otimes \s_1$, it can be easily seen that $\mu^{\ast}(\pi_1)$ contains an irreducible constituent of the form $\d([\nu^{-\frac{a_2 - 1}{2}} \rho_2$, $\nu^{\frac{b_2-1}{2}} \rho_2]) \otimes \pi'_1$. Using $\Jord_{\rho_2}(\s_1) = \Jord_{\rho_2}(\sigma_{sp})$ and the first part of the lemma, we conclude that the only irreducible constituent of such a form, appearing in $\mu^{\ast}(\d([\nu^{-\frac{a_2 - 1}{2}} \rho_2$, $\nu^{\frac{b_2-1}{2}} \rho_2]) \rtimes \sigma_{sp})$, is $\d([\nu^{-\frac{a_2 - 1}{2}} \rho_2, \nu^{\frac{b_2-1}{2}} \rho_2]) \otimes \sigma_{sp}$, which appears there with multiplicity two. Since $\d([\nu^{-\frac{a_2 - 1}{2}} \rho_2, \nu^{\frac{b_2-1}{2}} \rho_2]) \otimes \sigma_{sp}$ is contained in the Jacquet modules of both irreducible subrepresentations of $\d([\nu^{-\frac{a_2 - 1}{2}} \rho_2, \nu^{\frac{b_2-1}{2}} \rho_2]) \rtimes \sigma_{sp}$, we conclude that $\pi_1$ is a subrepresentation of $\d([\nu^{-\frac{a_2 - 1}{2}} \rho_2, \nu^{\frac{b_2-1}{2}} \rho_2]) \rtimes \sigma_{sp}$. Theorem \ref{dissub} shows that $\pi_1$ is a discrete series. Furthermore, no twists of $\rho$ appear in the cuspidal support of $\pi_1$ and $\Jord_{\rho_i}(\pi_1) = \Jord_{\rho_i}(\s_2)$ for $i = 3, \ldots, n$.

If we denote by $\s'_2$ the irreducible subrepresentation of the induced representation $\d([\nu^{-\frac{a_2 - 1}{2}} \rho_2$, $\nu^{\frac{b_2-1}{2}} \rho_2]) \rtimes \s_1$ different than $\s_2$, applying the same arguments we can conclude that there is an irreducible subrepresentation $\pi'_1$ of $\d([\nu^{-\frac{a_2 - 1}{2}} \rho_2, \nu^{\frac{b_2-1}{2}} \rho_2]) \rtimes \sigma_{sp}$ such that $\s'_2$ is a subrepresentation of $\delta_1 \rtimes \pi'_1$. Since both $\pi_1$ and $\pi'_1$ appear in the composition series of $\d([\nu^{-\frac{a_2 - 1}{2}} \rho_2, \nu^{\frac{b_2-1}{2}} \rho_2]) \rtimes \sigma_{sp}$ with multiplicity one, we get that $\delta_1 \otimes \pi_1$ appears in $\mu^{\ast}(\s_2)$ with multiplicity one. Obviously, for $(b, \rho') \in \Jord(\s_2)$, such that $b\_$ is defined and $\rho' \not\cong \rho$, we have $\epsilon_{\s_2}((b\_, \rho'), (b, \rho')) = \epsilon_{\pi_1}((b\_, \rho'), (b, \rho'))$.

Thus, if $n=2$ we can take $\delta \cong \delta_1$ and $\sigma' \cong \pi_1$. Suppose that $n \geq 3$. Repeating the same arguments, we deduce that $\sigma_3$ is a subrepresentation of $\delta_1 \rtimes \pi_2$, for a discrete series $\pi_2$ with desired properties. If $n \geq 4$, repeating this procedure we obtain the claim of the second part of the lemma. Note that for $\delta$ as in the statement of the lemma we have $\delta \cong \delta_1$. 
\end{proof}


\begin{lemma} \label{lemabeforesolo}
Let $\s \in R(G)$ denote a discrete series, and suppose that $\rho \in \Irr(GL)$ is a cuspidal unitarizable essentially self-dual representation such that $\rho \rt \s_{cusp}$ reduces. Write $\rho \rt \s_{cusp} = \tau_1 + \tau_{-1}$, where representations $\tau_1$ and $\tau_{-1}$ are irreducible tempered and mutually non-isomorphic. Suppose that $\Jord_{\rho}(\s) = \{ a\_, a \}$. Also, suppose that if $\epsilon_{\s}((b\_,\rho'),(b,\rho')) = 1$ for some $(b, \rho') \in \Jord(\s)$, then $\Jord_{\rho'}(\s_{cusp}) = \emptyset$ and $\Jord_{\rho'}(\s) = \{ b\_, b \}$. Then there is a unique $i \in \{ 1, -1 \}$ such that the Jacquet module of $\s$ with respect to an appropriate parabolic subgroup contains an irreducible constituent of the form $\pi \otimes \d([ \nu \rho, \nu^{\frac{a - 1}{2}} \rho]) \otimes \tau_i$. Also, if $\mu^{\ast}(\s)$ contains an irreducible constituent of the form $\pi' \otimes \tau$, where $\tau$ is an irreducible subquotient of $\d([ \nu^{x} \rho, \nu^{\frac{a - 1}{2}} \rho]) \rt \s_{cusp}$, then $x \leq 0$.
\end{lemma}
\begin{proof}
It follows from the classification of strongly positive discrete series, given in \cite[Theorem~A]{Kim1}, that for a strongly positive discrete series $\sigma_{sp}$ with the partial cuspidal support $\s_{cusp}$ we have $\Jord_{\rho}(\s_{sp}) = \emptyset$. Note that the cuspidal support of the strongly positive discrete series does not contain twists of $\rho$.

Suppose that $\epsilon((a\_, \rho), (a, \rho)) = -1$. Since $\Jord_{\rho}(\s) = \{ a\_, a \}$, an inductive application of Theorem \ref{embed} implies that there is a strongly positive discrete series $\sigma_{sp}$ such that $\Jord_{\rho}(\sigma_{sp}) = \{ a\_, a \}$, which is impossible.

By Theorem \ref{embed}, there exists an ordered $n$-tuple of discrete series $(\s_1, \s_2$, $\ldots$, $\s_n)$, $\s_i \in \Irr(G_{n_i})$, such that $\s_1$ is strongly positive, $\s_n \cong \s$, and for every $i = 2, \ldots, n$, there are $(a_i, \rho_i), (b_i, \rho_i) \in \Jord(\s)$ such that in $\Jord_{\rho_i}(\s_i)$ we have $a_i = (b_i)\_$ and
\begin{equation*}
\s_i \hookrightarrow \d([\nu^{-\frac{a_i -1}{2}} \rho_i, \nu^{\frac{b_i -1}{2}} \rho_i]) \rt \s_{i-1}.
\end{equation*}
We can take $(a_n, b_n, \rho_n) = (a\_, a, \rho)$, and for $i = 2, 3, \ldots, n$ we have $\Jord_{\rho_i}(\s_{cusp}) = \emptyset$ and $\rho_{i_1} \not\cong \rho_{i_2}$ for $i_1, i_2 \in \{ 2, 3, \ldots, n \}$, $i_1 \neq i_2$. Also, there are no twists of $\rho_2, \rho_3, \ldots, \rho_{n}$ appearing in the cuspidal support of $\s_1$.

It follows from \cite[Theorem~4.6]{Matic4} or \cite[Section~7]{MatTad} that there is a unique irreducible representation $\pi_1$ such that $\s_1$ is a subrepresentation of $\pi_1 \rt \s_{cusp}$. Also, $\mu^{\ast}(\s_1)$ contains $\pi_1 \ot \s_{cusp}$ with multiplicity one, and there are no twists of $\rho, \rho_2, \ldots, \rho_{n-1}$ appearing in the cuspidal support of $\pi_1$ (the explicit form of $\pi_1$ can be deduced from \cite[Theorem~4.6]{Matic4} or \cite[Section~7]{MatTad}).
We have an embedding
\begin{align*}
\s \hookrightarrow & \d([\nu^{-\frac{a\_ -1}{2}} \rho, \nu^{\frac{a -1}{2}} \rho]) \times \d([\nu^{-\frac{a_{n-1} -1}{2}} \rho_{n-1}, \nu^{\frac{b_{n-1} -1}{2}} \rho_{n-1}]) \times \cdots \times  \\
& \times \d([\nu^{-\frac{a_2 -1}{2}} \rho_2, \nu^{\frac{b_2 -1}{2}} \rho_2]) \times \pi_1 \rt \s_{cusp}.
\end{align*}

We note that the induced representation $\d([\nu^{-\frac{a_{n-1} -1}{2}} \rho_{n-1}, \nu^{\frac{b_{n-1} -1}{2}} \rho_{n-1}]) \times \cdots \times \d([\nu^{-\frac{a_2 -1}{2}} \rho_2, \nu^{\frac{b_2 -1}{2}} \rho_2]) \times \pi_1$ is irreducible (this is proved in much bigger generality in \cite[Subsection~6.3]{LapidMinguez1}) and denote it by $\pi_2$.

\begin{itemize}
    \item Claim $1$: $\pi_2 \otimes \sigma_{cusp}$ appears in $\mu^{\ast}(\s_{n-1})$ with multiplicity one.
\end{itemize}

Let us prove the Claim $1$. The transitivity of Jacquet modules shows that such a multiplicity is less than or equal to the multiplicity of
\begin{equation} \label{multpl}
\d([\nu^{-\frac{a_{n-1} -1}{2}} \rho_{n-1}, \nu^{\frac{b_{n-1} -1}{2}} \rho_{n-1}]) \otimes \cdots \otimes \d([\nu^{-\frac{a_2 -1}{2}} \rho_2, \nu^{\frac{b_2 -1}{2}} \rho_2]) \otimes \pi_1 \otimes \s_{cusp}
\end{equation}
in the Jacquet module of $\s_{n-1}$ with respect to the appropriate parabolic subgroup.

Since we have that $\s_{n-1} \hookrightarrow \d([\nu^{-\frac{a_{n-1} -1}{2}} \rho_{n-1}, \nu^{\frac{b_{n-1} -1}{2}} \rho_{n-1}]) \rt \s_{n-2}$, using Lemma \ref{lemaulaganje}$(1)$  one can see that the multiplicity of (\ref{multpl}) in the Jacquet module of $\s_{n-1}$ with respect to the appropriate parabolic subgroup equals the multiplicity of
\begin{equation*}
\d([\nu^{-\frac{a_{n-2} -1}{2}} \rho_{n-2}, \nu^{\frac{b_{n-2} -1}{2}} \rho_{n-2}]) \otimes \cdots \otimes \d([\nu^{-\frac{a_2 -1}{2}} \rho_2, \nu^{\frac{b_2 -1}{2}} \rho_2]) \otimes \pi_1 \otimes \s_{cusp}
\end{equation*}
in the Jacquet module of $\s_{n-2}$ with respect to the appropriate parabolic subgroup. A repeated application of this procedure shows that the multiplicity of (\ref{multpl}) in the Jacquet module of $\s_{n-1}$ with respect to the appropriate parabolic subgroup equals the multiplicity of $\pi_1 \ot \s_{cusp}$ in $\mu^{\ast}(\s_1)$. Consequently, the multiplicity of $\pi_2 \ot \s_{cusp}$ in $\mu^{\ast}(\s_{n-1})$ equals one.

Note that $\d([\nu^{-\frac{a\_ -1}{2}} \rho, \nu^{\frac{a -1}{2}} \rho]) \times \pi_2$ is irreducible. 

\begin{itemize}
    \item Claim $2$: $\d([\nu^{x} \rho, \nu^{y} \rho]) \rt \s_{cusp}$ is irreducible for $x > 0$.
\end{itemize}

Since $x > 0$, an irreducible tempered subquotient of $\d([\nu^{x} \rho, \nu^{y} \rho]) \rt \s_{cusp}$ has to be strongly positive, since otherwise the Jacquet module of $\d([\nu^{x} \rho, \nu^{y} \rho]) \rt \s_{cusp}$ would contain an irreducible constituent of the form $\d([\nu^{z_1} \rho', \nu^{z_2} \rho']) \otimes \pi$ where $z_1 \leq 0$ and $z_1 + z_2 \geq 0$, which is impossible. But, since $\rho \rtimes \sigma_{cusp}$ reduces, it follows from \cite[Theorem~A]{Kim1} that there are no twists of $\rho$ appearing in the cuspidal support of strongly positive discrete series with the partial cuspidal support isomorphic to $\sigma_{cusp}$. Consequenty, there are no irreducible tempered subquotients of $\d([\nu^{x} \rho, \nu^{y} \rho]) \rt \s_{cusp}$. In the same way as in the proof of Theorem \ref{dissub} we deduce that every irreducible non-tempered subquotient of $\d([\nu^{x} \rho, \nu^{y} \rho]) \rt \s_{cusp}$ is isomorphic to $L(\d([\nu^{-y} \rho, \nu^{-x} \rho]), \s_{cusp})$. It follows that $\d([\nu^{x} \rho, \nu^{y} \rho]) \rt \s_{cusp}$ is irreducible and the Claim $2$ is proved.
 
Thus, we have $\d([\nu^{x} \rho, \nu^{y} \rho]) \rt \s_{cusp} \cong \d([\nu^{-y} \rho, \nu^{-x} \rho]) \rt \s_{cusp}$. In this way we obtain an embedding
\begin{dmath*}
\s \hookrightarrow \pi_2 \times \d([ \rho, \nu^{\frac{a-1}{2}} \rho]) \times \d([\nu \rho, \nu^{\frac{a\_ -1}{2}} \rho]) \rt \s_{cusp}
\cong \pi_2 \times \d([\nu \rho, \nu^{\frac{a\_ -1}{2}} \rho]) \times \d([ \rho, \nu^{\frac{a-1}{2}} \rho]) \rt \s_{cusp}
\hookrightarrow
\pi_2 \times \d([\nu \rho, \nu^{\frac{a\_ -1}{2}} \rho]) \times \d([\nu \rho, \nu^{\frac{a -1}{2}} \rho]) \times \rho \rt \s_{cusp}.
\end{dmath*}
The induced representation $\pi_2 \times \d([\nu \rho, \nu^{\frac{a\_ -1}{2}} \rho])$ is irreducible since no twists of $\rho$ appear in the cuspidal supports of $\pi_2$, and we denote it by $\pi$.
Frobenius reciprocity and Lemma \ref{lemajantz} imply that the Jacquet module of $\s$ with respect to the appropriate parabolic subgroup contains an irreducible representation of the form $\pi \otimes \d([\nu \rho, \nu^{\frac{a -1}{2}} \rho]) \otimes \tau$, where $\tau$ is an irreducible representation such that $\mu^{\ast}(\tau) \geq \rho \ot \s_{cusp}$. Thus, there is an $i \in \{ 1, -1 \}$ such that the Jacquet module of $\s$ with respect to the appropriate parabolic subgroup contains $\pi \otimes \d([\nu \rho, \nu^{\frac{a -1}{2}} \rho]) \otimes \tau_{i}$.

\begin{itemize}
\begin{sloppypar}
\item Claim $3$: $\pi \otimes \d([\nu \rho, \nu^{\frac{a -1}{2}} \rho]) \otimes \tau_{i}$ appears with multiplicity one in the Jacquet module of $\d([\nu^{-\frac{a\_ -1}{2}} \rho, \nu^{\frac{a -1}{2}} \rho]) \rt \s_{n-1}$ with respect to the appropriate    parabolic subgroup.
\end{sloppypar}
\end{itemize}

Let us prove the Claim $3$. The transitivity of Jacquet modules implies that there exists an irreducible constituent $\pi \otimes \tau'$ of $\mu^{\ast}(\d([\nu^{-\frac{a\_ -1}{2}} \rho, \nu^{\frac{a -1}{2}} \rho]) \rt \s_{n-1})$ such that $\mu^{\ast}(\tau') \geq \d([\nu \rho, \nu^{\frac{a -1}{2}} \rho]) \otimes \tau_{i}$.

The structural formula implies that there are $i, j$ such that $-\frac{a\_ + 1}{2} \leq i \leq j \leq \frac{a -1}{2}$ and an irreducible constituent $\pi' \otimes \tau''$ of $\mu^{\ast}(\s_{n-1})$ such that
\begin{align*}
\pi & \leq \d([\nu^{-i} \rho, \nu^{\frac{a\_ - 1}{2}} \rho]) \times \d([\nu^{j+1} \rho, \nu^{\frac{a - 1}{2}} \rho]) \times \pi' \\
\intertext{and}
\tau' & \leq \d([\nu^{i+1} \rho, \nu^{j} \rho]) \rtimes \tau''.
\end{align*}
Since $\pi \cong \pi_2 \times \d([\nu \rho, \nu^{\frac{a\_ - 1}{2}} \rho])$ and twists of $\rho$ do not appear in the cuspidal support of $\s_{n-1}$, it follows that $i = -1$, $j = \frac{a - 1}{2}$ and $\pi' \otimes \tau'' \cong \pi_2 \otimes \s_{cusp}$ by Claim 1. Also, since $\d([\nu \rho, \nu^{\frac{a -1}{2}} \rho]) \otimes \tau_{i}$ appears with multiplicity one in $\mu^{\ast}(\d([\rho, \nu^{\frac{a - 1}{2}} \rho]) \rt \s_{cusp})$, it follows that the multiplicity of $\pi \otimes \d([\nu \rho, \nu^{\frac{a -1}{2}} \rho]) \otimes \tau_{i}$ in the Jacquet module of $\d([\nu^{-\frac{a\_ -1}{2}} \rho, \nu^{\frac{a -1}{2}} \rho]) \rt \s_{n-1}$ with respect to the appropriate parabolic subgroup equals the multiplicity of $\pi_2 \otimes \s_{cusp}$ in $\mu^{\ast}(\s_{n-1})$, which equals one.

We denote by $\s'$ the irreducible subrepresentation of $\d([\nu^{-\frac{a\_ -1}{2}} \rho, \nu^{\frac{a -1}{2}} \rho]) \rt \s_{n-1}$ different than $\s$. Repeating the same arguments as before, one can deduce that there is an $i' \in \{ 1, -1 \}$ such that $\pi \otimes \d([\nu \rho, \nu^{\frac{a -1}{2}} \rho]) \otimes \tau_{i'}$ appears in the Jacquet module of $\s'$ with respect to the appropriate parabolic subgroup. Also, $\pi \otimes \d([\nu \rho, \nu^{\frac{a -1}{2}} \rho]) \otimes \tau_{i'}$ appears with multiplicity one in the Jacquet module of $\d([\nu^{-\frac{a\_ -1}{2}} \rho, \nu^{\frac{a -1}{2}} \rho]) \rt \s_{n-1}$ with respect to the appropriate parabolic subgroup. Thus, there is a unique $i \in \{ 1, -1 \}$ such that $\pi \otimes \d([\nu \rho, \nu^{\frac{a -1}{2}} \rho]) \otimes \tau_{i}$ appears in the Jacquet module of $\s$ with respect to the appropriate parabolic subgroup.

Suppose that $\mu^{\ast}(\s)$ contains an irreducible constituent of the form $\pi' \otimes \tau$, where $\tau$ is an irreducible subquotient of $\d([ \nu^{x} \rho, \nu^{\frac{a - 1}{2}} \rho]) \rt \s_{cusp}$ for $x > 0$. By the Claim $2$, the induced representation $\d([ \nu^{x} \rho, \nu^{\frac{a - 1}{2}} \rho]) \rt \s_{cusp}$ is then irreducible, so $\tau \cong \d([ \nu^{-\frac{a - 1}{2}} \rho, \nu^{-x} \rho]) \rt \s_{cusp}$. Thus, $\mu^{\ast}(\s) \geq \pi \otimes \d([ \nu^{-\frac{a - 1}{2}} \rho, \nu^{-x} \rho]) \rt \s_{cusp}$ and one can directly see that this contradicts the square-integrability of $\s$. This finishes the proof.
\end{proof}

\begin{lemma} \label{lemasolo}
Let $\s \in R(G)$ denote a discrete series and suppose that $\rho \in \Irr(GL)$ is a cuspidal unitarizable essentially self-dual representation such that $\rho \rt \s_{cusp}$ reduces and write $\rho \rt \s_{cusp} = \tau_1 + \tau_{-1}$, where representations $\tau_1$ and $\tau_{-1}$ are irreducible, tempered and mutually non-isomorphic. Suppose that $\Jord_{\rho}(\s) \neq \emptyset$, and denote the maximal element of $\Jord_{\rho}(\s)$ by $a_{\max}$. Then there is a unique $i \in \{ 1, -1 \}$ such that the Jacquet module of $\s$ with respect to an appropriate parabolic subgroup contains an irreducible constituent of the form $\pi \otimes \d([ \nu \rho, \nu^{\frac{a_{\max} - 1}{2}} \rho]) \otimes \tau_i$.
\end{lemma}
\begin{proof}
Since $\Jord_{\rho}(\s) \neq \emptyset$, in the beginning of the proof of Lemma \ref{lemabeforesolo} we have seen that $\sigma$ is non-strongly positive. Since for a strongly positive discrete series $\sigma_{sp}$ we have $\Jord_{\rho}(\sigma_{sp}) = \emptyset$, an inductive application of Theorem \ref{embed} implies that $\Jord_{\rho}(\s)$ consists of an even number of elements.

In the same way as in the proof of Theorem \ref{embed}, one can see that there is an ordered $n$-tuple of discrete series $(\s_1, \s_2, \ldots$, $\s_n)$, $\s_i \in \Irr(G_{n_i})$, such that $\s_n \cong \s$, $\s_1$ is as in the statement of the previous lemma and $\Jord_{\rho}(\s_1) = \{ a, a_{\max}\}$, and for every $i = 2, 3, \ldots, n$, there are $(a_i, \rho_i), (b_i, \rho_i) \in \Jord(\s)$ such that $a_i = (b_i)\_ \in \Jord_{\rho_i}(\s_i)$,
\begin{equation*}
\s_i \hookrightarrow \d([\nu^{-\frac{a_i -1}{2}} \rho_i, \nu^{\frac{b_i -1}{2}} \rho_i]) \rt \s_{i-1},
\end{equation*}
and
\begin{itemize}
\item if $\rho_i \cong \rho_j$ for some $j > i$ and $b_i$ is even, then $a_i > a_j$,
\item if $\rho_i \cong \rho_j$ for some $j > i$ and $b_i$ is odd, then $b_i < b_j$.
\end{itemize}
Thus, we have an embedding
\begin{equation*}
\s \hookrightarrow \d([\nu^{-\frac{a_n -1}{2}} \rho_n, \nu^{\frac{b_n -1}{2}} \rho_n]) \times \cdots \times \d([\nu^{-\frac{a_2 -1}{2}} \rho_2, \nu^{\frac{b_2 -1}{2}} \rho_2]) \rt \s_{1},
\end{equation*}
and
\begin{equation*}
\s_1 \hookrightarrow \d([\nu^{-\frac{a -1}{2}} \rho, \nu^{\frac{a_{\max} -1}{2}} \rho]) \rt \s',
\end{equation*}
where $\s'$ is a discrete series and $\Jord_{\rho}(\s') = \emptyset$.

By the previous lemma, there is a unique $j \in \{ 1, -1 \}$ such that the Jacquet module of $\s_1$ with respect to an appropriate parabolic subgroup contains an irreducible constituent of the form $\pi' \otimes \d([ \nu \rho, \nu^{\frac{a_{\max} - 1}{2}} \rho]) \otimes \tau_j$.

Similarly as in the proof of the previous lemma, we obtain that there is an $i \in \{ 1, -1 \}$ such that the Jacquet module of $\s$ with respect to an appropriate parabolic subgroup contains an irreducible constituent of the form $\pi \otimes \d([ \nu \rho, \nu^{\frac{a_{\max} - 1}{2}} \rho]) \otimes \tau_i$. Thus, there is an irreducible constituent $\pi \otimes \tau$ of $\mu^{\ast}(\s)$ such that $\mu^{\ast}(\tau) \geq \d([ \nu \rho, \nu^{\frac{a_{\max} - 1}{2}} \rho]) \otimes \tau_i$.

Using the structural formula, we obtain that there is an irreducible constituent $\pi' \otimes \tau'$ of $\mu^{\ast}(\s_1)$ such that $\tau \leq \pi'' \rt \tau'$, for some irreducible representation $\pi''$. Since $\nu^{\frac{a_{\max}-1}{2}} \rho$ appears in the cuspidal support of $\tau$ and it appears neither in the cuspidal support of $\d([\nu^{-\frac{a_k -1}{2}} \rho_k, \nu^{\frac{b_k -1}{2}} \rho_k])$, for $k = 2, \ldots, n$, nor in the cuspidal support of $\s'$ (since $\Jord_{\rho}(\sigma')=\emptyset$), we get that $\tau'$ is an irreducible subquotient of $\d([ \nu^{x} \rho, \nu^{\frac{a_{\min} - 1}{2}} \rho]) \rt \s_{cusp}$ for $x$ such that $0 \leq x \leq \frac{a_{\min} - 1}{2}$. Now the previous lemma implies that $x = 0$ and that $\tau \cong \tau'$ is an irreducible subquotient of $\d([ \rho, \nu^{\frac{a_{\min} - 1}{2}} \rho]) \rt \s_{cusp}$. The transitivity of Jacquet modules yields that the Jacquet module of $\s_1$ contains an irreducible constituent of the form $\pi' \otimes \d([ \nu \rho, \nu^{\frac{a_{\max} - 1}{2}} \rho]) \otimes \tau_i$. Now the previous lemma implies $i = j$, and the lemma is proved.
\end{proof}

\begin{definition} \label{deftwo}
Let $\s$ denote a discrete series. We additionally define $\epsilon_{\s}$ on certain elements of $\Jord(\s)$:
\begin{enumerate}[(1)]
\item Suppose that $\Jord_{\rho}(\s)$ consists of even numbers and let $a_{\min}$ denote the minimal element of $\Jord_{\rho}(\s)$. We define $\epsilon_{\s}(a_{\min}, \rho)$ so as to satisfy $\epsilon_{\s}(a_{\min}, \rho) = 1$ if $\mu^{\ast}(\s)$ contains an irreducible constituent of the form $\d([ \nu^{\frac{1}{2}} \rho$, $\nu^{\frac{a_{\min}-1}{2}} \rho]) \otimes \pi$, and let $\epsilon_{\sigma}(a_{\min}, \rho) = -1$ otherwise. Also, for $a \in \Jord_{\rho}(\s)$ such that $a\_$ is defined, let $\epsilon_{\s}(a\_, \rho) \cdot \epsilon_{\s}(a, \rho) = \epsilon_{\s}((a\_, \rho), (a, \rho))$.
\item Suppose that $\Jord_{\rho}(\s)$ consists of odd numbers and $\rho \rt \sigma_{cusp}$ reduces. We denote by $\tau_1$ and $\tau_{-1}$ the irreducible, tempered, mutually non-isomorphic subrepresentations of $\rho \rt \sigma_{cusp}$. Let $a_{\max}$ denote the maximal element of $\Jord_{\rho}$.
    Let $\epsilon_{\sigma}(a_{\max}, \rho) = i$, where $i \in \{ 1, -1 \}$ is such that the Jacquet module of $\s$ with respect to an appropriate parabolic subgroup contains an irreducible constituent of the form $\pi \otimes \d([ \nu \rho, \nu^{\frac{a_{\max} - 1}{2}} \rho]) \otimes \tau_i$. Also, for $a \in \Jord_{\rho}(\s)$ such that $a\_$ is defined, define $\epsilon_{\s}(a_{\min}, \rho)$ so as to satisfy $\epsilon_{\s}(a\_, \rho) \cdot \epsilon_{\s}(a, \rho) = \epsilon_{\s}((a\_, \rho), (a, \rho))$.
\end{enumerate}
\end{definition}
In this way, we inductively define $\epsilon_{\s}$-function on single pairs.

\begin{lemma}  \label{lemasolotreca}
Suppose that $\rho \in \Irr(GL_{n_{\rho}})$ is a cuspidal unitarizable essentially self-dual representation such that $\rho \rtimes \s_{cusp}$ reduces. Let $a_{\max}$ denote the maximal element of $\Jord_{\rho}(\s)$. Suppose that $\epsilon_{\s}(((a_{\max})\_,\rho), (a_{\max},\rho)) = 1$ and let $\s'$ stand for a discrete series such that $\s$ is a subrepresentation of
\begin{equation*}
\d([\nu^{-\frac{(a_{\max})\_-1}{2}} \rho, \nu^{\frac{a_{\max}-1}{2}} \rho]) \rt \s'.
\end{equation*}
\begin{sloppypar}
Suppose that $\Jord_{\rho}(\s') \neq \emptyset$ and let $b_{\max}$ stands for the maximal element of $\Jord_{\rho}(\s')$, then $\epsilon_{\s}(a_{\max},\rho) \cdot \epsilon_{\s}((b_{\max},\rho),((a_{\max})\_,\rho)) = \epsilon_{\s'}(b_{\max},\rho)$, i.e., $\epsilon_{\s}(b_{\max}, \rho) = \epsilon_{\s'}(b_{\max}, \rho)$.
\end{sloppypar}
\end{lemma}
\begin{proof}
In the proof of Lemma \ref{lemasolo} we have seen that $\Jord_{\rho}(\s)$ consists of an even number of odd positive integers. We note that we have fixed a choice of irreducible tempered mutually non-isomorphic representations $\tau_1$ and $\tau_{-1}$ such that $\rho \rt \s_{cusp} = \tau_1 + \tau_{-1}$.

Let $\pi$ denote a discrete series subrepresentation of the induced representation $\d([\nu^{-\frac{(a_{\max})\_-1}{2}} \rho$, $\nu^{\frac{a_{\max}-1}{2}} \rho]) \rt \s'$. First, using a repeated application of Lemma \ref{lemaulaganje} for irreducible self-dual cuspidal representations non-isomorphic to $\rho$ whose twists appear in the cuspidal support of $\pi$ we obtain an ordered $k$-tuple $(\s'_1, \ldots, \s'_k)$ of discrete series in $R(G)$ such that $\pi \cong \sigma'_k$ and for every $i = 2, \ldots, k$ there is an irreducible representation $\d'_i \in R(GL)$ such that $\s'_{i} \hookrightarrow \d'_i \rtimes \s'_{i-1}$, and $\mu^{\ast}(\s'_i)$ contains $\d'_i \otimes \s'_{i-1}$ with multiplicity one. Here, note that for an irreducible essentially self-dual cuspidal representation $\rho'$, $\rho' \not\cong \rho$, we first use Lemma \ref{lemaulaganje}(1) as many times as possible, and then we use Lemma \ref{lemaulaganje}(2). Consequently, in such a way we end with a discrete series $\s'_1$ whose cuspidal support contains only twists of $\rho$ and the partial cuspidal support $\s_{cusp}$.

We further apply Lemma \ref{lemaulaganje} several times, together with Theorem \ref{embed}, so that there is an ordered $m$-tuple $(\s_1, \ldots, \s_m)$ of discrete series in $R(G)$ such that $\pi \cong \sigma_m$ and for every $i = 2, \ldots, m$ there is an irreducible representation $\d_i \in R(GL)$ such that $\s_{i} \hookrightarrow \d_i \rtimes \s_{i-1}$, $\mu^{\ast}(\s_i)$ contains $\d_i \otimes \s_{i-1}$ with multiplicity one, $\s_1$ is a subrepresentation of $\d([\nu^{-\frac{(a_{\max})\_-1}{2}} \rho, \nu^{\frac{a_{\max}-1}{2}} \rho]) \rt \s_{cusp}$ and $\d_2 \cong \d([\nu^{-\frac{c-1}{2}} \rho, \nu^{\frac{b_{\max}-1}{2}} \rho])$, for $c$ such that in $\Jord_{\rho}(\s_2)$ we have $(b_{\max})\_ = c$. Note that we can take $\sigma_j \cong \sigma'_{j+k-m}$ and $\delta_j \cong \delta'_{j+k-m}$ for $j = m-k+1 \ldots, m$. Also, if $\nu^{x} \rho$ appears in the cuspidal support of some $\d_i$ for $i \in \{ 2, \ldots, m \}$, then $- \frac{(a_{\max})\_-3}{2} \leq x \leq \frac{(a_{\max})\_-3}{2}$. We note that this also implies that the Jacquet module of $\s'$ with respect to the appropriate parabolic subgroup contains
\begin{equation} \label{jmcetvrti}
\d_{m} \otimes \d_{m-1} \otimes \cdots \otimes \d_3 \otimes \s_2'',
\end{equation}
where $\s_2''$ is an irreducible subrepresentation of $\d_2 \rtimes \s_{cusp}$.

Lemma \ref{lemabeforesolo} and the structural formula imply that there is an $i \in \{ 1, -1 \}$ such  that the Jacquet module of $\sigma_1$ with respect to the appropriate parabolic subgroup contains $\d([\nu \rho, \nu^{\frac{(a_{\max})\_-1}{2}} \rho]) \otimes \d([\nu \rho, \nu^{\frac{a_{\max}-1}{2}} \rho]) \otimes \tau_i$. 

Thus, the Jacquet module of $\pi$ with respect to the appropriate parabolic subgroup contains
\begin{equation} \label{jmprvi}
\d_{m} \otimes \d_{m-1} \otimes \cdots \otimes \d_2 \otimes \d([\nu \rho, \nu^{\frac{(a_{\max})\_-1}{2}} \rho]) \otimes \d([\nu \rho, \nu^{\frac{a_{\max}-1}{2}} \rho]) \otimes \tau_i.
\end{equation}
It can be easily seen that the multiplicity of (\ref{jmprvi}) in the Jacquet module of $\d([\nu^{-\frac{(a_{\max})\_-1}{2}} \rho, \nu^{\frac{a_{\max}-1}{2}} \rho]) \rt \s'$ with respect to the appropriate parabolic subgroup equals the multiplicity of $\d([\nu \rho, \nu^{\frac{(a_{\max})\_-1}{2}} \rho]) \otimes \d([\nu \rho, \nu^{\frac{a_{\max}-1}{2}} \rho]) \otimes \tau_i$ in the Jacquet module of $\d([\nu^{-\frac{(a_{\max})\_-1}{2}} \rho, \nu^{\frac{a_{\max}-1}{2}} \rho]) \rt \s_{cusp}$ with respect to the appropriate parabolic subgroup, and it is a direct consequence of the structural formula that such a multiplicity equals one.

This proves the first claim of the proof:
\begin{sloppypar}
\begin{itemize}
    \item Claim $1$: For every discrete series subrepresentation of $\d([\nu^{-\frac{(a_{\max})\_-1}{2}} \rho$, $\nu^{\frac{a_{\max}-1}{2}} \rho]) \rt \s'$ there is a unique $i \in \{ 1, - 1 \}$ such that its Jacquet module with respect to the appropriate parabolic subgroup contains (\ref{jmprvi}).
\end{itemize}
\end{sloppypar}
We denote by $\pi_i$ a unique discrete series subrepresentation of $\d([\nu^{-\frac{(a_{\max})\_-1}{2}} \rho$, $\nu^{\frac{a_{\max}-1}{2}} \rho]) \rt \s'$ whose Jacquet module with respect to the appropriate parabolic subgroup contains (\ref{jmprvi}). Note that the transitivity of Jacquet modules implies $\epsilon_{\pi_i}(a_{\max}, \rho) = i$.

Let $i_1 \in \{ 1, -1 \}$ be such that $\pi_{i_1}$ is an irreducible subrepresentation of the induced representation $\d([\nu^{-\frac{(a_{\max})\_-1}{2}} \rho, \nu^{\frac{a_{\max}-1}{2}} \rho]) \rt \s'$ such that 
\begin{equation*}
\epsilon_{\pi_{i_1}}((b_{\max},\rho),((a_{\max})\_,\rho)) = 1.     
\end{equation*}
From Claims $1$, $2$ and $4$ from the proof of Theorem \ref{prepairemain} follows that 
\begin{equation*}
\epsilon_{\pi_{-i_1}}((b_{\max},\rho),((a_{\max})\_,\rho)) = -1.     
\end{equation*}

\begin{itemize}
    \item Claim $2$: $\epsilon_{\pi_{i_1}}(a_{\max}, \rho) = \epsilon_{\s'}(b_{\max}, \rho)$.
\end{itemize}
Let us prove the Claim $2$. For simplicity of the notation, let $j = \epsilon_{\s'}(b_{\max}, \rho)$.

To $\pi_{i_1}$ we attach an ordered $m$-tuple $(\sigma_1, \sigma_2, \ldots, \sigma_m)$, as in the first part of the proof. Using the construction provided in the proof of the previous lemma, we obtain $\epsilon_{\s_2}((b_{\max},\rho),((a_{\max})\_,\rho)) = 1$ and deduce that the Jacquet module of $\d([\nu^{-\frac{(a_{\max})\_-1}{2}} \rho, \nu^{\frac{a_{\max}-1}{2}} \rho]) \rt \s'$ with respect to the appropriate parabolic subgroup contains
\begin{equation*}
\d_{m} \otimes \cdots \otimes \d_3 \otimes \d([\nu^{- \frac{b_{\max} - 1}{2}} \rho, \nu^{\frac{(a_{\max})\_-1}{2}} \rho]) \otimes \d([\nu \rho, \nu^{\frac{c-1}{2}} \rho]) \otimes \d([\nu \rho, \nu^{\frac{a_{\max}-1}{2}} \rho]) \otimes \tau_{i_1}.
\end{equation*}
Using (\ref{jmcetvrti}) and the definition of the ordered $m$-tuple $(\s_1, \ldots, \s_m)$, we get that
\begin{equation*}
\d([\nu^{- \frac{b_{\max} - 1}{2}} \rho, \nu^{\frac{(a_{\max})\_-1}{2}} \rho]) \otimes \d([\nu \rho, \nu^{\frac{c-1}{2}} \rho]) \otimes \d([\nu \rho, \nu^{\frac{a_{\max}-1}{2}} \rho]) \otimes \tau_{i_1}
\end{equation*}
is contained in the Jacquet module of $\d([\nu^{-\frac{(a_{\max})\_-1}{2}} \rho, \nu^{\frac{a_{\max}-1}{2}} \rho]) \rt \s_{2}''$ with respect to the appropriate parabolic subgroup. Consequently, $\mu^{\ast}(\d([\nu^{-\frac{(a_{\max})\_-1}{2}} \rho$, $\nu^{\frac{a_{\max}-1}{2}} \rho]) \rt \s_{2}'')$ contains an irreducible constituent of the form $\d([\nu^{- \frac{b_{\max} - 1}{2}} \rho$, $\nu^{\frac{(a_{\max})\_-1}{2}} \rho]) \otimes \pi'$, for $\pi'$ such that the Jacquet module of $\pi'$ with respect to the appropriate parabolic subgroup contains $\d([\nu \rho, \nu^{\frac{c-1}{2}} \rho]) \otimes \d([\nu \rho, \nu^{\frac{a_{\max}-1}{2}} \rho]) \otimes \tau_{i_1}$. The structural formula implies that there are $l_1, l_2$, $-\frac{(a_{\max})\_+1}{2} \leq l_1 \leq l_2 \leq \frac{a_{\max}-1}{2}$, and an irreducible constituent $\d \otimes \pi''$ of $\mu^{\ast}(\s_2'')$ such that
\begin{align*}
\d([\nu^{- \frac{b_{\max} - 1}{2}} \rho, \nu^{\frac{(a_{\max})\_-1}{2}} \rho]) & \leq  \d([\nu^{-l_1} \rho, \nu^{\frac{(a_{\max})\_-1}{2}} \rho]) \times \d([\nu^{l_2 + 1} \rho, \nu^{\frac{a_{\max}-1}{2}} \rho]) \times \d
\intertext{and}
\pi' & \leq \d([\nu^{l_1 + 1} \rho, \nu^{l_2} \rho]) \rtimes \pi''.
\end{align*}
\begin{sloppypar}
Obviously, $l_2 = \frac{a_{\max}-1}{2}$. Since $\s_2''$ is a discrete series subrepresentation of $\d([\nu^{-\frac{c-1}{2}} \rho, \nu^{\frac{b_{\max}-1}{2}} \rho]) \rtimes \s_{cusp}$, we deduce at once that $\nu^{- \frac{b_{\max} - 1}{2}} \rho$ is not contained in the cuspidal support of $\d$. Consequently,
$l_1 = \frac{b_{\max} - 1}{2}$ and $\pi'' \cong \s_2''$. It follows that
\end{sloppypar}
\begin{equation*}
\d([\nu \rho, \nu^{\frac{c-1}{2}} \rho]) \otimes \d([\nu \rho, \nu^{\frac{a_{\max}-1}{2}} \rho]) \otimes \tau_{i_1}
\end{equation*}
is contained in the Jacquet module of $\d([\nu^{\frac{b_{\max} + 1}{2}} \rho, \nu^{\frac{a_{\max}-1}{2}} \rho]) \rt \s_{2}''$ with respect to the appropriate parabolic subgroup. From the cuspidal support of $\s_2''$ and (\ref{jmcetvrti}), we deduce that the Jacquet module of $\s'$ with respect to the appropriate parabolic subgroup contains
\begin{equation*}
\d_{m-1} \otimes \d_{m-2} \otimes \cdots \otimes \d_3 \otimes \d([\nu \rho, \nu^{\frac{c-1}{2}} \rho]) \otimes \d([\nu \rho, \nu^{\frac{b_{\max}-1}{2}} \rho]) \otimes \tau_{i_1}.
\end{equation*}
Now the transitivity of Jacquet modules yields $i_1 = j$, and the Claim $2$ is proven. 

Consequently, $\epsilon_{\pi_{-i_1}}(a_{\max}, \rho) = -\epsilon_{\s'}(b_{\max}, \rho)$, and for a discrete series subrepresentation $\s$ of $\d([\nu^{-\frac{(a_{\max})\_-1}{2}} \rho$, $\nu^{\frac{a_{\max}-1}{2}} \rho]) \rt \s'$ we have 
\begin{equation*}
\epsilon_{\s}(a_{\max}, \rho) \cdot \epsilon_{\s}((b_{\max},\rho),((a_{\max})\_,\rho)) = \epsilon_{\s'}(b_{\max}, \rho).     
\end{equation*}
This finishes the proof.
\end{proof}

From the proof of Lemma \ref{lemasolotreca}, one also obtains the following result:

\begin{corollary} \label{cornonhalf}
Suppose that $\rho \in \Irr(GL_{n_{\rho}})$ is a cuspidal unitarizable essentially self-dual representation such that $\rho \rt \s_{cusp}$ reduces and let us write $\rho \rt \s_{cusp} = \tau_1 + \tau_{-1}$, where representations $\tau_1$ and $\tau_{-1}$ are irreducible tempered and mutually non-isomorphic. Let $\s'$ denote a discrete series such that $\Jord_{\rho}(\s') = \emptyset$. For odd positive integers $a$ and $b$, such that $a < b$, and $i \in \{ 1, -1 \}$ there is a unique irreducible discrete series subrepresentation $\s$ of
\begin{equation*}
\d([\nu^{-\frac{a-1}{2}} \rho, \nu^{\frac{b-1}{2}} \rho]) \rt \s'
\end{equation*}
such that $\epsilon_{\s}(b, \rho) = i$.
\end{corollary}

In the same way as in the proof of Lemma \ref{lemaprva} we obtain the following result. 

\begin{lemma} \label{lemaembed}
Suppose that $\Jord_{\rho}(\s)$ consists of even numbers and let $a_{\min}$ denote the minimal element of $\Jord_{\rho}(\s)$. Then $\epsilon_{\s}(a_{\min}, \rho) = 1$ if and only if $\s$ is a subrepresentation of an induced representation of the form $\d([ \nu^{\frac{1}{2}} \rho, \nu^{\frac{a_{\min}-1}{2}} \rho]) \rtimes \pi$, for an irreducible representation $\pi$.
\end{lemma}

\begin{lemma} \label{lemahalf}
Suppose that $\rho \in \Irr(GL_{n_{\rho}})$ is a cuspidal unitarizable essentially self-dual representation such that $\nu^{\frac{1}{2}} \rho \rt \s_{cusp}$ reduces. Let $\s'$ denote a discrete series such that $\Jord_{\rho}(\s') = \emptyset$. For positive half-integers $a$ and $b$, such that $a < b$, there is a unique irreducible subrepresentation of
\begin{equation*}
\d([\nu^{-a} \rho, \nu^{b} \rho]) \rt \s'
\end{equation*}
which contains an irreducible constituent of the form $\d([\nu^{\frac{1}{2}} \rho, \nu^{a} \rho]) \otimes \pi$ in its Jacquet module with respect to the appropriate parabolic subgroup.
\end{lemma}
\begin{proof}
In Theorem \ref{dissub} we have seen that the induced representation $\d([\nu^{-a} \rho$, $\nu^{b} \rho]) \rt \s'$ contains two irreducible subrepresentations which are mutually non-isomorphic and square-integrable, let us denote them by $\s_1$ and $\s_2$. For $i = 1, 2$, there is a unique irreducible tempered subrepresentation $\tau_i$ of $\d([\nu^{-a} \rho, \nu^{a} \rho]) \rt \s'$ such that $\s_i$ is a subrepresentation of $\d([\nu^{a+1} \rho, \nu^{b} \rho]) \rt \tau_i$.

It follows from the structural formula and the description of $\Jord_{\rho}(\s')$ that the only irreducible constituent of the form $\d([\nu^{\frac{1}{2}} \rho, \nu^{a} \rho]) \times \d([\nu^{\frac{1}{2}} \rho, \nu^{a} \rho]) \otimes \pi$ appearing in $\mu^{\ast}(\d([\nu^{-a} \rho, \nu^{a} \rho]) \rt \s')$ is
\begin{equation*}
\d([\nu^{\frac{1}{2}} \rho, \nu^{a} \rho]) \times \d([\nu^{\frac{1}{2}} \rho, \nu^{a} \rho]) \otimes \s',
\end{equation*}
which appears there with multiplicity one. Thus, there is exactly one $i \in \{ 1, 2 \}$ such that $\mu^{\ast}(\tau_i) \geq \d([\nu^{\frac{1}{2}} \rho, \nu^{a} \rho]) \times \d([\nu^{\frac{1}{2}} \rho, \nu^{a} \rho]) \otimes \s'$, and we denote it by $i_2$. 

Following the same lines as in the proof of the Claim $4$ from Theorem \ref{prepairemain}, we deduce that then $\mu^{\ast}(\s_{i_1}) \geq \d([\nu^{\frac{1}{2}} \rho, \nu^{a} \rho]) \otimes \pi'$, for some irreducible representation $\pi'$.

Let $i_2 \in \{ 1, 2 \}$ such that $i_1 \neq i_2$. Suppose that $\mu^{\ast}(\s_{i_2})$ also contains an irreducible constituent of the form $\d([\nu^{\frac{1}{2}} \rho, \nu^{a} \rho]) \otimes \pi_1$, for some irreducible representation $\pi_1$.

Since $\s_{i_2}$ is a subrepresentation of $\d([\nu^{\frac{1}{2}} \rho, \nu^{b} \rho]) \times \d([\nu^{-a} \rho, \nu^{-\frac{1}{2}} \rho]) \rt \s'$, by Lemma \ref{lemajantz} there is an irreducible subquotient $\pi_2$ of $\d([\nu^{-a} \rho, \nu^{-\frac{1}{2}} \rho]) \rt \s'$ such that $\s_{i_2}$ is a subrepresentation of $\d([\nu^{\frac{1}{2}} \rho, \nu^{b} \rho]) \rt \pi_2$. Obviously, $\mu^{\ast}(\s_{i_2}) \geq \d([\nu^{\frac{1}{2}} \rho, \nu^{b} \rho]) \otimes \pi_2$. Since $\mu^{\ast}(\s_{i_2}) \geq \d([\nu^{\frac{1}{2}} \rho, \nu^{a} \rho]) \otimes \pi_1$ and $a < b$, it follows from the structural formula that $\mu^{\ast}(\pi_2)$ contains an irreducible constituent of the form $\d([\nu^{\frac{1}{2}} \rho, \nu^{a} \rho]) \otimes \pi_3$. Thus, the Jacquet module of $\s_{i_2}$ with respect to the appropriate parabolic subgroup contains $\d([\nu^{\frac{1}{2}} \rho, \nu^{b} \rho]) \otimes \d([\nu^{\frac{1}{2}} \rho, \nu^{a} \rho]) \otimes \pi_3$.

So, there is an irreducible representation $\d$ such that $\mu^{\ast}(\s_{i_2}) \geq \d \otimes \pi_3$ and $m^{\ast}(\d) \geq \d([\nu^{\frac{1}{2}} \rho, \nu^{b} \rho]) \otimes \d([\nu^{\frac{1}{2}} \rho, \nu^{a} \rho])$. Since $\s_{i_2}$ is a subrepresentation of $\d([\nu^{-a} \rho, \nu^{b} \rho]) \rt \s'$, there are $i, j$ such that $-a-1 \leq i \leq j \leq b$ and an irreducible constituent $\d' \otimes \pi'$ of $\mu^{\ast}(\s')$ such that
\begin{align*}
\d & \leq \d([\nu^{-i} \rho, \nu^{a} \rho]) \times \d([\nu^{j+1} \rho, \nu^{b} \rho]) \times \d' \\
\intertext{and}
\pi_3 & \leq \d([\nu^{i+1} \rho, \nu^{j} \rho]) \rt \pi'.
\end{align*}
From $\Jord_{\rho}(\s') = \emptyset$ we obtain that $i = j = - \frac{1}{2}$ and $\pi_3 \cong \s'$. Thus, $\mu^{\ast}(\s_{i_2}) \geq \d([\nu^{\frac{1}{2}} \rho, \nu^{b} \rho]) \times \d([\nu^{\frac{1}{2}} \rho, \nu^{a} \rho]) \otimes \s'$. This leads to 
\begin{equation*}
\mu^{\ast}(\d([\nu^{a+1} \rho, \nu^{b} \rho]) \rt \tau_{i_2}) \geq \d([\nu^{\frac{1}{2}} \rho, \nu^{b} \rho]) \times \d([\nu^{\frac{1}{2}} \rho, \nu^{a} \rho]) \otimes \s',
\end{equation*}
and one readily sees that this gives $\mu^{\ast}(\tau_{i_2}) \geq \d([\nu^{\frac{1}{2}} \rho, \nu^{a} \rho]) \times \d([\nu^{\frac{1}{2}} \rho, \nu^{a} \rho]) \otimes \s'$, which is impossible. This proves the lemma.
\end{proof}

\begin{lemma}  \label{lemasolodruga}
Suppose that $\rho \in \Irr(GL_{n_{\rho}})$ is a cuspidal unitarizable essentially self-dual representation such that $\Jord_{\rho}(\s)$ consists of even integers and let $a_{\min}$ denote the minimal element of $\Jord_{\rho}(\s)$. Suppose that $\epsilon_{\s}((a_{\min},\rho), (a,\rho)) = 1$, where $a_{\min} = a\_$ in $\Jord_{\rho}(\s)$. Let $\s'$ stand for a discrete series such that $\s$ is a subrepresentation of
\begin{equation*}
\d([\nu^{-\frac{a_{\min}-1}{2}} \rho, \nu^{\frac{a-1}{2}} \rho]) \rt \s'.
\end{equation*}
If $\Jord_{\rho}(\s') \neq \emptyset$ and $b_{\min}$ stands for the minimal element of $\Jord_{\rho}(\s')$, then $\epsilon_{\s}(a_{\min},\rho) \cdot \epsilon_{\s}((a,\rho),(b_{\min},\rho)) = \epsilon_{\s'}(b_{\min},\rho)$, i.e., $\epsilon_{\s}(b_{\min}, \rho) = \epsilon_{\s'}(b_{\min}, \rho)$.
\end{lemma}
\begin{proof}
The proof is similar to the one of Theorem \ref{prepairemain}, and we divide it in a series of claims. First we note that in $R(G)$ we have
\begin{equation*}
\d([\nu^{-\frac{a_{\min}-1}{2}} \rho, \nu^{\frac{a_{\min}-1}{2}} \rho]) \rt \s' = \tau_1 + \tau_2,    
\end{equation*}
for mutually non-isomorphic irreducible tempered representations $\tau_1$ and $\tau_2$. The first claim is analogous to the Claim $1$ in the proof of Theorem \ref{prepairemain}.
\begin{itemize}
    \item Claim $1$: There exists a unique $\alpha \in \{ 1, 2 \}$ such that $\sigma$ is a subrepresentation of $\d([\nu^{\frac{a_{\min} + 1}{2}} \rho, \nu^{\frac{a-1}{2}} \rho]) \rtimes \tau_{\alpha}$.
\end{itemize}

\begin{itemize}
    \item Claim $2$: There exists a unique $\beta \in \{ 1, 2 \}$ such that $\mu^{\ast}(\tau_{\beta})$ contains 
\begin{equation} \label{constdva}
\d([\nu^{\frac{1}{2}} \rho, \nu^{\frac{a_{\min}-1}{2}} \rho]) \times \d([\nu^{\frac{1}{2}} \rho, \nu^{\frac{a_{\min}-1}{2}} \rho]) \otimes \s'.
\end{equation}
\end{itemize}
The Claim $2$ can be proved in the same way as the Claim $2$ in the proof of Theorem \ref{prepairemain}. Also, in the same way we obtain that (\ref{constdva}) is a unique irreducible constituent of $\mu^{\ast}(\d([\nu^{-\frac{a_{\min}-1}{2}} \rho, \nu^{\frac{a_{\min}-1}{2}} \rho]) \rt \s')$ of the form $\d([\nu^{\frac{1}{2}} \rho, \nu^{\frac{a_{\min}-1}{2}} \rho]) \times \d([\nu^{\frac{1}{2}} \rho, \nu^{\frac{a_{\min}-1}{2}} \rho]) \otimes \pi$, appears there with multiplicity one, and that $\tau_{\beta}$ is a subrepresentation of 
\begin{equation*} 
\d([\nu^{\frac{1}{2}} \rho, \nu^{\frac{a_{\min}-1}{2}} \rho]) \times \d([\nu^{\frac{1}{2}} \rho, \nu^{\frac{a_{\min}-1}{2}} \rho]) \rtimes \s'.
\end{equation*}

\begin{itemize}
    \item Claim $3$: There exists a unique $\gamma \in \{ 1, 2 \}$ such that $\mu^{\ast}(\tau_{\gamma})$ contains an irreducible constituent of the form
    \begin{equation*}  
    \delta([\nu^{\frac{a_{\min} + 1}{2}} \rho, \nu^{\frac{b_{\min} - 1}{2}} \rho]) \otimes \pi.
    \end{equation*}
\end{itemize}
The Claim $3$ can be proved in the same way as the Claim $3$ in the proof of Theorem \ref{prepairemain}. Also, $\pi \cong \d([\nu^{-\frac{a_{\min}-1}{2}} \rho, \nu^{\frac{a_{\min}-1}{2}} \rho]) \rtimes \sigma''$, for a discrete series $\sigma''$ such that $\s'$ is a subrepresentation of $\d([\nu^{\frac{a_{\min} + 1}{2}} \rho, \nu^{\frac{b_{\min}-1}{2}} \rho]) \rt \s''$. Furthermore, $\tau_{\gamma}$ is a subrepresentation of $\delta([\nu^{\frac{a_{\min} + 1}{2}} \rho, \nu^{\frac{b_{\min} - 1}{2}} \rho]) \rtimes \pi$.

\begin{itemize}
    \item Claim $4$: $\epsilon_{\sigma}(a_{\min}, \rho)=1$ if and only if $\alpha = \beta$.
\end{itemize}
Suppose that $\epsilon_{\s}(a_{\min}, \rho) = 1$. Using the definition of the $\epsilon$-function and Lemma \ref{lemajantz}, we obtain that there are irreducible representations $\pi_1$ and $\pi_2$ such that we have the following embeddings:
\begin{align*}
\s & \hookrightarrow \d([\nu^{\frac{1}{2}} \rho, \nu^{\frac{a-1}{2}} \rho]) \rt \pi_1 \\
\s & \hookrightarrow  \d([\nu^{\frac{1}{2}} \rho, \nu^{\frac{a_{\min}-1}{2}} \rho]) \rt \pi_2.
\end{align*}
Frobenius reciprocity gives $\mu^{\ast}(\s) \geq \d([\nu^{\frac{1}{2}} \rho, \nu^{\frac{a_{\min}-1}{2}} \rho]) \otimes \pi_2$ and it follows directly from the structural formula that $\mu^{\ast}(\pi_1)$ contains an irreducible constituent of the form $\d([\nu^{\frac{1}{2}} \rho, \nu^{\frac{a_{\min}-1}{2}} \rho]) \otimes \pi$. Since $\mu^{\ast}(\s)$ does not contain an irreducible constituent of the form $\nu^{x} \rho \otimes \pi'$ for $x < \frac{a_{\min}-1}{2}$, it follows that $\mu^{\ast}(\pi_1)$ also does not contain an irreducible constituent of such a form $\nu^{x} \rho \otimes \pi'$, since otherwise there would be an irreducible representation $\pi'_1$ such that $\pi_1$ is a subrepresentation of $\nu^{x} \rho \rt \pi'_1$, and we would have
\begin{align*}
\s & \hookrightarrow \d([\nu^{\frac{1}{2}} \rho, \nu^{\frac{a-1}{2}} \rho]) \rt \pi_1 \hookrightarrow \d([\nu^{\frac{1}{2}} \rho, \nu^{\frac{a-1}{2}} \rho]) \times \nu^{x} \rho \rt \pi'_1 \\
& \cong \nu^{x} \rho \times \d([\nu^{\frac{1}{2}} \rho, \nu^{\frac{a-1}{2}} \rho]) \rt \pi'_1,
\end{align*}
a contradiction.

Using Lemma \ref{lemaprva} we deduce that there is an irreducible representation $\pi''$ such that $\pi_1$ is a subrepresentation of $\d([\nu^{\frac{1}{2}} \rho, \nu^{\frac{a_{\min}-1}{2}} \rho]) \rt \pi''$. The irreducibility of $\d([\nu^{\frac{1}{2}} \rho, \nu^{\frac{a-1}{2}} \rho]) \times \d([\nu^{\frac{1}{2}} \rho, \nu^{\frac{a_{\min}-1}{2}} \rho])$, together with Frobenius reciprocity, shows that $\mu^{\ast}(\s)$ contains
\begin{equation} \label{constjedan}
\d([\nu^{\frac{1}{2}} \rho, \nu^{\frac{a-1}{2}} \rho]) \times \d([\nu^{\frac{1}{2}} \rho, \nu^{\frac{a_{\min}-1}{2}} \rho]) \otimes \pi''.
\end{equation}
Since the irreducible constituent (\ref{constjedan}) also appears in $\mu^{\ast}(\d([\nu^{\frac{a_{\min} + 1}{2}} \rho, \nu^{\frac{a-1}{2}} \rho]) \rt \tau_{\alpha})$, there are $i, j$ such that $\frac{a_{\min} - 1}{2} \leq i \leq j \leq \frac{a-1}{2}$ and an irreducible constituent $\d \otimes \pi''_2$ of $\mu^{\ast}(\tau_{\alpha})$ such that
\begin{equation*}
\d([\nu^{\frac{1}{2}} \rho, \nu^{\frac{a-1}{2}} \rho]) \times \d([\nu^{\frac{1}{2}} \rho, \nu^{\frac{a_{\min}-1}{2}} \rho]) \leq \d([\nu^{-i} \rho, \nu^{-\frac{a_{\min} + 1}{2}} \rho]) \times \d([\nu^{j+1} \rho, \nu^{\frac{a-1}{2}} \rho]) \times \d.
\end{equation*}
\begin{sloppypar}
\noindent It follows that $i = \frac{a_{\min} - 1}{2}$. Since $\tau_{\alpha}$ is a subrepresentation of $\d([\nu^{-\frac{a_{\min}-1}{2}} \rho$, $\nu^{\frac{a_{\min}-1}{2}} \rho]) \rtimes \s'$ and $a \not\in \Jord_{\rho}(\s')$, $\mu^{\ast}(\tau_{\alpha})$ does not contain an irreducible constituent of the form $\nu^{x} \rho \otimes \pi''_3$ such that $\frac{a_{\min}+1}{2} \leq x \leq \frac{a-1}{2}$. It follows that $j = \frac{a-1}{2}$ and, consequently,
\end{sloppypar}
\begin{equation*}
\d \cong \d([\nu^{\frac{1}{2}} \rho, \nu^{\frac{a_{\min}-1}{2}} \rho]) \times \d([\nu^{\frac{1}{2}} \rho, \nu^{\frac{a_{\min}-1}{2}} \rho]).
\end{equation*}
Thus, $\alpha = \beta$.

Conversely, if $\alpha = \beta$ in the same way as in the proof of the Claim $4$ of Theorem \ref{prepairemain} we obtain that $\epsilon_{\s}(a_{\min}, \rho) = 1$.

The following claim is analogous to the Claim $5$ in the proof of Theorem \ref{prepairemain}:
\begin{itemize}
    \item Claim $5$: $\epsilon_{\sigma}((a, \rho), (b_{\min}, \rho))=1$ if and only if $\alpha = \gamma$.
\end{itemize}

\begin{itemize}
    \item Claim $6$: $\epsilon_{\sigma'}((b_{\min}, \rho))=1$ if and only if $\beta = \gamma$.
\end{itemize}
Let $\sigma''$ denote a discrete series such that $\sigma'$ is a subrepresentation of $\d([\nu^{\frac{a_{\min}+1}{2}} \rho$, $\nu^{\frac{b_{\min}-1}{2}} \rho]) \rt \s''$, given by Lemma \ref{lemapripremazadnja}.

The induced representation $\d([\nu^{-\frac{a_{\min}-1}{2}} \rho, \nu^{\frac{a_{\min}-1}{2}} \rho]) \rt \s'$ is a subrepresentation of
\begin{equation*}
\d([\nu^{-\frac{a_{\min}-1}{2}} \rho, \nu^{\frac{a_{\min}-1}{2}} \rho]) \times \d([\nu^{\frac{a_{\min}+1}{2}} \rho, \nu^{\frac{b_{\min}-1}{2}} \rho]) \rt \s'',
\end{equation*}
and in $R(G)$ we have
\begin{gather*}
\d([\nu^{-\frac{a_{\min}-1}{2}} \rho, \nu^{\frac{a_{\min}-1}{2}} \rho]) \times \d([\nu^{\frac{a_{\min}+1}{2}} \rho, \nu^{\frac{b_{\min}-1}{2}} \rho]) \rt \s'' = \\
\d([\nu^{-\frac{a_{\min}-1}{2}} \rho, \nu^{\frac{b_{\min}-1}{2}} \rho]) \rt \s'' + \\
+ L(\d([\nu^{-\frac{a_{\min}-1}{2}} \rho, \nu^{\frac{a_{\min}-1}{2}} \rho]), \d([\nu^{\frac{a_{\min}+1}{2}} \rho, \nu^{\frac{b_{\min}-1}{2}} \rho])) \rt \s''.
\end{gather*}
We split the rest of the proof in three sub-claims. First of them can be proved in the same way as Sub-claim $1$ in the proof of Theorem \ref{prepairemain}.
\begin{itemize}
    \item Sub-claim $1$: Each of the induced representations 
\begin{gather*}
\d([\nu^{-\frac{a_{\min}-1}{2}} \rho, \nu^{\frac{b_{\min}-1}{2}} \rho]) \rt \s'' \\
\intertext{and}
L(\d([\nu^{-\frac{a_{\min}-1}{2}} \rho, \nu^{\frac{a_{\min}-1}{2}} \rho]), \d([\nu^{\frac{a_{\min}+1}{2}} \rho, \nu^{\frac{b_{\min}-1}{2}} \rho])) \rt \s''
\end{gather*}
contains exactly one irreducible tempered subrepresentation of 
\begin{equation*}
\delta([\nu^{-\frac{a- 1}{2}} \rho, \nu^{\frac{a-1}{2}} \rho]) \rtimes \s'.
\end{equation*}
\end{itemize}

The second sub-claim is analogous to Sub-claim $2$ in the proof of Theorem \ref{prepairemain}.
\begin{itemize}
    \item Sub-claim $2$: $\tau_{\gamma}$ is contained in $\d([\nu^{-\frac{a_{\min}-1}{2}} \rho, \nu^{\frac{b_{\min}-1}{2}} \rho]) \rt \s''$, but $\tau_{3-\gamma}$ is not.
\end{itemize}

\begin{itemize}
    \item Sub-claim $3$: $\epsilon_{\sigma'}(b_{\min}, \rho)=1$ if and only if $\tau_{\beta}$ is contained in $\d([\nu^{-\frac{a_{\min}-1}{2}} \rho$, $\nu^{\frac{b_{\min}-1}{2}} \rho]) \rt \s''$.
\end{itemize}
Assume that $\epsilon_{\s'}(b_{\min}, \rho) = 1$. It follows that there is an irreducible representation $\pi$ such that $\mu^{\ast}(\s') \geq \d([\nu^{\frac{1}{2}} \rho, \nu^{\frac{b_{\min}-1}{2}} \rho]) \otimes \pi$. Thus, the Jacquet module of $\tau_{\beta}$ with respect to the appropriate parabolic subgroup contains
\begin{equation*}
\d([\nu^{\frac{1}{2}} \rho, \nu^{\frac{a_{\min}-1}{2}} \rho]) \times \d([\nu^{\frac{1}{2}} \rho, \nu^{\frac{a_{\min}-1}{2}} \rho]) \otimes \d([\nu^{\frac{1}{2}} \rho, \nu^{\frac{b_{\min}-1}{2}} \rho]) \otimes \pi.
\end{equation*}
The transitivity of Jacquet modules now implies that there is an irreducible constituent $\delta_1 \otimes \pi_1$ of $\mu^{\ast}(\tau_{\beta})$ such that $m^{\ast}(\delta_1)$ contains
\begin{equation} \label{jmsesti}
\d([\nu^{\frac{1}{2}} \rho, \nu^{\frac{a_{\min}-1}{2}} \rho]) \times \d([\nu^{\frac{1}{2}} \rho, \nu^{\frac{a_{\min}-1}{2}} \rho]) \otimes \d([\nu^{\frac{1}{2}} \rho, \nu^{\frac{b_{\min}-1}{2}} \rho]).
\end{equation}
Since $\tau_{\beta}$ is a subrepresentation of $\d([\nu^{-\frac{a_{\min}-1}{2}} \rho, \nu^{\frac{a_{\min}-1}{2}} \rho]) \rt \s'$, using the structural formula we deduce that there are $i, j$ such that $-\frac{a_{\min}+1}{2} \leq i \leq j \leq \frac{a_{\min}-1}{2}$ and an irreducible constituent $\delta_2 \otimes \pi_2$ of $\mu^{\ast}(\s')$ such that
\begin{equation*}
\delta_1 \leq \d([\nu^{-i} \rho, \nu^{\frac{a_{\min}-1}{2}} \rho]) \times \d([\nu^{j} \rho, \nu^{\frac{a_{\min}-1}{2}} \rho]) \times \delta_2.
\end{equation*}
From the description of $\Jord(\s')$ we obtain that $m^{\ast}(\delta')$ does not contain an irreducible representation of the form $\nu^{x} \rho \otimes \pi''$ such that $x < \frac{b_{\min}-1}{2}$. Since $\frac{a_{\min}-1}{2} < \frac{b_{\min}-1}{2}$, we get $i = -\frac{1}{2}$ and $j = \frac{1}{2}$. Using (\ref{jmsesti}) we deduce that $\delta_2 \cong \d([\nu^{\frac{1}{2}} \rho, \nu^{\frac{b_{\min}-1}{2}} \rho])$.

Consequently, $\delta_1$ is isomorphic to
\begin{gather*}
\d([\nu^{\frac{1}{2}} \rho, \nu^{\frac{a_{\min}-1}{2}} \rho]) \times \d([\nu^{\frac{1}{2}} \rho, \nu^{\frac{a_{\min}-1}{2}} \rho]) \times \d([\nu^{\frac{1}{2}} \rho, \nu^{\frac{b_{\min}-1}{2}} \rho]) \cong \\
\cong \d([\nu^{\frac{1}{2}} \rho, \nu^{\frac{b_{\min}-1}{2}} \rho]) \times \d([\nu^{\frac{1}{2}} \rho, \nu^{\frac{a_{\min}-1}{2}} \rho]) \times \d([\nu^{\frac{1}{2}} \rho, \nu^{\frac{a_{\min}-1}{2}} \rho]).
\end{gather*}
It follows that $\mu^{\ast}(\tau_{\beta})$ contains an irreducible constituent of the form $\d([\nu^{\frac{a_{\min}+1}{2}} \rho$, $\nu^{\frac{b_{\min}-1}{2}} \rho]) \otimes \pi$. Claim $3$ implies $\beta = \gamma$, and by the previous sub-claim $\tau_{\beta}$ is contained in $\d([\nu^{-\frac{a_{\min}-1}{2}} \rho$, $\nu^{\frac{b_{\min}-1}{2}} \rho]) \rt \s''$.

Conversely, let us assume that $\tau_{\beta}$ is contained in $\d([\nu^{-\frac{a_{\min}-1}{2}} \rho$, $\nu^{\frac{b_{\min}-1}{2}} \rho]) \rt \s''$. By the structural formula, there are $i, j$ such that $-\frac{a_{\min}+1}{2} \leq i \leq j \leq \frac{b_{\min}-1}{2}$ and an irreducible constituent $\d_1 \otimes \pi_1$ of $\mu^{\ast}(\s'')$ such that
\begin{gather*}
\d([\nu^{\frac{1}{2}} \rho, \nu^{\frac{a_{\min}-1}{2}} \rho]) \times \d([\nu^{\frac{1}{2}} \rho, \nu^{\frac{a_{\min}-1}{2}} \rho]) \leq \\
\leq \d([\nu^{-i} \rho, \nu^{\frac{a_{\min}-1}{2}} \rho]) \times \d([\nu^{j+1} \rho, \nu^{\frac{b_{\min}-1}{2}} \rho]) \times \d_1.
\end{gather*}
Since $a_{\min} < b_{\min}$, it follows at once that $j = \frac{b_{\min}-1}{2}$ and that $\mu^{\ast}(\s'')$ contains an irreducible constituent of the form $\d([\nu^{\frac{1}{2}} \rho, \nu^{\frac{a_{\min}-1}{2}} \rho]) \otimes \pi_1$ and, consequently, that the Jacquet module of $\s'$ with respect to the appropriate parabolic subgroup contains
\begin{equation*}
\d([\nu^{\frac{a_{\min} + 1}{2}} \rho, \nu^{\frac{b_{\min}-1}{2}} \rho]) \otimes \d([\nu^{\frac{1}{2}} \rho, \nu^{\frac{a_{\min}-1}{2}} \rho]) \otimes \pi. 
\end{equation*}
Now in the same way as in the proof of the Claim $5$ from the proof of Theorem \ref{prepairemain} one can see that there is an irreducible representation $\pi_2$ such that $\s'$ is a subrepresentation of $\d([\nu^{\frac{1}{2}} \rho, \nu^{\frac{b_{\min}-1}{2}} \rho]) \rt \pi_2$. Thus, $\epsilon_{\s'}(b_{\min}, \rho) = 1$. This ends the proof.
\end{proof}

We now prove the main result of this section.

\begin{thrm} \label{propprva}
Let $\s \in R(G)$ denote a non-strongly positive discrete series. Let $(a, \rho) \in \Jord(\s)$ be such that $a\_$ is defined, $\epsilon_{\s}((a\_, \rho), (a, \rho)) = 1$ and let $\s'$ stand for a discrete series such that $\s$ is a subrepresentation of $\delta([\nu^{-\frac{a\_ - 1}{2}} \rho, \nu^{\frac{a-1}{2}} \rho]) \rtimes \s'$. If $\epsilon_{\s'}(b, \rho')$ is defined for some $(b, \rho') \in \Jord(\s')$, then $\epsilon_{\s'}(b, \rho') = \epsilon_{\s}(b, \rho')$.
\end{thrm}
\begin{proof}
Let $(b, \rho')$ denote an element of $\Jord(\s')$ such that $\epsilon_{\s'}(b, \rho')$ is defined. There are two possibilities to consider:
\begin{enumerate}[(1)]
\item Suppose that $b$ is odd. Let us write $\rho' \rt \s_{cusp} = \tau_{1} + \tau_{-1}$, where $\tau_{1}$ and $\tau_{-1}$ are irreducible tempered mutually non-isomorphic representations. Also, we denote by $b_{\max}$ the maximal element of $\Jord_{\rho'}(\s)$. If $(b_{\max}, \rho')$ does not appear in $\Jord(\s')$, the claim of the theorem follows from Lemma \ref{lemasolotreca}, Theorem \ref{prepairemain} and Definition \ref{deftwo}. Suppose that $(b_{\max}, \rho') \in \Jord(\s')$. By Lemma \ref{lemasolo}, there are unique $i, j \in \{ 1, -1 \}$ such that the Jacquet module of $\s$ with respect to the appropriate parabolic subgroup contains an irreducible constituent of the form $\pi_1 \otimes \d([ \nu \rho', \nu^{\frac{b_{\max} - 1}{2}} \rho']) \otimes \tau_i$ and such that the Jacquet module of $\s'$ with respect to an appropriate parabolic subgroup contains an irreducible constituent of the form $\pi_2 \otimes \d([ \nu \rho', \nu^{\frac{b_{\max} - 1}{2}} \rho']) \otimes \tau_j$. Following the same lines as in the proof of Lemma \ref{lemasolo} we deduce that $i = j$ and, consequently, $\epsilon_{\s}(b_{\max}, \rho') = \epsilon_{\s'}(b_{\max}, \rho')$. From Theorem \ref{prepairemain} and Definition \ref{deftwo}, we obtain that $\epsilon_{\s'}(b, \rho') = \epsilon_{\s}(b, \rho')$.
\item Suppose that $b$ is even. We denote by $b_{\min}$ the minimal element of $\Jord_{\rho'}(\s)$. If $(b_{\min}, \rho')$ does not appear in $\Jord(\s')$, the claim of the theorem follows from the previous lemma, Theorem \ref{prepairemain} and Definition \ref{deftwo}. Suppose that $(b_{\min}, \rho') \in \Jord(\s')$. If $\epsilon_{\s'}(b_{\min}, \rho') = 1$, by Lemma \ref{lemaembed} there is an irreducible representation $\pi$ such that $\s'$ is a subrepresentation of $\d([\nu^{\frac{1}{2}} \rho', \nu^{\frac{b_{\min} -1}{2}} \rho']) \rt \pi$. Since $\s$ is a subrepresentation of $\d([\nu^{-\frac{a\_ -1}{2}} \rho, \nu^{\frac{a -1}{2}} \rho]) \rt \s'$, the fact that $a > b_{\min}$ in the case $\rho \cong \rho'$ implies that there is an embedding
    \begin{equation*}
    \s \hookrightarrow \d([\nu^{\frac{1}{2}} \rho', \nu^{\frac{b_{\min} -1}{2}} \rho']) \times \d([\nu^{-\frac{a\_ -1}{2}} \rho, \nu^{\frac{a -1}{2}} \rho]) \rt \pi,
    \end{equation*}
    which implies that $\epsilon_{\s}(b_{\min}, \rho') = 1$.

    On the other hand, if $\epsilon_{\s}(b_{\min}, \rho') = 1$, there is an irreducible representation $\pi$ such that $\mu^{\ast}(\s) \geq \d([\nu^{\frac{1}{2}} \rho', \nu^{\frac{b_{\min} -1}{2}} \rho']) \ot \pi$. Using the embedding $\s \hookrightarrow \d([\nu^{-\frac{a\_ -1}{2}} \rho, \nu^{\frac{a -1}{2}} \rho]) \rt \s'$ and the fact that $a\_ > b_{\min}$ in the case $\rho \cong \rho'$, the structural formula implies that $\mu^{\ast}(\s')$ contains an irreducible constituent of the form $\d([\nu^{\frac{1}{2}} \rho', \nu^{\frac{b_{\min} -1}{2}} \rho']) \ot \pi'$. Thus, $\epsilon_{\s'}(b_{\min}, \rho') = 1$. It follows that $\epsilon_{\s}(b_{\min}, \rho') = \epsilon_{\s'}(b_{\min}, \rho')$. From Theorem \ref{prepairemain} and Definition \ref{deftwo}, we obtain that $\epsilon_{\s'}(b, \rho') = \epsilon_{\s}(b, \rho')$ and theorem is proved. {\qedhere}
\end{enumerate}
\end{proof}

\section{Classification of discrete series}

We start this section with the definition of the Jordan triples. 

These are the triples of the form $(\Jord, \sigma_c, \epsilon)$ where
\begin{enumerate}[(1)]
  \item $\sigma_c$ is an irreducible cuspidal representation of some $G_{n}$.
  \item $\Jord$ is the finite set (possibly empty) of pairs $(a, \rho)$, where $\rho \in R(GL)$ is an irreducible essentially self-dual cuspidal unitarizable representation, and $a$ is a positive integer such that $a$ is odd if and only if $L(s, \rho, r)$ does not have a pole at $s=0$, for the local $L$-function $L(s, \rho, r)$ as in the beginning of Section $3$. For such a representation $\rho$, let $\Jord_{\rho}$ stand for the set of all $a$ such that $(a, \rho) \in \Jord$. For $a \in \Jord_{\rho}$, let $a\_ = \max \{ b \in \Jord_{\rho} : b < a \}$, if this set is non-empty.
  \item $\epsilon$ is a function defined on a subset of $\Jord \cup (\Jord \times \Jord)$ and attains values 1 and -1. Furthermore, $\epsilon$ is defined on a pair $((a, \rho), (a', \rho')) \in \Jord \times \Jord$ if and only if $\rho \cong \rho'$ and $a = a'\_$. Also, $\epsilon$ is defined on an ordered pair $(a, \rho)$ if and only if $a$ is even or $\Jord_{\rho}(\s') = \emptyset$. Note that if for some $(a, \rho) \in \Jord$ both $a\_$ and $\epsilon(a, \rho)$ are defined, then $\epsilon(a\_, \rho)$ is also defined.
  \item For $(a, \rho) \in \Jord$ such that both $a\_$ and $\epsilon(a, \rho)$ are defined, we have the following compatibility condition: $$\epsilon(a\_, \rho)  \cdot \epsilon(a, \rho) = \epsilon((a\_, \rho), (a, \rho)).$$ 
\end{enumerate}

We say that the Jordan triple $(\Jord', \sigma_c', \epsilon')$ is subordinated to the Jordan triple $(\Jord, \sigma_c, \epsilon)$ if there exists $(a, \rho) \in \Jord$ such that $a\_ \in \Jord_{\rho}$ is defined, $\epsilon((a\_, \rho), (a, \rho)) = 1$, $\sigma_c' \cong \s_c$, $\Jord' = \Jord \setminus \{ (a\_, \rho), (a, \rho) \} $, and we have:
\begin{enumerate}[(1)]
\item  $\epsilon((b\_, \rho'), (b, \rho')) = \epsilon'((b\_, \rho'), (b, \rho'))$, for all $\rho' \not\cong \rho$ and all $b \in \Jord_{\rho'}$ such that $b\_$ is defined.
\item $\epsilon((b\_, \rho), (b, \rho)) = \epsilon'((b\_, \rho), (b, \rho))$, for all $b \in \Jord_{\rho}$ such that $b\_$  is defined and either $b < a\_$ or $b\_ > a$.
\item If in $\Jord_{\rho}$ we have $a = b\_$ and $(a\_)\_ = c$, then
\begin{equation*}
\epsilon'((c, \rho), (b, \rho)) = \epsilon((c, \rho), (a\_, \rho)) \cdot \epsilon((a, \rho), (b, \rho)).
\end{equation*}
\item If $\epsilon'(b, \rho')$ is defined for some $(b, \rho') \in \Jord$, then $\epsilon'(b, \rho') = \epsilon(b, \rho')$.
\end{enumerate}

\begin{definition} \label{deftripl}
We say that the Jordan triple $(\Jord, \s_c, \epsilon)$ is a triple of alternated type if for every $\rho$ such that $\Jord_{\rho} \neq \emptyset$ we have 
\begin{itemize}
    \item $\epsilon((a\_, \rho),(a, \rho)) = -1$ whenever $a\_$ is defined,
    \item there is an increasing bijection $\phi_{\rho}: \Jord_{\rho} \rightarrow \Jord'_{\rho}(\s_c)$, where
\end{itemize}
\begin{equation*}
\mbox{Jord}'_{\rho}(\s_c)=\begin{cases}
       \Jord_{\rho}(\s_c) \cup \{ 0 \} & \mbox{if $\min (\Jord_{\rho})$ is even and $\epsilon(\min (\Jord_{\rho}), \rho) = 1$;}\\
      \Jord_{\rho}(\s_c) & \mbox{otherwise}.
    \end{cases}  
\end{equation*}
\begin{sloppypar}
\noindent We say that the Jordan triple $(\Jord, \s_c, \epsilon)$  dominates the Jordan triple $(\Jord', \s_c, \epsilon')$ if there is a sequence of Jordan triples $(\Jord_{i}, \s_c, \epsilon_{i})$, $1 \leq i \leq k$, such that $(\Jord_{k}, \s_c,$ $\epsilon_{k}) = (\Jord, \s_c, \epsilon), (\Jord_{1}, \s_c, \epsilon_{1}) = (\Jord', \s_c, \epsilon')$, and $(\Jord_{i-1}, \s_c,$ $\epsilon_{i-1})$ is subordinated to $(\Jord_{i}, \s_c, \epsilon_{i})$ for $i \in \{ 2, 3, \ldots, k \}$. Jordan triple $(\Jord, \s_c, $ $\epsilon)$ is called an admissible triple if it dominates a triple of alternated type.
\end{sloppypar}
\end{definition}

For the precise connection between the triples of alternated type and the strongly positive discrete series, we refer the reader to beginning of the proof of Theorem \ref{tmsurj}.

The proof of the following lemma is straightforward:

\begin{lemma}  \label{lemapros}
Let $(\Jord, \s_c, \epsilon)$ denote an admissible triple. Let $\rho \in R(GL)$ denote an irreducible essentially self-dual cuspidal unitarizable representation, and let $a, b$ denote positive integers which are odd if and only if $L(s, \rho, r)$ does not have a pole at $s=0$. Suppose that $a < b$ and that there is no $x \in \Jord_{\rho}$ such that $a \leq x \leq b$. Then there are exactly two admissible triples of the form $(\Jord \cup \{ (a, \rho), (b, \rho) \}, \s_c, \epsilon')$ such that $\epsilon'((a, \rho), (b, \rho)) = 1$, which dominate the admissible triple $(\Jord, \s_c, \epsilon)$.
\end{lemma}

It follows from the results obtained in the previous section that to a discrete series $\s$ one can attach an admissible triple $(\Jord(\s), \s_{cusp}, \epsilon_{\s})$, where $\Jord(\s)$ is the set of the Jordan blocks of $\s$, $\s_{cusp}$ is the partial cuspidal support of $\s$, and $\epsilon_{\s}$ is given by Definitions \ref{defone} and \ref{deftwo}. We show that in this way we obtain a bijection between the set of all isomorphism classes of discrete series in $R(G)$ and the set of all admissible triples.

Let us first show that in a described way we obtain an injection.

\begin{thrm}  \label{tminj}
Suppose that $\s$ and $\s'$ are discrete series such that 
\begin{equation*}
(\Jord(\s), \s_{cusp}, \epsilon_{\s}) = (\Jord(\s'), \s_{cusp}', \epsilon_{\s'}). 
\end{equation*}
Then $\s \cong \s'$.
\end{thrm}
\begin{proof}
If the Jordan triple $(\Jord(\s), \s_{cusp}, \epsilon_{\s})$ is an admissible triple of alternated type, then the claim follows directly from known results for the strongly positive discrete series (\cite[Theorem~5.14]{Kim1} and \cite[Lemma~3.5]{Matic4}).

\begin{sloppypar}
Now suppose that the Jordan triple $(\Jord(\s), \s_{cusp}, \epsilon_{\s})$ is not an admissible triple of alternated type, i.e., that $\s$ is not a strongly positive discrete series. Let $(\s_1, \s_2, \ldots, \s_n)$, $\s_i \in \Irr(G_{n_i})$, denote an ordered $n$-tuple of discrete series such that $\s_1$ is strongly positive, $\s_n \cong \s$ and, for every $i = 2, 3, \ldots, n$, there are $(a_i, \rho_i), (b_i, \rho_i) \in \Jord(\s)$ such that in $\Jord_{\rho_i}(\s_i)$ hold $a_i = (b_i)\_ $, $\epsilon_{\s_i}((a_i, \rho_i), (b_i, \rho_i)) = 1$ and
\end{sloppypar}
\begin{equation*}
\s_i \hookrightarrow \d([\nu^{-\frac{a_i -1}{2}} \rho_i, \nu^{\frac{b_i -1}{2}} \rho_i]) \rt \s_{i-1}.
\end{equation*}

The rest of the proof goes by induction over $n$, and we have seen that our claim holds for $n = 1$.

Let us assume that $n \geq 2$ and that the claim holds for all $k < n$. We prove it for $n$.

Since $(\Jord(\s), \s_{cusp}, \epsilon_{\s}) = (\Jord(\s'), \s_{cusp}', \epsilon_{\s'})$, it follows that there is a discrete series $\s'_{n-1}$ such that $\s'$ is a subrepresentation of $\d([\nu^{-\frac{a_n -1}{2}} \rho_n$, $\nu^{\frac{b_n -1}{2}} \rho_n]) \rt \s'_{n-1}$. Using Proposition \ref{propjord}, Theorem \ref{prepairemain} and Theorem \ref{propprva}, we obtain that
\begin{equation*}
(\Jord(\s_{n-1}), \s_{cusp}, \epsilon_{\s_{n-1}}) = (\Jord(\s'_{n-1}), \s_{cusp}', \epsilon_{\s'_{n-1}}).
\end{equation*}

The inductive assumption now implies that $\s_{n-1} \cong \s'_{n-1}$, so $\s$ and $\s'$ are irreducible subrepresentations of $\d([\nu^{-\frac{a_n -1}{2}} \rho_n, \nu^{\frac{b_n -1}{2}} \rho_n]) \rt \s_{n-1}$. Theorem \ref{dissub} implies that such an induced representation contains two irreducible subrepresentations, which are mutually non-isomorphic discrete series. We denote them by $\s^{(1)}$ and $\s^{(2)}$. Let us now prove that the admissible triples attached to $\s^{(1)}$ and $\s^{(2)}$ are different.

If $\Jord_{\rho_n}(\s_{n-1}) = \emptyset$, it can be concluded from Proposition \ref{propjord} and Theorem \ref{embed} that  $\Jord_{\rho_n}(\s_{cusp}) = \emptyset$. By the Basic Assumption, either $\rho_n \rtimes \sigma_{cusp}$ reduces or $\nu^{\frac{1}{2}} \rho_n \rtimes \sigma_{cusp}$ reduces. Using Corollary \ref{cornonhalf} and Lemma \ref{lemahalf}, we deduce that there are $i_1, i_2 \in \{ 1, 2 \}$ such that $i_1 \neq i_2$ and $\epsilon_{\s^{(i_1)}}(a_n, \rho_n) = \epsilon_{\s^{(i_1)}}(b_n, \rho_n) = 1$ and $\epsilon_{\s^{(i_2)}}(a_n, \rho_n) = \epsilon_{\s^{(i_2)}}(b_n, \rho_n) = -1$.

If $\Jord_{\rho_n}(\s_{n-1}) \neq \emptyset$ and $(a_n)\_ \in \Jord_{\rho_n}(\s)$ is defined, Claims $1$ and $2$ in the proof of Theorem \ref{prepairemain} imply that there is a unique $i_1 \in \{ 1, 2 \}$ such that $\s^{(i_1)}$ is a subrepresentation of $\d([\nu^{\frac{a_n + 1}{2}} \rho_n, \nu^{\frac{b_n -1}{2}} \rho_n]) \rt \tau$ for a tempered representation $\tau$ such that
\begin{gather*}
\mu^{\ast}(\tau) \geq \d([\nu^{\frac{(a_n)\_ + 1}{2}} \rho_n, \nu^{\frac{a_n -1}{2}} \rho_n]) \times \d([\nu^{\frac{(a_n)\_ + 1}{2}} \rho_n, \nu^{\frac{a_n -1}{2}} \rho_n]) \otimes \\
\d([\nu^{-\frac{(a_n)\_ + 1}{2}} \rho_n, \nu^{\frac{(a_n)\_ -1}{2}} \rho_n]) \rt \s_{n-1}.
\end{gather*}

Using Claims $1$, $2$ and $4$ in the proof of Theorem \ref{prepairemain} one can deduce that
\begin{equation*}
\epsilon_{\s^{(i_1)}}(((a_n)\_,\rho_{n}),(a_n,\rho_{n})) = 1
\end{equation*}
and
\begin{equation*}
\epsilon_{\s^{(i_2)}}(((a_n)\_,\rho_{n}),(a_n,\rho_{n})) = -1
\end{equation*}
for $i_2 \in \{ 1, 2 \}$, $i_1 \neq i_2$.

If $\Jord_{\rho_n}(\s_{n-1}) \neq \emptyset$ and $a_n$ is the minimal element of  $\Jord_{\rho_n}(\s)$, we denote by $c_n$ an element of $\Jord_{\rho_n}(\s)$ such that $(c_n)\_ = b_n$. Using Claims $1$, $3$ and $5$ in the proof of Theorem \ref{prepairemain} one can see that there is a unique $i_1 \in \{ 1, 2 \}$ such that $\s^{(i_1)}$ is a subrepresentation of $\d([\nu^{\frac{a_n + 1}{2}} \rho_n, \nu^{\frac{b_n -1}{2}} \rho_n]) \rt \tau$ for a tempered representation $\tau$ such that $\mu^{\ast}(\tau)$ contains an irreducible constituent of the form $\d([\nu^{\frac{a_n + 1}{2}} \rho_n, \nu^{\frac{c_n -1}{2}} \rho_n]) \ot \pi_1$, and it follows that
\begin{equation*}
\epsilon_{\s^{(i_1)}}((b_n,\rho_{n}),(c_n,\rho_{n})) = 1
\end{equation*}
and
\begin{equation*}
\epsilon_{\s^{(i_2)}}((b_n,\rho_{n}),(c_n,\rho_{n})) = -1
\end{equation*}
for $i_2 \in \{ 1, 2 \}$, $i_1 \neq i_2$.

Consequently, the admissible triples attached to $\s^{(1)}$ and $\s^{(2)}$ are different, and it follows that $\s \cong \s'$.
\end{proof}

It remains to show  that to every admissible triple corresponds a discrete series.

\begin{thrm} \label{tmsurj}
Let $(\Jord, \s_c, \epsilon)$ denote an admissible triple. Then there is a discrete series $\s \in R(G)$ such that $(\Jord(\s), \s_{cusp}, \epsilon_{\s}) = (\Jord, \s_c, \epsilon)$.
\end{thrm}
\begin{proof}
Let us first assume that $(\Jord, \s_c, \epsilon)$ is an admissible triple of alternated type. For each $\rho$ such that $\Jord_{\rho} \neq \emptyset$, we write the elements of $\Jord_{\rho}$ in increasing order $a^{\rho}_{1} < a^{\rho}_{2} < \cdots < a^{\rho}_{k_{\rho}}$. Using \cite[Theorem~5.15]{Kim1} we deduce that the induced representation 
\begin{equation*}
    \big ( \prod_{\rho} \prod_{i=1}^{k_{\rho}} \delta([\nu^{\frac{\phi_{\rho}(a^{\rho}_i)+1}{2}} \rho, \nu^{\frac{a^{\rho}_i - 1}{2}} \rho]) \big ) \rtimes \s_c
\end{equation*}
has a unique irreducible subrepresentation, which is a strongly positive discrete series and we denote it by $\s$. From \cite[Theorem~5.3]{Matic4} or \cite[Section~7]{MatTad} follows at once that $(\Jord(\s), \s_{cusp}, \epsilon_{\s}) = (\Jord, \s_c, \epsilon)$.

Now suppose that $(\Jord, \s_c, \epsilon)$ is not an admissible triple of alternated type. Let $(\Jord_{i}, \s_c, \epsilon_{i})$, $1 \leq i \leq n$, denote a sequence of Jordan triples such that $(\Jord_{n}, \s_c, \epsilon_{n}) = (\Jord, \s_c, \epsilon)$, $(\Jord_{1}, \s_c, \epsilon_{1})$ is an admissible triple of alternated type, and $(\Jord_{i-1}, \s_c$, $\epsilon_{i-1})$ is subordinated to $(\Jord_{i}, \s_c, \epsilon_{i})$ for $i \in \{ 2, 3, \ldots, n \}$.

The rest of the proof goes by induction over $n$, and we have seen that our claim holds for $n = 1$.

Let us assume that $n \geq 2$ and that the claim holds for all $k < n$. We prove it for $n$.

\begin{sloppypar}
Suppose that $\Jord = \Jord_{n-1} \cup \{ (a\_, \rho), (a, \rho) \}$. By the inductive assumption, there is a discrete series $\s' \in R(G)$ such that $(\Jord(\s'), \s_{cusp}, \epsilon_{\s'}) = (\Jord_{n-1}, \s_c, \epsilon_{n-1})$.
\end{sloppypar}
By Theorem \ref{dissub}, there are two mutually non-isomorphic discrete series subrepresentations of $\d([\nu^{-\frac{a\_ - 1}{2}} \rho, \nu^{\frac{a-1}{2}} \rho]) \rt \s'$, which we denote by $\s_1$ and $\s_2$. Note that both the  admissible triples $(\Jord(\s_{1}), \s_{cusp}, \epsilon_{\s_{1}})$ and $(\Jord(\s_{2}), \s_{cusp}$, $\epsilon_{\s_{2}})$ dominate the admissible triple $(\Jord_{n-1}, \s_c, \epsilon_{n-1})$ and, by the previous theorem, $(\Jord(\s_{1}), \s_{cusp}, \epsilon_{\s_{1}}) \neq (\Jord(\s_{2}), \s_{cusp}, \epsilon_{\s_{2}})$. By Lemma \ref{lemapros}, there is an $i \in \{ 1, 2 \}$ such that $(\Jord(\s_{i}), \s_{cusp}, \epsilon_{\s_{i}}) = (\Jord, \s_c, \epsilon)$ and the theorem is proved.
\end{proof}

\bibliographystyle{siam}
\bibliography{Literatura}

\begin{flushleft}
{Yeansu Kim \\
Department of Mathematics education, Chonnam National University \\
77 Yongbong-ro, Buk-gu, Gwangju city, South Korea\\
E-mail: ykim@jnu.ac.kr}
\end{flushleft}

\begin{flushleft}
{Ivan Mati\'{c} \\
Department of Mathematics, University of Osijek \\ Trg Ljudevita
Gaja 6, Osijek, Croatia\\ E-mail: imatic@mathos.hr}
\end{flushleft}

\end{document}